\documentclass[11pt]{article}

\usepackage[a4paper,hmargin=3cm,vmargin=3cm]{geometry}
\usepackage{amsmath,amsthm,amsfonts,amssymb,amscd,amstext,hyperref}
\usepackage{graphicx}
\usepackage{color}
\usepackage[dvips]{epsfig}
\usepackage[latin1]{inputenc}
\usepackage{setspace}
\usepackage{titletoc}

\usepackage{fancyhdr}
\usepackage{times}

\usepackage{enumitem}
\usepackage{titlesec}
\usepackage{mathrsfs}
\usepackage{stmaryrd}

\setlist[enumerate]{itemsep=0pt}

\def\R{\mathbb{R}}
\def\B{\mathbb{B}}
\def\N{\mathbb{N}}
\def\Z{\mathbb{Z}}
\def\C{\mathbb{C}}

\newcommand{\ben}{\begin{enumerate}}
\newcommand{\bit}{\begin{itemize}}
\newcommand{\een}{\end{enumerate}}
\newcommand{\eit}{\end{itemize}}

\newcommand{\ed}{\end{document}}

\def\cA{\mathcal{A}}
\def\cU{\mathcal{U}}
\def\cP{\mathcal{P}}

\def\cD{\mathcal{D}}

\def\cB{\mathcal{B}}
\def\cW{\mathcal{W}}

\def\cV{\mathcal{V}}
\def\cJ{\mathcal{J}}

\def\cN{\mathcal{N}}

\def\cQ{\mathcal{Q}}

\let\8=\infty \let\0=\emptyset  
\let\hat=\widehat

\let\tilde=\widetilde

\let\landa=\lambda
\let\alfa=\alpha

\let\parc=\partial

\def\ep{\varepsilon}
\def\vart{\vartheta}
\def\landa{\lambda}

\def\flecha{\rightarrow}
\def\esiz{\langle}
\def\esde{\rangle}

\def\cte.{\mathop{\rm cte.}\nolimits}

\def\cosh{\mathop{\rm cosh }\nolimits}
\def\tanh{\mathop{\rm tanh }\nolimits}

\def\N{\mathbb{N}}

\def\B{\mathbb{B}}

\def\R{\mathbb{R}}
\def\Z{\mathbb{Z}}
\def\C{\mathbb{C}}

\def\H{\mathbb{H}}
\def\S{\mathbb{S}}

\titleformat{\section}
{\filcenter\bfseries\large} {\thesection{.}}{0.2cm}{}
\titleformat{\subsection}[runin]
{\bfseries} {\thesubsection{.}}{0.15cm}{}[.]
\titleformat{\subsubsection}[runin]
{\em}{\thesubsubsection{.}}{0.15cm}{}[.]

\usepackage[up,bf]{caption}

\newtheorem{theorem}{Theorem}[section]
\newtheorem{lemma}[theorem]{Lemma}
\newtheorem{proposition}[theorem]{Proposition}

\newtheorem{remark}[theorem]{Remark}
\newtheorem{corollary}[theorem]{Corollary}
\newtheorem{definition}[theorem]{Definition}
\newtheorem{conjecture}[theorem]{Conjecture}

\numberwithin{equation}{section}
\numberwithin{figure}{section}

\usepackage{color}

\begin{document}
\fancyhead[LO]{On the classification of Serrin planar domains}
\fancyhead[RE]{Alberto Cerezo, Isabel Fernández, Pablo Mira}
\fancyhead[RO,LE]{\thepage}

\thispagestyle{empty}

\begin{center}
{\bf \LARGE On the classification of Serrin planar domains}

\vspace*{5mm}

\hspace{0.2cm} {\Large Alberto Cerezo, Isabel Fernández and Pablo Mira}
\end{center}

%



\footnote[0]{
\mbox{} \hspace{0.25cm} \emph{Keywords}: Serrin problem, torsion equation, overdetermined elliptic problem, mKdV potentials, global bifurcation, capillary curves, elliptic functions}

\vspace*{-7mm}

\begin{quote}
{\small
\noindent {\bf Abstract}\hspace*{0.1cm}
We show that all smooth ring domains $\Omega\subset \R^2$ that admit a solution to Serrin's classical problem $\Delta u+2=0$ with locally constant overdetermined boundary conditions along $\parc \Omega$ can be described as algebro-geometric potentials of the mKdV hierarchy. The same result holds for periodic unbounded domains with two boundary components. In particular, any such domain is determined by suitable holomorphic data in some algebraic curve. As a consequence, the space of all Serrin ring domains, or periodic Serrin bands, can be ordered into a sequence of finite-dimensional complexity levels. By studying the first non-trivial level, given by elliptic functions, we construct: $(i)$ a global $1$-parameter family of periodic solutions to Serrin's problem that interpolates between a flat band and a chain of disks along an axis, following an unduloid pattern, and $(ii)$ for any $n>1$, a two-dimensional moduli space ${\bf T}_n$ of non-radial Serrin ring domains with a dihedral symmetry group of order $2n$. This moduli space ${\bf T}_n$ is geometrically a triangle, and has radial bands on one side of ${\bf T}_n$, and a necklace of $n$ pairwise tangent disks distributed along the unit circle at its opposite vertex in ${\bf T}_n$. 

\vspace*{0.1cm}

}
\end{quote}

\begin{spacing}{0.1}
\tableofcontents
\end{spacing}

\section{Introduction and presentation of the main results}

\subsection{Statement of the problem}\label{sec:introintro}
A fundamental result in the theory of overdetermined elliptic problems is Serrin's theorem \cite{Se}: if the problem 
\begin{equation}\label{serrinprob}\def\arraystretch{1.8}\left\{\begin{array}{lll} \Delta u +2=0 & \text{ in } & \Omega, \\
u=a, \hspace{0.5cm} \displaystyle \frac{\parc u}{\parc \nu} = b  & \text{ on } & \parc \Omega, 
\end{array} \right.
\end{equation} 
can be solved on a smooth bounded domain $\Omega\subset \R^n$, then $\Omega$ is a ball and $u$ is radial. Here, $a,b\in \R$ are constants, and $\nu$ is the exterior unit normal of $\parc \Omega$. This result is the overdetermined version of Alexandrov's soap bubble theorem \cite{A1}, according to which compact embedded constant mean curvature hypersurfaces in $\R^{n+1}$ are spheres.

Serrin's theorem has generated, among others, two fruitful lines of research. One of them is its extension to more general equations and boundary conditions, primarily based on applications of the maximum principle via moving planes or adequate $P$-functions, as in Weinberger's approach \cite{Wn} to \eqref{serrinprob}. We will not discuss this extense line here. The other one is the construction of \emph{exceptional domains}, i.e., non-symmetric domains $\Omega\subset \R^n$ supporting solutions to elliptic equations $\Delta u +f(u)=0$ with constant overdetermined boundary conditions along $\parc \Omega$. This approach is typically done via bifurcation or desingularization techniques starting from a family of simple examples. Without any attempt at completeness, we mention \cite{ABM,CF,DZ,DPW,EFRS,FMW1,FMW2,HHP,KS,LWW,RRS2,Ru,RSW,Sic,SS,T,Wh} as some relevant existence results (see also the survey \cite{Sic2}).

All these ideas merge beautifully in the study of two well-known overdetermined problems for the original torsion equation $\Delta u+2=0$ in $\R^2$. Both seek to understand the scope of Serrin's original theorem, and have been studied for some time. These problems, which we describe below, appear naturally in fluid dynamics and elasticity theory, and so they have interesting physical applications \cite{ABM,NT,Se,Sir2}.

\begin{enumerate}
\item[{\bf (I)}] Let $\Omega\subset \R^2$ be a smooth, bounded domain with boundary $\parc \Omega=\cup_{j=1}^n \parc_j\Omega$, $n\geq 2$, where 
\begin{equation}\label{overeq00}\def\arraystretch{1.8}\left\{\begin{array}{lll} \Delta u +2=0 & \text{ in } & \Omega, \\
u=a_j, \hspace{0.5cm} \displaystyle \frac{\parc u}{\parc \nu} = b_j  & \text{ on } & \parc_j \Omega, 
\end{array} \right.
\end{equation}
can be solved for $a_j,b_j\in \R$, $j=1,\dots ,n$. \emph{Is then $\Omega$ a radial annulus?}
\end{enumerate}
\begin{enumerate}
\item[{\bf (II)}] Let $\Omega\subset \R^2$ be a smooth unbounded domain where \eqref{serrinprob} can be solved. \emph{Is then $\Omega$ a flat band, i.e., is $\parc \Omega$ formed by two parallel lines?}
\end{enumerate}

Starting with the work by Payne and Philippin in 1991 \cite{PP}, several authors have provided affirmative answers to these problems under additional hypotheses, see e.g. \cite{ABM,Bo,KSV,Rei,RS,Sir,WGS}. However, the answer to both {\bf (I)} and {\bf (II)} is negative, even in their simplest topological versions (ring domains and bands, respectively). 
Specifically, first Kamburov and Sciaraffia \cite{KS} for $b_1=b_2$, and then Agostiniani, Borghini and Mazzieri \cite{ABM} for $a_1=a_2$, have constructed non-radial ring-type solutions to {\bf (I)}, by bifurcation from adequate radial annuli. Regarding problem {\bf (II)}, Fall, Minlend and Weth \cite{FMW1} constructed non-flat periodic solutions, bifurcating from flat bands. And in \cite{DDMW}, Dávila, Del Pino, Musso and Wheeler constructed new periodic band-type solutions to {\bf (II)}, this time by desingularizing a chain of disks along the $x_1$-axis. 

These results justify the following definition:
\begin{definition}\label{def:serrin}
A smooth bounded doubly connected domain $\Omega\subset \R^2$ where problem {\bf (I)} can be solved will be called a \emph{Serrin ring domain}.

A smooth (unbounded) domain $\Omega\subset \R^2$ with two boundary components where problem {\bf (II)} can be solved will be called a \emph{Serrin band}. If $\Omega$ is invariant by a non-zero translation in $\R^2$, we will say that $\Omega$ is a \emph{periodic Serrin band}.
\end{definition}

\subsection{Summary of our main results}

The examples in \cite{ABM,DDMW,FMW1,KS} show that the space of Serrin ring domains and periodic bands is very large, and so their classification seems challenging. In this paper we present three global structure theorems for such domains. Each of them is, up to our knowledge, the first of its nature in the context of Serrin-type overdetermined problems.

\begin{enumerate}
\item
We construct a real analytic family $\{\Omega_\tau : \tau\in (0,1)\}$ of periodic Serrin bands that converge to a radial band as $\tau\to 1$, and to a chain of tangent disks along an axis as $\tau\to 0$. This family presents a Delaunay pattern, and its existence proves, in particular, that the bifurcation branches of the Serrin bands of Fall, Minlend and Weth \cite{FMW1} persist until reaching a singular configuration with the properties of the domains by Dávila, Del Pino, Musso and Wheeler \cite{DDMW}. See Figure \ref{figbandas}.
\item
We construct, for any $n\geq 2$, a real-analytic $2$-parameter family of Serrin ring domains $\Omega(s,\tau)$ with a dihedral symmetry group $D_n$, and whose associated moduli space is a triangle ${\bf T}_n$. In this triangle, one side corresponds to radial annuli, while its opposite vertex represents a \emph{necklace} of $n$ pairwise tangent disks along the unit circle $\S^1$ (Theorem \ref{th:main}). See Figures \ref{fig1:annuli} and \ref{fig2:annuli}. In this moduli space there exist two small curves $\Upsilon_1$ (resp. $\Upsilon_2$) emanating from two specific points on the \emph{radial} side of ${\bf T}_n$, and along which $b_1=b_2$ (resp. $a_1=a_2$) holds in \eqref{overeq00}. The examples along these two curves correspond to the ring domains obtained in \cite{KS,ABM}.
\item
We prove that any Serrin ring domain, as well any periodic Serrin band, can be described in terms of a family of well-known algebraic objects from integrable systems theory, called \emph{algebro-geometric potentials} of the modified Korteweg-de Vries (mKdV) hierarchy  (Theorem \ref{th:kdvintro}). 
In particular, this shows that the space of all such Serrin domains is ordered into a sequence of complexity levels, labeled by the spectral genus $\mathfrak{m}\in \N$ of its associated mKdV potential. In this algebraic description, radial annuli and flat bands describe the level $\mathfrak{m}=0$, and the moduli spaces of items (1), (2) above belong to level $\mathfrak{m}=1$, determined by elliptic functions.
\end{enumerate}

To prove these results, we develop a new framework for the description of Serrin domains, based on the notion of \emph{capillary curves} associated to them. This approach has a geometric nature and allows, in many cases, an explicit control of the resulting domains.

We state our contributions in detail below, and indicate how they compare with previous results regarding Serrin-type problems and also in connection with some fundamental theorems of constant mean curvature surface theory.

\subsection{The Serrin problem on ring domains}\label{intro:rings}
Our main existence result for Serrin ring domains is Theorem \ref{th:main} below.

\begin{theorem}\label{th:main}
Fix $n\in \N$, $n\geq 2$. Then, there exist continuous functions $$\mathfrak{h}_0(\tau)<\mathfrak{h}_1(\tau):(0,1)\flecha (-\8,0)$$ with $$\lim_{\tau\to 0} \, (\mathfrak{h}_0(\tau),\mathfrak{h}_1(\tau))=(-\pi,-\pi),\hspace{0.5cm} \lim_{\tau\to 1} \,  (\mathfrak{h}_0(\tau),\mathfrak{h}_1(\tau))=(-\8, 0),$$ such that, if we denote 
\begin{equation}\label{def:wo}
\cW_0:=\{(s,\tau): \tau \in (0,1), \mathfrak{h}_0(\tau)<s<\mathfrak{h}_1(\tau)\},
\end{equation}
then:
\begin{enumerate}\itemsep=0pt
\item
For any $(s,\tau)\in \cW_0$ there exists a non-radial Serrin ring domain $\Omega_{(s,\tau)}$ with a dihedral symmetry group of order $2n$ centered at the origin. The dependence of $\Omega_{(s,\tau)}$ on the parameters $(s,\tau)\in \cW_0$ is real analytic.
\item
If $(s,\tau)\neq (s',\tau')$, then $\Omega_{(s,\tau)}$ and $\Omega_{(s',\tau')}$ do not differ by a similarity of $\R^2$.
\item
If $(s,\tau)\to (\mathfrak{h}_0(\tau_0),\tau_0)$ for some $\tau_0\in (0,1)$, then the exterior component of $\parc \Omega_{(s,\tau)}$ converges to a regular, real analytic closed curve that thas $n$ points of tangential self-intersection.
\item
If $(s,\tau)\to (\mathfrak{h}_1(\tau_0),\tau_0)$ for some $\tau_0\in (0,1)$, then the interior component of $\parc \Omega_{(s,\tau)}$ has the property of item (3) above.
\item
If $(s,\tau)\to (\hat{s},1)$ for some $\hat{s}<0$, then $\Omega_{(s,\tau)}$ converges analytically to a radial annular domain $\Omega_{(\hat{s},1)}$.
\item
As $(s,\tau)\to (-\pi,0)$, the Serrin ring domains $\Omega_{(s,\tau)}$ converge to a singular necklace of $n$ pairwise tangent disks of the same radius along the unit circle $\S^1$.
\item
The domains $\Omega_{(s,\tau)}$ are described explicitly in terms of elliptic functions.
\end{enumerate}
\end{theorem}

We should clarify that, when $n=2$, the limit necklace has a different structure. Each of the two ``disks'' is actually a half-plane on one side of the $x_2$-axis, and the singular intersection point where the necks are placed are, in this case, the north and south poles of $\S^1$. This behavior can be intuited in Figure \ref{fig1:annuli}, left.

\begin{figure}[h]
  \centering
  \includegraphics[width=0.333\textwidth]{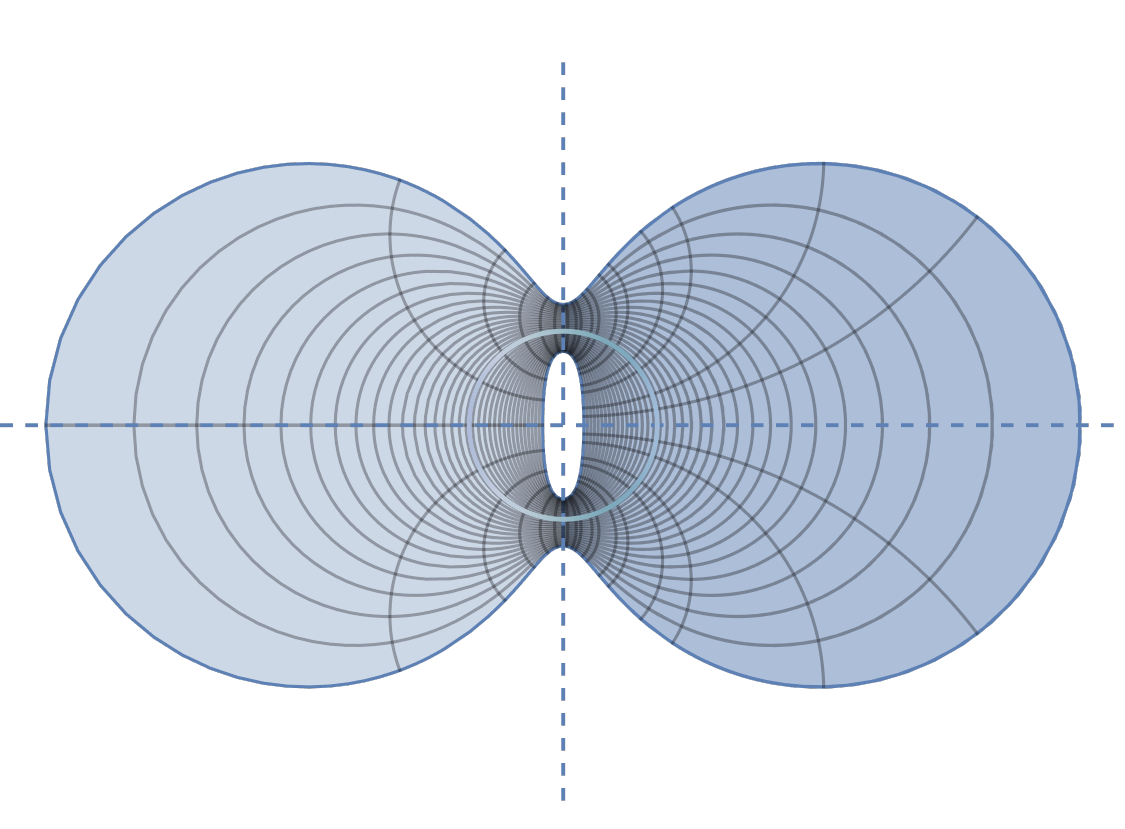}  \hspace{0.5cm} \includegraphics[width=0.29\textwidth]{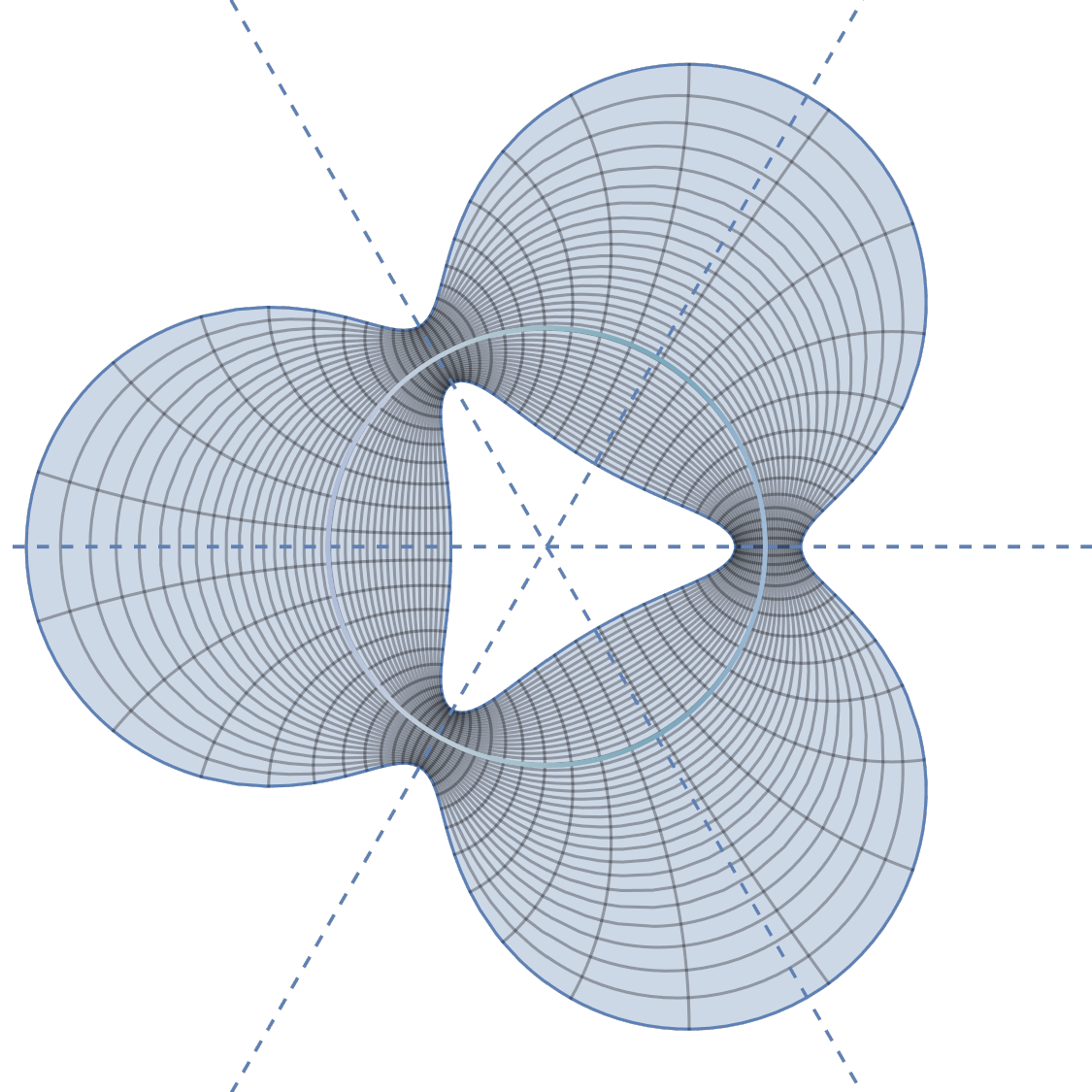}  \hspace{0.5cm}\includegraphics[width=0.29\textwidth]{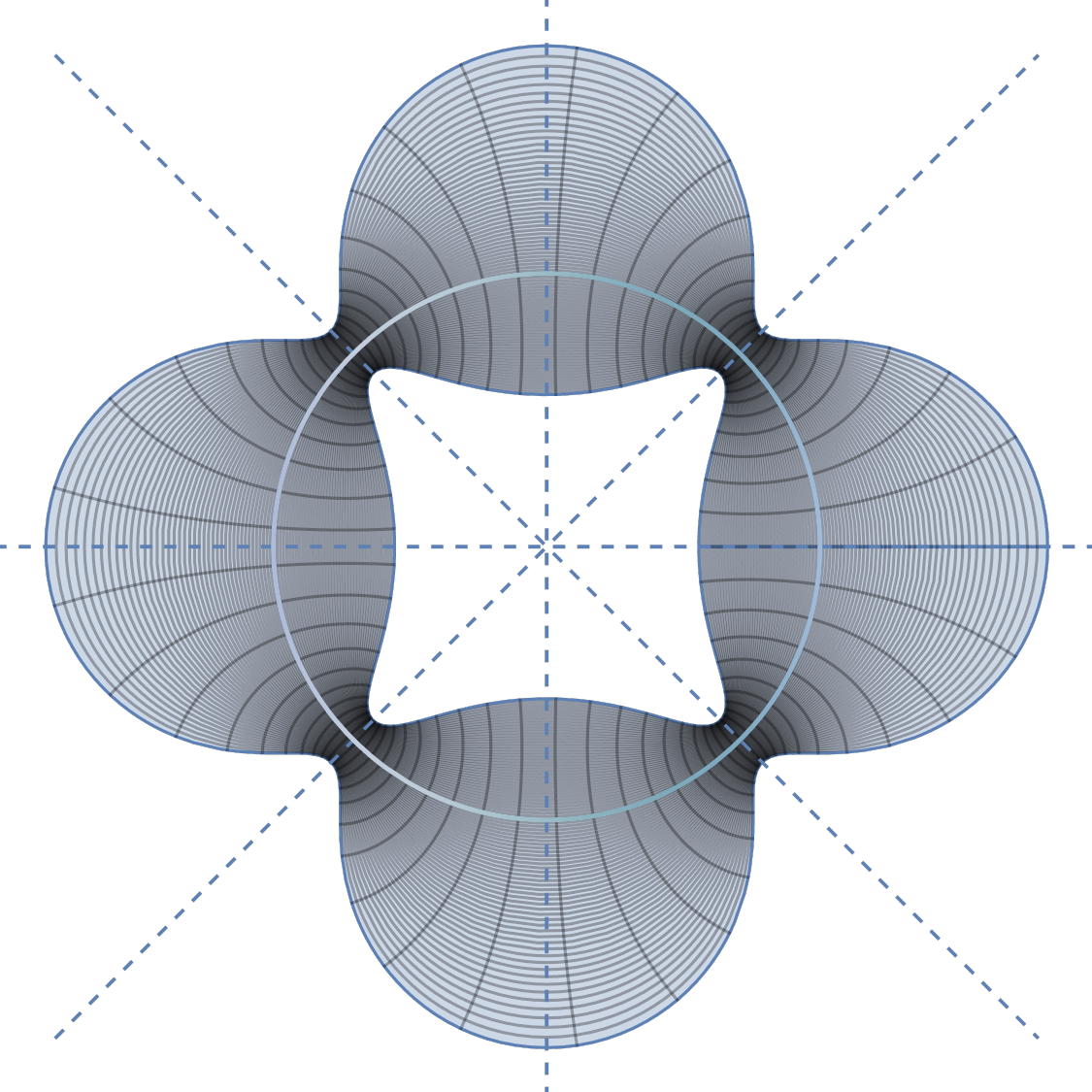}  
\caption{Serrin ring domains for $n=2$, $n=3$ and $n=4$ at an intermediate point of their respective moduli spaces. In all figures parametrizations are conformal, and explicit in terms of elliptic functions. The unit circle is always a parameter curve contained in the domain.\label{fig1:annuli}} 
\end{figure}

\begin{figure}[h]
  \centering
  \includegraphics[width=0.3\textwidth]{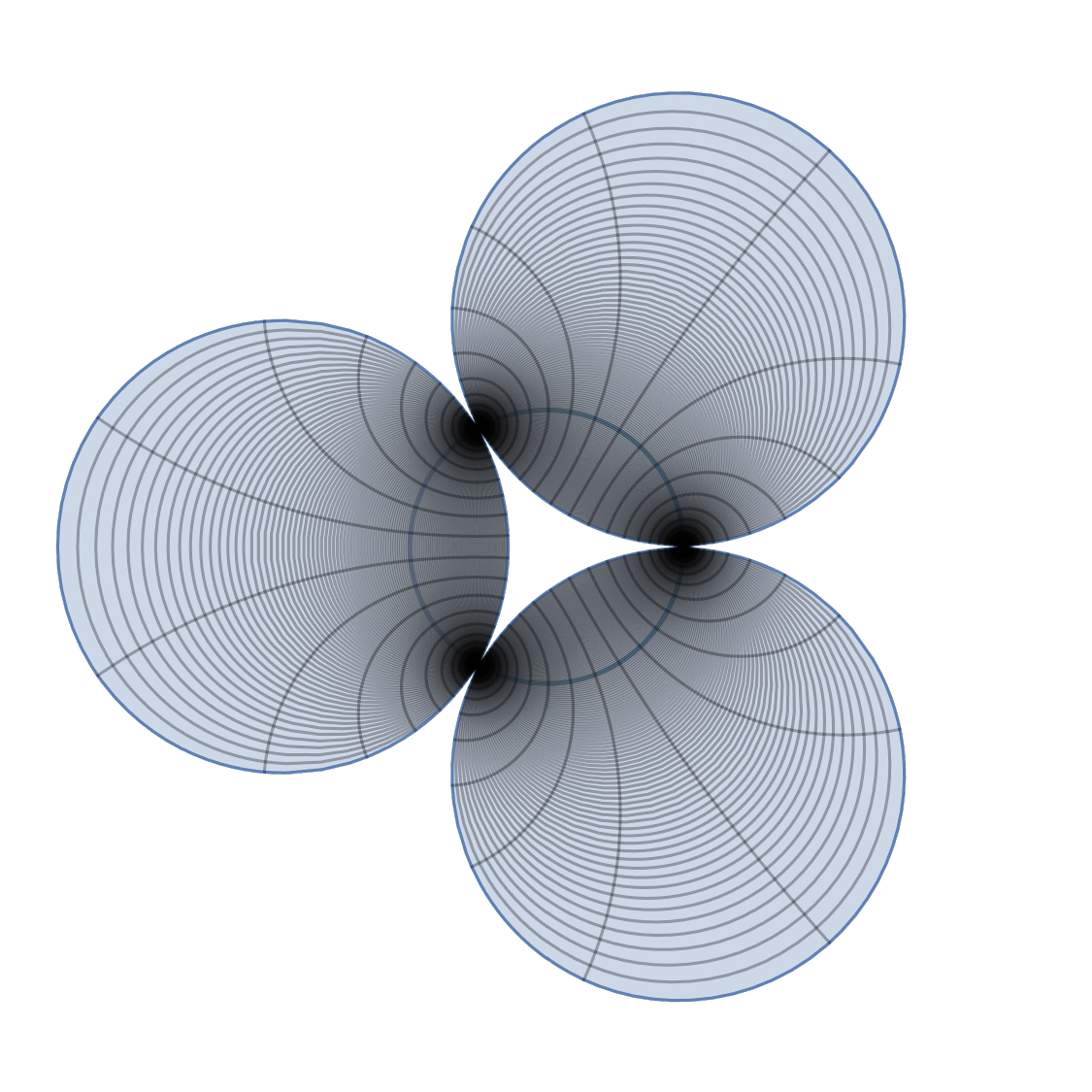}  \hspace{0.32cm} \includegraphics[width=0.3\textwidth]{Per3-12.pdf}  \hspace{0.5cm}\includegraphics[width=0.3\textwidth]{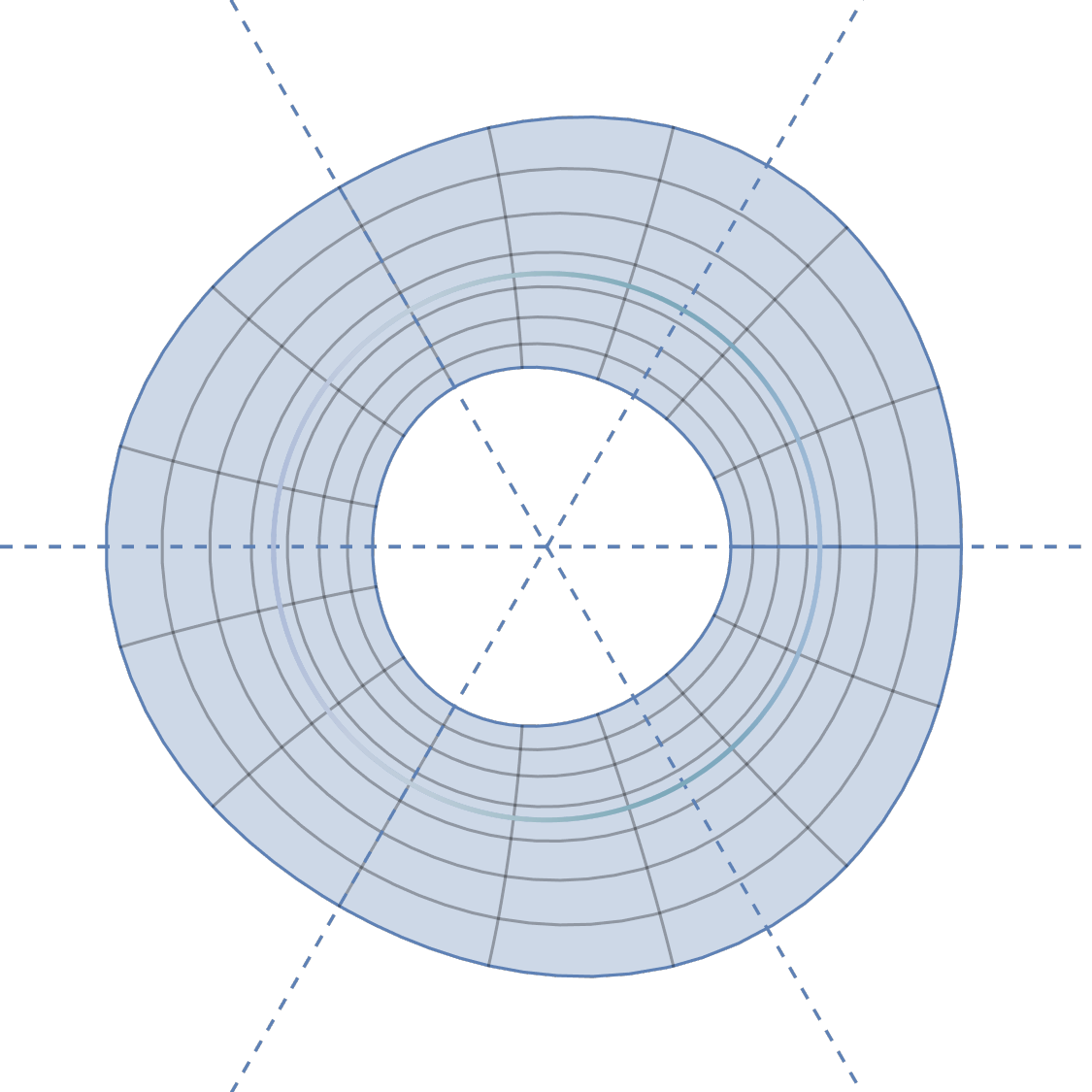}  
  \caption{Three snapshots of Serrin ring domains for $n=3$. Left: $\tau$ close to zero. Middle: $\tau$ at an intermediate value. Right: $\tau$ close to $1$.\label{fig2:annuli}} 
  \end{figure}

The moduli space ${\bf T}_n$ of Serrin ring domains with dihedral symmetry $D_n$, parametrized by $\cW_0$ in \eqref{def:wo}, is geometrically a triangle. The side ${\bf L}_1$ of ${\bf T}_n$ corresponding to $\tau=1$ in $\cW_0$ gives radial annuli. The opposite vertex in ${\bf T}_n$ is given by $(s,\tau)=(-\pi,0)$, and describes a disk $n$-necklace along $\S^1$. The side ${\bf L}_2$ (resp. ${\bf L}_3$) corresponding to $s=\mathfrak{h}_0(\tau)$ (resp. $s=\mathfrak{h}_1(\tau)$) describes domains where the exterior (resp. interior) boundary curve of $\parc \Omega$ self-intersects, and so $\Omega$ stops being a (smooth) Serrin ring domain. In both cases, these self-intersection points are placed on the set of symmetry lines of $\parc \Omega$. See Figure \ref{pierdeembe}.

%
%

By construction, all ring domains $\Omega_{(s,\tau)}$ in the triangular moduli space ${\bf T}_n$ solve \eqref{overeq00} for values $a_j,b_j\in \R$, $j=1,2$ that depend on $(s,\tau)$ and $n$; that is, $a_j,b_j$ vary as we move through ${\bf T}_n$. Now, up to dilations in the domain $\Omega$ and addition of constants to the solution $u:\Omega\flecha \R$, two of these four values can be fixed within each similarity class of Serrin ring domains; e.g., we can take $a_1=0, b_1=1$. Thus, the prescription of an analytic relation $\Phi(a_2,b_2)=0$ with $\Phi\not\equiv (0,0)$ between the remaining values $a_2(s,\tau)$ and $b_2(s,\tau)$ will give a union of real analytic arcs inside ${\bf T}_n$, as long as this set is not empty. We will show at the end of Section \ref{sec:mainth} that there exists a real analytic curve $\Upsilon_1\subset {\bf T}_n$ starting at some specific point on the \emph{radial side} $\tau=1$ of ${\bf T}_n$, and along which $b_1=b_2$ holds. This curve $\Upsilon_1$ can be seen as a $1$-dimensional bifurcation of Serrin ring domains with dihedral symmetry, and the resulting examples correspond, up to similarities of $\R^2$, with those constructed by Kamburov and Sciaraffia \cite{KS} for $n=2$. A similar process can be done for the condition $a_1=a_2$, obtaining a curve $\Upsilon_2$ that represents, up to similarities, the examples in the bifurcation theorem of Agostiniani, Borghini and Mazzieri \cite{ABM}.


\begin{figure}[h]
  \centering
  \includegraphics[width=0.45\textwidth]{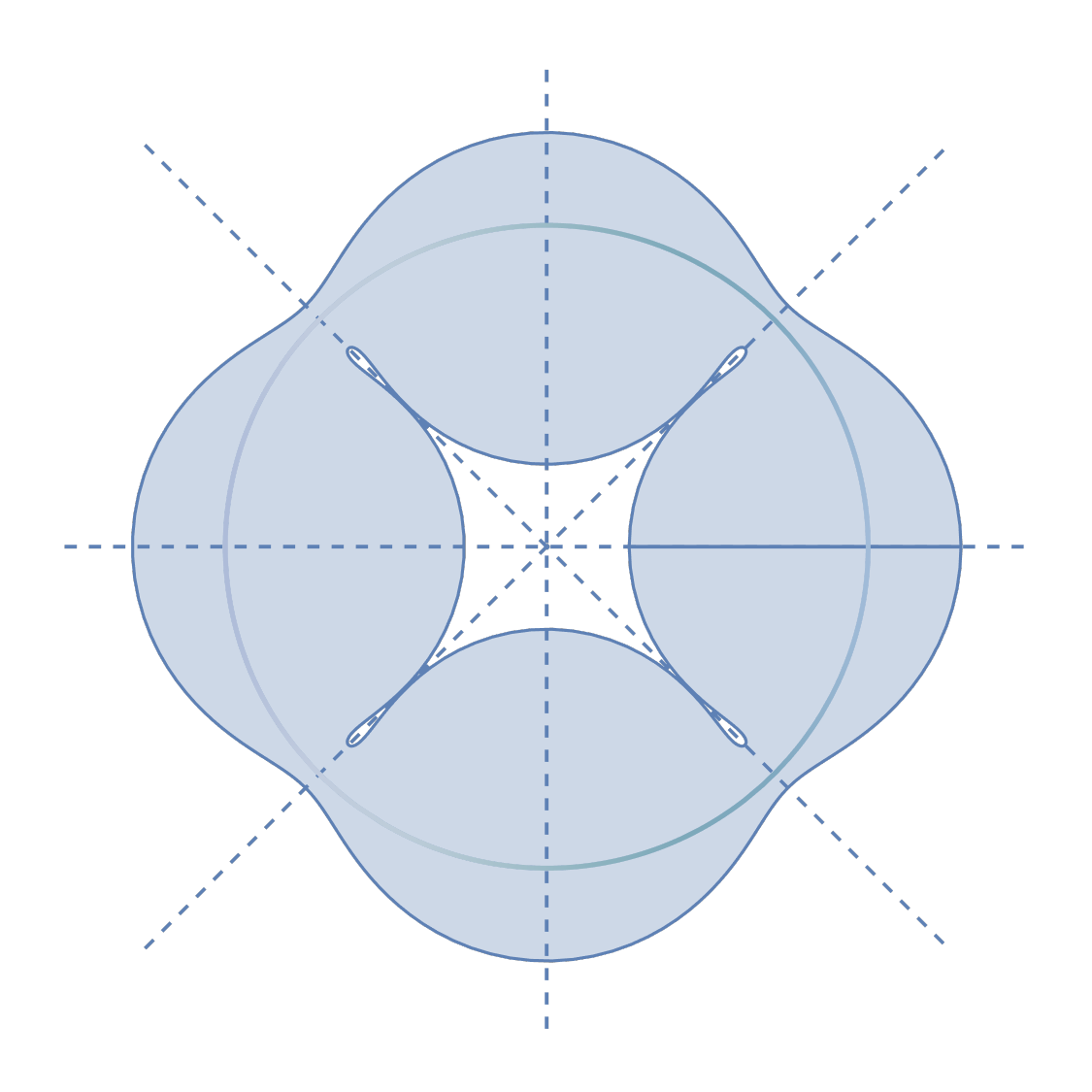}  \hspace{0.4cm} \includegraphics[width=0.45\textwidth]{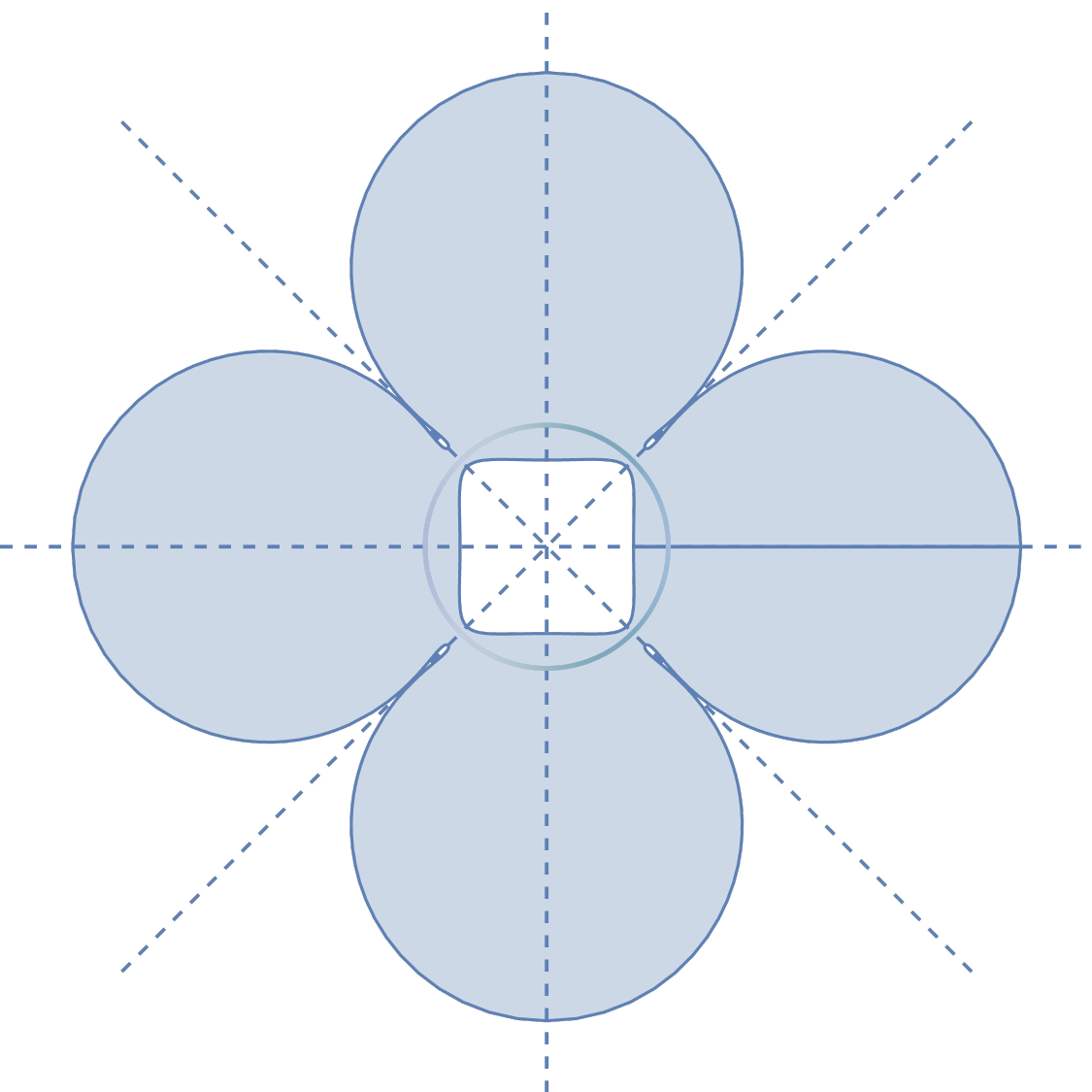}  \caption{Loss of embeddedness for $\parc \Omega$ at the boundary of the moduli space. Left: interior boundary curve. Right: exterior one.\label{pierdeembe}} 
\end{figure}

Theorem \ref{th:main} is the analogue in the context of Serrin's overdetermined problem of an important theorem by Wente \cite{W} for constant mean curvature (CMC) surfaces in $\R^3$. To explain this, we start by recalling an old famous conjecture by Hopf claiming that spheres should be the only closed (i.e. compact without boundary) CMC surfaces in $\R^3$. Alexandrov \cite{A1} gave an affirmative answer in the embedded case, and Hopf himself \cite{Ho} proved it for surfaces of genus zero. But surprisingly, in 1986 Wente \cite{W} found counterexamples to the conjecture, with the topology of a torus. Soon after, Abresch showed in \cite{Ab} that Wente tori have the special feature of being foliated by planar curvature lines, and used this property to describe them in terms of elliptic integrals.

The Serrin ring domains in Theorem \ref{th:main} can then be regarded as \emph{overdetermined Wente tori} for \eqref{overeq00}. Indeed, our examples are constructed by imposing that the annuli $\Omega\subset \R^2$ in Theorem \ref{th:main} are \emph{foliated by capillary curves}. By definition, any such capillary curve $\gamma\subset \Omega$ has the property that the graph of $u$ over $\gamma$ lies in a rotational paraboloid or in a plane $\cP_\gamma\subset \R^3$, and $|\nabla(u-u^\gamma)|$ is constant along $\gamma$, where $u^\gamma$ is the graphing function of $\cP_\gamma$. 

Theorem \ref{th:main} is also related to free boundary CMC surfaces in the unit ball $\B^3$ of $\R^3$. An old problem by Nitsche \cite{Nit} and Wente \cite{W2} asked whether any embedded free boundary CMC surface in $\B^3$ with the topology of an annulus should be rotational, i.e., part of a catenoid or a Delaunay example. A counterexample was recently obtained by the authors in \cite{CFM} for the CMC case. For the minimal case, Fernández, Hauswirth and Mira constructed in \cite{FHM} non-rotational free  boundary minimal annuli in $\B^3$ with self-intersections. In these works, which are one of the main sources of inspiration for the present paper, the new examples were constructed by imposing a foliation structure by \emph{spherical} curvature lines on the surface. See also \cite{Ce,CFM2}.

\subsection{The periodic Serrin problem}\label{intro:bands} Let $\Omega$ be a smooth domain of $\R^2$ where Serrin's problem \eqref{serrinprob} can be solved. If $\Omega$ is bounded, then $\Omega$ is a disk, by Serrin's theorem. When $\Omega$ is not bounded, Ros and Sicbaldi \cite{RS} proved that $\parc \Omega$ must have at least two \emph{unbounded} connected components. Also, if $\parc \Omega$ has exactly two unbounded components, then $\overline\Omega$ is diffeomorphic to $[0,1]\times \R$ and $\Omega$ is a Serrin band (Definition \ref{def:serrin}). Moreover, one has:

\begin{theorem}[Ros-Sicbaldi, \cite{RS}]\label{rosic}
Any Serrin band $\Omega\subset \R^2$ is a symmetric bigraph lying at a finite distance from a line in $\R^2$; that is, up to isometry of $\R^2$, it holds $\Omega=\Omega^{\varphi}$ where
\begin{equation}\label{bigraph}
\Omega^{\varphi}:=\{(x_1,x_2) : -\varphi(x_1)< x_2 < \varphi(x_1)\},
\end{equation} with $\varphi:\R\flecha \R$ a positive, bounded, real analytic function.
\end{theorem}

The simplest examples of such Serrin bands are the \emph{flat bands} $\Omega^{\varphi_0}$, where $\varphi_0>0$ is constant. Non-flat periodic Serrin bands have been constructed in \cite{FMW1} as small bifurcations of flat ones, and in \cite{DDMW} by desingularizing a chain of disks along a line. 

Our main result regarding periodic Serrin bands is Theorem \ref{th:bandas} below. In it, we show that there exists a real analytic $1$-parameter family of periodic Serrin bands in $\R^2$ that presents a Delaunay-type pattern, and whose moduli space connects the geometries observed in \cite{FMW1} and \cite{DDMW}, along with all intermediate situations.

\begin{theorem}\label{th:bandas}
There exists a real analytic family $\{\Omega_\tau :\tau \in (0,1]\}$ of periodic Serrin bands in $\R^2$ with the following properties:
\begin{enumerate}
\item
All domains $\Omega_{\tau}$ are of the form $\Omega^{\varphi_\tau}$ in \eqref{bigraph}, where $\varphi_\tau$ is a real analytic, periodic positive function on $\R$.
\item
If $\tau=1$, then $\varphi_1$ is constant, and so $\Omega_1$ is a flat band.
\item
If $\tau\in (0,1)$, then $\varphi_\tau$ is not constant, and has exactly two critical points (a maximum and a minimum) along each fundamental period.
\item
Different domains $\Omega_\tau$ and $\Omega_{\tau'}$ do not differ by a similarity of $\R^2$.
\item
As $\tau\to 0$, the domains $\Omega_\tau$ converge to a chain of pairwise tangent disks of the same radius along the $x_1$-axis.
\item
The domains $\Omega_\tau$ are described explicitly in terms of elliptic functions.
\end{enumerate}
\end{theorem}

The band domains $\Omega_\tau$ of Theorem \ref{th:bandas} are again foliated by capillary curves but this time, in contrast with the ring domains of Theorem \ref{th:main}, such capillary curves are \emph{planar}. This means that each capillary curve $\gamma\subset \Omega_\tau$ of the foliation has the property that the graph of the solution $u$ intersects a plane of $\R^3$ with a constant angle along $\gamma$.

\begin{figure}[h]
  \centering
  \includegraphics[width=.76\textwidth]{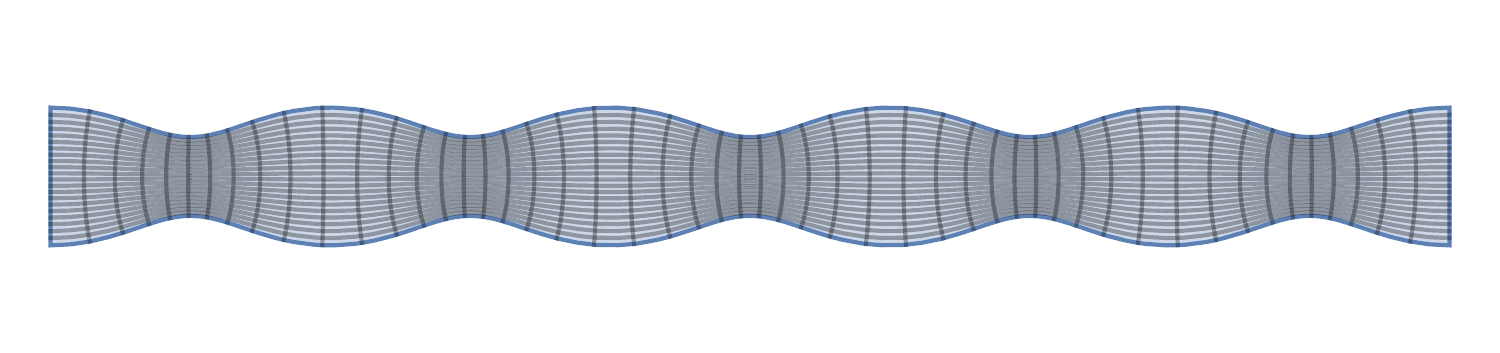} \\  \includegraphics[width=0.71\textwidth]{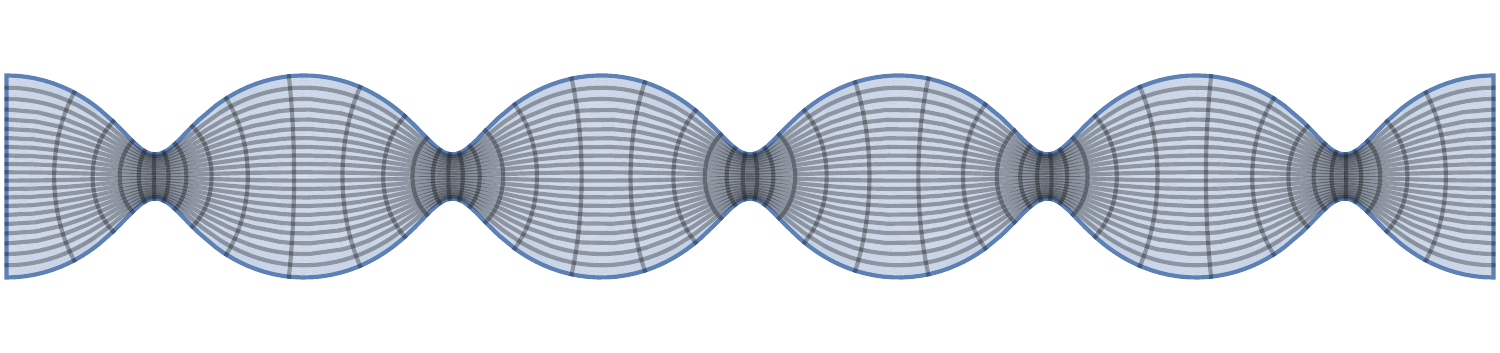} \\ \includegraphics[width=0.72\textwidth]{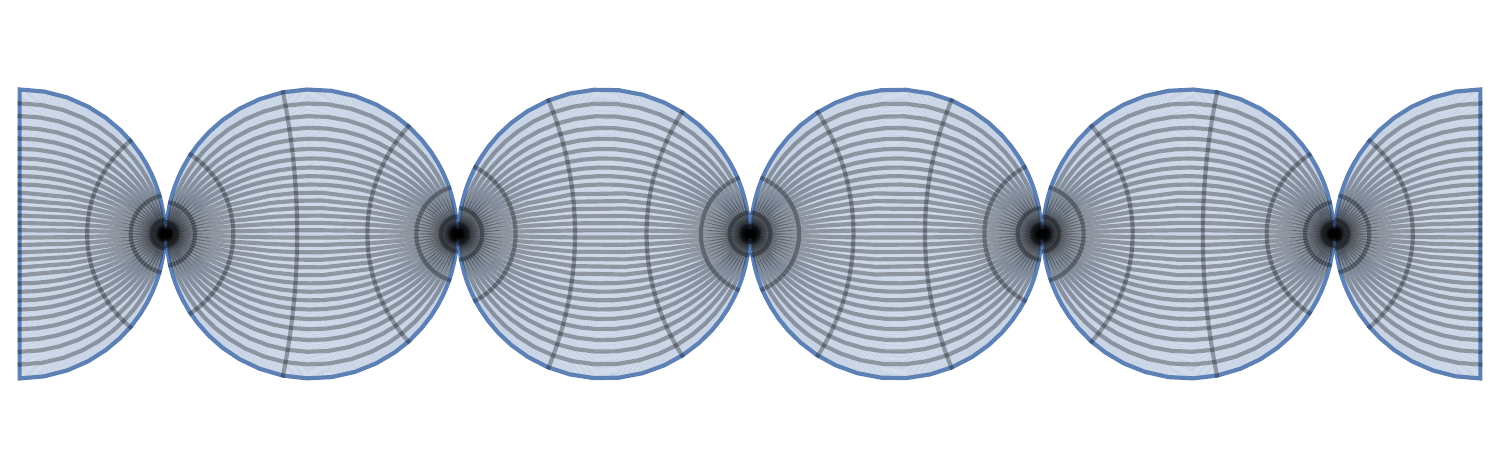}  
  \caption{Top: A Serrin strip domain $\Omega_\tau$ for $\tau$ close to $1$. Middle: $\Omega_\tau$ for an intermediate $\tau\in (0,1)$. Bottom: $\Omega_\tau$ for $\tau>0$ close to zero. In all figures the parameter lines describe a conformal parametrization of the strip. These parametrizations are explicit in terms of a Weierstrass elliptic function $\wp(z)$.\label{figbandas}} 
  \end{figure}

For $\tau$ close enough to $1$, our Serrin bands $\Omega_\tau$ coincide (up to a dilation) with the examples obtained by Fall, Minlend and Weth in \cite{FMW1}. We remark that our approach is totally different to \cite{FMW1}. The proof in \cite{FMW1} actually works in any dimension, and is obtained using Crandall-Rabinowitz bifurcation from flat bands. As a result, the Serrin bands in \cite{FMW1} are \emph{almost flat} and not described explicitly. Our procedure only works for dimension two, but it describes explicitly the resulting domains throughout the whole moduli space, even far way from flat bands.

At the other end of the moduli space, i.e., when $\tau$ is close to zero, our Serrin bands $\Omega_\tau$ converge to a chain of disks along the $x_1$-axis. Serrin bands with this behavior had been recently constructed by Dávila, Del Pino, Musso and Wheeler \cite{DDMW}, by desingularization of such a disk chain. Again, the examples in \cite{DDMW} are shown to exist only close enough to the singular chain of disks. We note that the construction in \cite{DDMW} is part of a more general desingularization process that also applies to the construction of overhanging solitary water waves with constant vorticity. It seems very likely that our domains $\Omega_\tau$ when $\tau\to 0$ coincide (up to similarity) with the Serrin bands in \cite{DDMW}, but we do not have a proof of this fact.

Again, Theorem \ref{th:bandas} is related to minimal and CMC surface theory in different directions. 

On the one hand, the domains $\Omega_\tau\subset \R^2$ appear as the natural analogues of the family of CMC \emph{unduloids} in $\R^3$, obtained by Delaunay in the 19th century. These constitute a $1$-parameter family of embedded rotational CMC surfaces in $\R^3$, with the cylinder on one side of the moduli space, and a singular chain of spheres on the other side. 
We remark that, in contrast, the band domains $\Omega_\tau$ of Theorem \ref{th:bandas} have a discrete symmetry group. 

In another direction, the foliation structure by planar capillary curves of the band domains $\Omega_\tau$ in Theorem \ref{th:bandas} is reminiscent of a classical $1$-parameter family of minimal surfaces discovered by Riemann in the 19th century. These surfaces are foliated by circles in parallel planes, 
and their $1$-dimensional moduli space has as singular limits the catenoid and the helicoid. This relates closely with the structure of the band domains $\Omega_\tau$ in Theorem \ref{th:bandas}, even though Riemann's examples have an infinite number of ends. 

\subsection{The space of Serrin domains}\label{intro:space} Our constructions of Serrin bands and ring domains in Theorems \ref{th:main} and \ref{th:bandas} depend on elliptic functions, and their respective moduli spaces are determined by the coefficients of their defining polynomials. This is not an isolated fact. Our main result regarding the general structure of the space of Serrin ring domains and of periodic Serrin bands (Theorem \ref{th:kdvintro} below) implies that any such domain comes from algebraic data.

To explain this, let $\Omega\subset \R^2$ be a Serrin ring domain, which we view as the image $\overline\Omega=g(\cU)$ of a holomorphic bijection $g$ from a vertical quotient band $\cU\subset \R^2\equiv \C$ given by 
\begin{equation}\label{bandintro}
\cU=([s_1,s_2]\times \R, \sim ), \hspace{0.5cm} (x,y)\sim (x,y+T),
\end{equation}
for some $s_1<s_2$ and $T>0$. We call $g$ the \emph{developing map} of $\Omega$. We assume that the solution $u$ to $\Delta u+2=0$ associated to $\Omega$ is \emph{not} a (trivial) paraboloid solution $P({\bf x})=a -\frac{1}{2}\|{\bf x}-v_0\|^2$.

 
In the same way, any Serrin band $\Omega\subset \R^2$ can be seen in conformal parameters $(x,y)$ as $\overline\Omega=g([s_1,s_2]\times \R)$ with $g(x+iy)$ a holomorphic injective map.
 
On the other hand, we consider the recursive operators $\{Q_n[\eta]: n\in \N\}$ given by the modified Korteweg-de Vries (mKdV) hierarchy. The basic definition and properties of these classical operators will be explained in Section \ref{sec2:kdv}. We merely indicate here that $Q_n[\eta]$ acts on holomorphic functions $\eta(z)$ and their derivatives up to order $2n$ in a polynomial way, and the first elements of the sequence are $Q_{-1}[\eta]=0$,
$$Q_0[\eta] =\eta, \hspace{0.5cm} Q_1[\eta]= \eta''-2\eta^3,\hspace{0.5cm} Q_2[\eta] = \eta^{(4)}-10 \eta'' \eta^2 -10 \eta (\eta')^2 +6 \eta^5.$$
A holomorphic function $\eta(z)$ is called an \emph{algebro-geometric potential} of the mKdV hierarchy if, for some $\mathfrak{m}\in \N$, there exist constants $a_0,c_0,\dots, c_{\mathfrak{m}-1}\in \C$ such that 
\begin{equation}\label{kdvintro1}
Q_{\mathfrak{m}}[\eta]=a_0+\sum_{j=0}^{\mathfrak{m}-1} c_j Q_j[\eta].
\end{equation}
These mKdV potentials have been deeply studied over the last 50 years using techniques from integrable systems and algebraic geometry. We remark for completeness that all of them can be described in terms of Riemann's theta-function on a (maybe singular) compact Riemann surface of genus $\mathfrak{m}$; see e.g. \cite{GW,IM}. However, this theory will not be needed in this paper.

Theorem \ref{th:kdvintro} below introduces Serrin planar domains into this algebraic theory. 

\begin{theorem}\label{th:kdvintro}
Let $\Omega\subset \R^2$ be a Serrin ring domain or a periodic Serrin band. Then, 
\begin{equation}\label{potin}
\eta:=\frac{g''}{2g'}
\end{equation}
is an algebro-geometric potential of the mKdV hierarchy, with real coefficients $a_0,c_j$.
\end{theorem}

In particular, this implies that if $\Omega$ is a Serrin ring domain or a periodic Serrin band, then there exists some $\mathfrak{m}\in \N$ such that the holomorphic function $\eta(z)$ in \eqref{potin} 
satisfies an ODE 
\begin{equation}\label{odekdv}
\eta^{(2\mathfrak{m})} = \cP\left(\eta,\dots, \eta^{(2\mathfrak{m}-1)}\right),
\end{equation}
where $\cP$ is a polynomial with real coefficients depending on the (real) constants $c_0,\dots, c_{\mathfrak{m}-1}$ of \eqref{kdvintro1} in a specific way coherent with the mKdV hierarchy. Thus, $\eta$ is determined by its initial conditions $\eta(z_0),\dots, \eta^{(2\mathfrak{m}-1)}(z_0)$ at any point. So, Theorem \ref{th:kdvintro} proves that any Serrin ring domain (or periodic band) is uniquely determined, for some $\mathfrak{m}\in \N$, by the $(2\mathfrak{m}+1)$-jet of its developing map $g(z)$ at any $z_0\in \cU$. From a certain perspective, this is a rather fascinating property, as it links a global problem like \eqref{serrinprob} or \eqref{overeq00} with a specific finite order condition at one arbitrary point of $\Omega$.

The number $\mathfrak{m}\in \N$ measures the complexity of the domain $\Omega$ and its associated solution $u$ to \eqref{serrinprob} or \eqref{overeq00}. For instance, if $\mathfrak{m}=0$ then $\Omega$ is \emph{trivial}, i.e., a radial disk or a flat band. If $\mathfrak{m}=1$, then we will show that the resulting examples are foliated by capillary curves. Thus, our examples in Theorems \ref{th:main} and \ref{th:bandas} correspond to the $\mathfrak{m}=1$ level of the hierarchy. 

As in our previous results, Theorem \ref{th:kdvintro} is connected to a fundamental theorem of CMC theory, this time due to Pinkall and Sterling \cite{PS}. In \cite{PS} these authors devised a theoretical method that finds \emph{all} CMC tori in $\R^3$ in terms of algebraic equations, via a suitable hierarchy of Jacobi fields of CMC surfaces. 
This algebraic nature of CMC tori in $\R^3$ also holds in space forms $\S^3$, $\H^3$, see Hitchin \cite{Hi} and Bobenko \cite{Bob}, and for capillary CMC annuli in the unit ball of $\R^3$, see Kilian-Smith \cite{KSm}. In \cite{MPR}, Meeks, Pérez and Ros found a hierarchy of Jacobi fields for certain minimal surfaces in $\R^3$, based on the KdV hierarchy, with Shiffman's classical function \cite{Sh} as its first element. The study of its algebraic finiteness was a key element in their proof of the outstanding theorem that any properly embedded minimal planar domain in $\R^3$ of infinite topology is a Riemann example \cite{MPR}.

Our Theorem \ref{th:kdvintro} is inspired by the geometric results in \cite{PS,MPR,KSm}, and proves that all Serrin ring domains and all periodic Serrin bands are algebro-geometric, in the sense of the Pinkall-Sterling theorem \cite{PS} explained above. 
Theoretically, this gives a procedure to construct all such Serrin domains. In practice, one setback with respect to the CMC tori case in \cite{PS} is that, in our situation, the study of the boundary conditions via integrable systems is quite challenging. In particular, we do not know if there exist Serrin domains with a level $\mathfrak{m}>1$ in the mKdV hierarchy. 

Theorem \ref{th:kdvintro} can also be seen as the overdetermined version of the algebraic finite-type property of the Meeks-Pérez-Ros hierarchy for minimal surfaces \cite{MPR} mentioned above.

\subsection{Structure of the paper}
In Section \ref{sec:capi} we introduce the notion of \emph{capillary curve} associated to a solution of $\Delta u+2=0$ in a domain $\Omega\subset \R^2$. We consider ring domains $\Omega\subset \R^2$ bounded by capillary curves, and study their basic properties in terms of a conformal representation $g:\cU\flecha \Omega$, where $\cU$ is a vertical quotient band in $\C$ and $g$ is a conformal isomorphism between $\cU$ and $\Omega$. We call $g$ the \emph{developing map} of the capillary ring domain $\Omega$.

In Section \ref{sec:folicapi} we explain the basic structure of planar domains that are foliated by capillary curves. 
In Section \ref{sec:gaussmaps} we construct a $2$-parameter family of holomorphic mappings $g(z)$ compatible with this foliation structure and study their properties. These two sections are inspired by our previous works \cite{CFM,FHM} on free boundary minimal and CMC annuli in $\R^3$.

In Section \ref{sec:immersed}, the holomorphic maps from Section \ref{sec:gaussmaps} are used to create, for each $n\geq 2$, a $2$-parameter family of \emph{immersed} Serrin ring domains in $\R^2$ with a dihedral symmetry group of order $2n$ with respect to the origin. Here, \emph{immersed} means that these domains are given as $\Omega=g(\cU)$ but the developing map $g:\cU\flecha \Omega$ might not be injective. Thus, $\Omega$ might have self-intersections.

In Section \ref{sec:embedded} we characterize the \emph{embeddedness} of $\Omega$, i.e., the injectivity of $g$, within our $2$-dimensional moduli space of immersed domains. Calling $(s,\tau)$ to the coordinates of this moduli space, with $\tau\in (0,1)$, the limit case $\tau=1$ corresponds to radial annuli. In Section \ref{sec:necklaces} we show that as $\tau\to 0$, the annuli $\Omega_{(s,\tau)}$ converge to a necklace of pairwise tangent disks along the unit circle $\S^1\subset \R^2$. In Section \ref{sec:mainth} we put together all this theory and complete the proof of Theorem \ref{th:main}. 

In Section \ref{sec:bandas} we deal with the periodic, unbounded, Serrin problem and prove Theorem \ref{th:bandas}.
For that we use the theory developed in previous sections, but we impose a simpler ansatz, by considering foliations with \emph{planar} capillary curves. This simplification makes the construction more direct than in the case of Serrin ring domains. The embeddedness is also simpler, by Ros-Sicbaldi's Theorem \ref{rosic}.

In Section \ref{sec:kdv} we prove Theorem \ref{th:kdvintro} in a more general setting, for ring domains bounded by capillary curves (not necessarily of Serrin type); see Theorem \ref{th:finite2} and Corollary \ref{cor:finite}. The proof also works for periodic bands. We also characterize the foliation property by capillary curves of these domains in terms of the mKdV hierarchy. The paper closes with Section \ref{sec:problems}, where we discuss the scope of our results as well as some open problems, and with an Appendix where we describe the holomorphic maps constructed in Section \ref{sec:gaussmaps} in terms of Weierstrass $\wp$ functions.

We remark that, in our presentation, we have avoided the fundamental theory of integrable systems via Lax pairs and loop groups, in order to make the paper as self-contained as possible, and accessible to researchers primarily interested in overdetermined elliptic problems. For instance, our discussion of the mKdV hierarchy is carried out in a simple, direct way via recursion operators, and is strictly focused to proving the connection of mKdV potentials with Serrin's overdetermined problem.



{\bf Acknowledgement:} We are grateful to David Ruiz and Pieralberto Sicbaldi for bringing our attention to the global bifurcation problem of periodic Serrin bands, and for pointing out some relevant references for it. We also thank Laurent Hauswirth for many instructive conversations on integrable systems in surface theory.

\section{Capillary curves and solutions to $\Delta u +2=0$}\label{sec:capi}

\subsection{Capillary curves}
Let $u=u(x_1,x_2)$ denote a $C^2$ function on a planar domain $\Omega\subset \R^2$. We say that a regular curve $\gamma(t)$ in $\Omega$ is a \emph{Hessian eigencurve} of $u$ if $\gamma'(t)$ is an eigenvector of $D^2 u(\gamma(t))$ for every $t$. The $D^2u$-eigendirections $w=(w_1,w_2)\in \R^2-\{{\bf 0}\}$ are the solutions to
\begin{equation}\label{prineq0}
-u_{12} (w_1^2-w_2^2) + (u_{11}-u_{22}) w_1 w_2 =0,
\end{equation}
where $u_{11}:=u_{x_1x_1}$, etc. So, $\gamma(t)=(x_1(t),x_2(t))$ is a Hessian eigencurve if and only if
\begin{equation}\label{prineq}
-u_{12}(\gamma(t)) (x_1'(t)^2-x_2'(t)^2) + \big(u_{11}(\gamma(t))-u_{22}(\gamma(t))\big) x_1'(t) x_2'(t) =0
\end{equation}
holds for every $t$. 

Consider next the polynomial 
\begin{equation}\label{cupo}
P(x_1,x_2):= a+ d_1 x_1 + d_2 x_2 -\frac{c}{2}(x_1^2+x_2^2), \hspace{0.5cm} a,d_1,d_2,c\in \R.
\end{equation} Its graph is a rotational elliptic paraboloid $\cP$ in $\R^3$ if $c\neq 0$, or a plane if $c=0$. The following elementary lemma will be useful for our discussion. 

\begin{lemma}\label{joach}
Assume that $u(\gamma(t))=P(\gamma(t))$ along a regular curve $\gamma(t)$ in $\Omega$ for some polynomial $P$ as in \eqref{cupo}. The following three conditions are equivalent:
\begin{enumerate}
\item[(i)]
$\gamma(t)$ is a Hessian eigencurve of $u$.
\item[(ii)]
$\|\nabla (u-P)(\gamma(t))\|$ is constant.
\item[(iii)]
The intersection angle of $z=u-P$ and the plane $z=0$ is constant along $\gamma(t)$.
\end{enumerate}
\end{lemma}
\begin{proof}
The equivalence between (ii) and (iii) is immediate. To prove that (ii) implies (i) we first consider the case $c=d_j=0$, i.e., $P$ in \eqref{cupo} is constant. In that case, $u(\gamma(t))$ is constant, and so
\begin{equation}\label{joa1}
\nabla u(\gamma(t)) = \landa(t) J (\gamma'(t))
\end{equation} 
for some function $\landa(t)$, where $J$ is the $\pi/2$ rotation in $\R^2$. Assuming (ii), $t\mapsto \|\nabla u(\gamma(t))\|^2$ is constant. Differentiating this expression and using \eqref{joa1}, we obtain directly that \eqref{prineq} holds along $\gamma(t)$ if $\landa(t)\neq 0$. If $\landa(t)\equiv 0$, then it follows directly from \eqref{joa1} that \eqref{prineq} also holds. Thus, in any case, $\gamma(t)$ is a Hessian eigencurve, and we have proved (ii) $\Rightarrow$ (i) for the case $c=d_j=0$.

For the general case, we denote $\hat{u}:=u-P$. It is immediate that \eqref{prineq} holds for $u$ if and only if it holds for $\hat{u}$. Also, assuming (ii), both $\hat{u}(\gamma(t))$ and $\|\nabla \hat{u}(\gamma(t)\|$ are constant. Therefore we can reduce the implication (ii) $\Rightarrow$ (i) to the case $c=d_j=0$, which has already been proved.

These computations can be easily reversed to prove that (i) $\Rightarrow$ (ii) for $c=d_j=0$, and so for every $P$. This proves Lemma \ref{joach}.
\end{proof}

\begin{definition}\label{def:capi}
We say that a curve $\gamma(t)$ in $\Omega$ is a \emph{capillary curve} for $u:\Omega\flecha \R$ if $\gamma(t)$ satisfies the conditions in Lemma \ref{joach}; that is, if $(u-P)(\gamma(t))=0$ and $\|\nabla (u-P)(\gamma(t))\|={\rm const}$ for some polynomial $P$ as in \eqref{cupo}.
\end{definition}

\subsection{Basic formulas in conformal parameters}\label{sec:confpa}
Let $u(x_1,x_2)$ be a solution to $\Delta u+2=0$ on a domain $\Omega\subset \R^2$, and let $g(z)$ be an injective holomorphic map with image contained in $\Omega$. We identify $z=x+iy$ with $(x,y)\in \R^2$. Since $g$ is holomorphic and $g'$ never vanishes, the function
\begin{equation}\label{def:ome}
\omega(x,y):=\log |g'(z)|
\end{equation} is well defined and harmonic, and we have $2\omega_z = g''/g'$. If we let  
\begin{equation}\label{defuve}
v(x,y):=u(g(x+iy)),
\end{equation} 
then equation $\Delta u +2 =0$ transforms under $g(z)$ to 
\begin{equation}\label{pdeom}
\Delta v +2 e^{2\omega}=0,
\end{equation} 
where $\Delta v= v_{xx}+v_{yy}$. Thus,
\begin{equation}\label{rep}
v= h +\sigma, \hspace{0.5cm} \sigma:=- |g|^2/2,
\end{equation}for some harmonic function $h=h(x,y)$. 
An immediate consequence of \eqref{pdeom} is that
\begin{equation}\label{def:q2}
\mathfrak{q}:=v_{zz}-2\omega_z v_z
\end{equation}
is a holomorphic function, i.e. $\mathfrak{q}_{\bar{z}}=0$.

The holomorphic function $\mathfrak{q}$ detects the eigenlines of $D^2 u$. To see this, we first remark that \eqref{prineq0} can be written in complex terms as 
\begin{equation}\label{prinecom0}
{\rm Im}\left(u_{\zeta \zeta}\,d\zeta^2 \right)=0,
\end{equation}
where $\zeta:=x_1+ix_2$ and $d\zeta= w_1+iw_2$. Thus, \eqref{prinecom0} defines a holomorphic quadratic differential. We now make the conformal coordinate change $\zeta=g(z)$ on this quadratic differential, and deduce from \eqref{def:q2} and \eqref{defuve} that \eqref{prinecom0} is written as 
\begin{equation}\label{prineq2}
{\rm Im}\left(\mathfrak{q}\, dz^2 \right)=0
\end{equation}
with respect to the $z$ parameter. 
So, a curve $z(t)$ in the conformal parameter $z=x+iy$ is a Hessian eigencurve for $u$ if and only if 
\begin{equation}\label{prineq3}
{\rm Im}(\mathfrak{q}(z(t))\,z'(t)^2)=0 \hspace{0.5cm} \forall \, t.
\end{equation}
Motivated by surface theory, we call $\mathfrak{q}$ the \emph{Hopf differential} associated to $u$ in the conformal parameter $z=x+iy$. 

\begin{remark}\label{rem:arriba}
The Hopf differential vanishes, i.e. $\mathfrak{q}=0$, if and only if $u_{\zeta \zeta}=0$, and so $u=P$ for some polynomial $P$ as in \eqref{cupo}, with $c=1$ (in order for $\Delta u+2=0$ to hold). 
\end{remark}
Note that the function $\sigma$ in \eqref{rep} is a solution to \eqref{pdeom} with $\mathfrak{q}=0$, i.e., 
\begin{equation}\label{sizz}
\sigma_{zz}= 2\omega_z \sigma_z.
\end{equation}
Conversely, it follows by Remark \ref{rem:arriba} above that any solution to \eqref{pdeom} with $\mathfrak{q}=0$ is of the form $-|A g(z)+B|^2/2$ for some $A,B\in \C$, $A\neq 0$.

\subsection{Detecting capillary curves in conformal parameters}\label{sec:cacon}
Following the notations in Section \ref{sec:confpa}, we consider a solution $u(x_1,x_2)$ to $\Delta u+2=0$ in a smooth domain $\Omega\subset \R^2$, and we let $g(z):\cV\flecha \Omega$ be a bijective holomorphic mapping from a complex domain $\cV$ into $\Omega$.

\begin{definition}\label{eigenline}
We say that $g(z)$ is a \emph{conformal eigenline parametrization} of $\Omega$ for $u$ if each parameter curve $y\mapsto g(x_0+iy)$ is a Hessian eigencurve for $u$. In that situation, the $x$-curves $x\mapsto g(x+iy_0)$ are also Hessian eigencurves of $u$. 
\end{definition}
It follows directly from \eqref{prineq3} that $g(z)$ is a conformal eigenline parametrization of $\Omega$ for $u$ if and only if the Hopf differential $\mathfrak{q}$ of $u$ for the conformal parameter $z=x+iy$ is a real constant:
\begin{equation}\label{qcons}
\mathfrak{q} = {\rm const} \in \R.
\end{equation} 
In that situation, the Hessian eigencurve $y\mapsto g(x_0+iy):\R\flecha \Omega$ will be a capillary curve for $u$ if and only if $v(x_0,y)=P(g(x_0+iy))$ for some $P$ as in \eqref{cupo}, i.e., if and only if there exist $a,d_1,d_2,c\in \R$ so that
\begin{equation}\label{ove1}
v(x_0,y)=a+ c \, \sigma(x_0,y)+L(x_0,y),
\end{equation} 
where $\sigma(x,y)$ is given by \eqref{rep} and 
\begin{equation}\label{def:ele}
L(x,y):=d_1 {\rm Re}(g(x+iy)) + d_2  {\rm Im}(g(x+iy)).
\end{equation}
If $g(x_0+iy)$ is capillary, then from item (ii) in Lemma \ref{joach} and $|g'|=e^{\omega}$ we obtain that $$ \|\nabla (u-P)(g)\| = 2 \left|(u-P)_\zeta (g)\right| = 2 e^{-\omega} \left|(v-(a+c \,\sigma +L))_z\right|$$ is constant along $x=x_0$. Using now \eqref{ove1}, this means that 
\begin{equation}\label{ove2}
v_x(x_0,y)= b\, e^{\omega(x_0,y)}+ c \, \sigma_x(x_0,y)+L_x(x_0,y)
\end{equation}
also holds, for some $b\in \R$. 

\begin{definition}\label{capilarcons}
We call $b$ the \emph{capillarity constant} associated to the capillary curve $g(x_0+iy)$. We note that $|b|$ is the constant value of $\|\nabla (u-P)(\gamma(t))\|$ in item (ii) of Lemma \ref{joach}.
\end{definition}

\begin{remark}[The role of the Hopf differential] By \eqref{rep}, we can write $v=h+\sigma$, with $h$ a harmonic function on $\cV$. Let $\mathfrak{q}\in \R\setminus\{0\}$ be the Hopf differential of $u$ for the complex parameter $z=x+iy$. Then, for any $\landa>0$, the function $v^{\landa}:= \landa h + \sigma$ is a solution to \eqref{pdeom} in $\cV$ with Hopf differential $\landa \mathfrak{q}$, and the same capillary curves as $v$. Specifically, if $g(x_0+iy)$ is a capillary curve for $v$ with associated constants $a,b,c, d_1,d_2$ as in equations \eqref{ove1}, \eqref{def:ele}, \eqref{ove2}, then it is also capillary for $v^{\landa}$, with new constants
$$a_\landa= \landa a, \hspace{0.5cm} b_\landa = \landa b, \hspace{0.5cm} (d_j)_{\landa}= \landa d_j, \hspace{0.5cm} c_\landa= 1+ \landa(c-1).$$ 
Note, however, that the property of $u$ being constant along $g(x_0+iy)$, i.e. the property that $d_1=d_2=c=0$, is not preserved, because $c_\landa$ does not have the same zeros as $c$. 
\end{remark}

The next proposition characterizes capillary curves of $\Omega$ in terms of $\omega$.

\begin{proposition}\label{ekicapi}
Let $u$ be a solution to $\Delta u+2=0$ on a smooth domain $\Omega\subset \R^2$, and let $g:\cV\flecha \Omega$ be a conformal eigenline parametrization of $\Omega$ for $u$.  If $y\mapsto g(x_0,y)$ is a capillary curve of $\Omega$ with capillarity constant $b\neq 0$, then, for any $(x_0,y)$ it holds 
\begin{equation}\label{capicon}
2\omega_x = -\alfa e^{-\omega} -\beta e^{\omega},
\end{equation}
where $\alfa,\beta\in \R$ are given in terms of $a,b,c$ in \eqref{ove1}, \eqref{ove2} by
\begin{equation}\label{relcon}
\alfa= \frac{4\mathfrak{q}}{b}, \hspace{0.5cm} \beta= \frac{2(1-c)}{b}.
\end{equation}
And conversely, if \eqref{capicon} holds along $(x_0,y)$ for some $\alfa\neq 0$, $\beta\in \R$, then the Hessian eigencurve $y\mapsto g(x_0,y)$ is capillary, and \eqref{ove1}, \eqref{ove2} hold for $a,b,c$ as in \eqref{relcon} and for some $d_1,d_2\in \R$.
\end{proposition}
\begin{proof}
Note first of all that, since ${\rm Re}(v_{zz}-2\omega_z v_z) =\mathfrak{q}$, we can deduce from \eqref{pdeom} that 
\begin{equation}\label{vi1}
v_{yy}= - \left(2\mathfrak{q}+e^{2\omega}+v_{x}\omega_x- v_y \omega_y\right).
\end{equation} The same computation for $\sigma$ using \eqref{sizz} and $\Delta \sigma +2 e^{2\omega}=0$ gives 
\begin{equation}\label{vi2}
\sigma_{yy}= - \left(e^{2\omega}+\sigma_x\omega_x -\sigma_y \omega_y\right).
\end{equation}
And since $L(x,y)$ in \eqref{def:ele} is harmonic with $L_{zz}-2\omega_zL_z=0$, we obtain in the same way
\begin{equation}\label{vi22}
L_{yy}= -(L_x \omega_x - L_y \omega_y).
\end{equation}
Similarly, by looking at the corresponding imaginary parts, and since $\mathfrak{q}\in \R$, we have
\begin{equation}\label{vi23}
v_{xy}= \omega_x v_y + \omega_y v_x, \hspace{0.5cm} 
\sigma_{xy}= \omega_x \sigma_y + \omega_y \sigma_x, \hspace{0.5cm} 
L_{xy}= \omega_x L_y + \omega_y L_x.
\end{equation}
Assume now that $y\mapsto g(x_0+iy)$ is a capillary curve, and that $b\neq 0$ in \eqref{ove2}. By \eqref{ove1}, we have $v_{yy}=c\, \sigma_{yy}+L_{yy}$ along $(x_0,y)$. Plugging \eqref{vi1}, \eqref{vi2}  and \eqref{vi22} above into this equation, and using \eqref{ove2} together with $v_{y}=c\, \sigma_{y}+L_{y}$ to cancel out terms, we deduce that 
\begin{equation}\label{ecube}
b\,e^{\omega}\omega_x = (c-1) \, e^{2\omega} -2\mathfrak{q}  \hspace{0.5cm} \text{along $y\mapsto (x_0,y)$}.
\end{equation} 
Since $b\neq 0$, we obtain
$$2\omega_x = -\alfa \, e^{-\omega} -\beta\, e^{\omega} \hspace{0.5cm} \text{along $y\mapsto (x_0,y)$},$$ where $\alfa,\beta\in \R$ are given by \eqref{relcon}.

For the converse, we assume that $\omega(x,y)$ satisfies \eqref{capicon} along $(x_0,y)$ for some $\alfa\neq 0$. Since $y\mapsto g(x_0+iy)$ is a Hessian eigencurve, it suffices to show that \eqref{ove1} holds for every $(x_0,y)$. To start, we consider the functions $F(y):=v_x(x_0,y)$ and $G(y):=v_y(x_0,y)$. By decomposing \eqref{def:q2} into real and imaginary parts and using \eqref{pdeom}, we obtain that $(F,G)$ solves the linear ODE system
\begin{equation}\label{sisfg}
F_y =\omega_x G + \omega_y F, \hspace{0.5cm} G_y = -2\mathfrak{q} - e^{2\omega}-\omega_x F+\omega_y G.
\end{equation} 
On the other hand, let $$\hat{F}(y):=(b e^\omega + c\sigma_x + L_x)(x_0,y), \hspace{0.5cm} \hat{G}(y):= (c\sigma_y + L_y)(x_0,y),$$ where $b,c,d_1,d_2$ are real constants to be determined later on. Then, an elementary computation using \eqref{vi23} along $(x_0,y)$ proves that $(\hat{F},\hat{G})$ solves system \eqref{sisfg} if and only if \eqref{ecube} holds. Chosing $b,c$ so that $\alfa b = 4\mathfrak{q}$ and $\beta b = 2(c-1)$, we deduce from \eqref{capicon} that $(\hat{F},\hat{G})$ is also a solution to \eqref{sisfg}.

In addition, it is direct from \eqref{def:ele} that $$(L_x-iL_y )(x_0,0)=(d_1-i d_2)g'(x_0).$$ Since $g'(x_0)\neq 0$, we can therefore choose unique $d_1,d_2$ so that $(\hat{F},\hat{G})(x_0,0)=(F,G)(x_0,0)$. 

Now, taking into account that any solution to system \eqref{sisfg} is uniquely determined by its value at $y=0$, we conclude that $(\hat{F},\hat{G})=(F,G)$ for every $y$. From $G(y)=\hat{G}(y)$ we obtain after integrating in $y$ that there is some $a\in \R$ such that \eqref{ove1} holds along $(x_0,y)$. This completes the proof.
\end{proof}

We compute next a useful formula for the curvature of a capillary curve.
\begin{lemma}\label{lem:curvatura}
In the conditions of Proposition \ref{ekicapi} above, the curvature $\kappa_\gamma$  of the planar curve $\gamma(y):=g(x_0,y)$ is given by
\begin{equation}\label{curvy}
\kappa_{\gamma}= e^{-\omega}\, {\rm Re} (g'' /g').
\end{equation} 
Moreover, if $\omega$ satisfies \eqref{capicon} along $(x_0,y)$, then
\begin{equation}\label{curcapi}
-2\kappa_\gamma= \beta+\alfa \, e^{-2\omega(x_0,y)}.
\end{equation}
In particular, when $\alfa=0$, $\gamma(y)$ is a circle (or line) of curvature $-\beta/2$.
\end{lemma}
\begin{proof}
First, we work along $(x_0,y)$ and we do not use the capillary hypothesis \eqref{capicon}. Using that $|\gamma'|^2 =|g'|^2= e^{2\omega}$, and $\esiz z_1,z_2\esde =  {\rm Re}(z_1 \overline{z_2})$, we have
\begin{equation*}
\kappa_{\gamma}= \frac{\esiz \gamma'',J\gamma'\esde}{|\gamma'|^3} = e^{-3\omega}\, {\rm Re} (g'' \overline{g'}) = e^{-\omega}\, {\rm Re} (g'' /g').
\end{equation*} 
Recalling now that $g''/g'=2\omega_z$, we have ${\rm Re}(g''/g')=\omega_x$. So, if \eqref{capicon} holds along $(x_0,y)$, we immediately obtain \eqref{curcapi}.
\end{proof}

\subsection{Capillary domains}\label{seccapi} 
By a \emph{ring domain} in $\R^2$ we mean a smooth doubly connected domain $\Omega\subset \R^2$ bounded by two Jordan curves $\gamma_1,\gamma_2$. By a \emph{band domain} we mean a smooth simply connected (unbounded) domain $\Omega\subset \R^2$ with two connected boundary components: $\parc \Omega=\gamma_1\cup \gamma_2$.

We will  be specially interested in a certain type of ring and band domains of $\R^2$:

\begin{definition}\label{def:capicu}
A ring (resp. band) domain $\Omega\subset \R^2$ is a \emph{capillary ring domain} (resp. \emph{capillary band}) if there exists a solution $u$ to $\Delta u +2=0$ in $\Omega$ such that each component $\gamma_1,\gamma_2$ of $\parc \Omega$ is a capillary curve for $u$, as in Definition \ref{def:capi}.

Observe that any Serrin ring domain or any Serrin band (Definition \ref{def:serrin}) is, trivially, a capillary ring domain or a capillary band, respectively.
\end{definition}

Let $\Omega$ be a ring domain in $\R^2$. By uniformization, there exists a holomorphic bijection 
\begin{equation}\label{deman}
g:\cU\flecha \overline\Omega\subset \R^2,
\end{equation} 
where $\cU$ is a vertical quotient band 
\begin{equation}\label{quoban}
\cU=\left([s_1,s_2]\times \R, \sim \right), \hspace{0.5cm} (x,y)\sim (x,y+T),
\end{equation}
for some $s_1<s_2$ and $T>0$. Thus, we have $g(z+iT)=g(z)$, and each boundary curve $\{x=s_j\}$ of $\parc\cU$ goes to a boundary curve $\gamma_j$ of $\parc \Omega$, for $j=1,2$.
We can choose $g(z)$ so that $g(s_1,y)$ parametrizes the \emph{exterior} component of $\parc \Omega$. In that situation, $g$ is unique up to the change $g(z)\mapsto g(\landa z+z_0)$, for $z_0\in \C$ and $\landa>0$.

Assume next that $\Omega$ is a capillary ring domain associated to a solution $u$ to $\Delta u +2=0$. Let $v:=u\circ g: \cU\flecha \R$ be defined by \eqref{defuve}. As $\parc \Omega$ is capillary, then \eqref{prineq3} holds along the curves $t\mapsto (s_j,t)$ of $\parc \cU$. This means that ${\rm Im}(\mathfrak{q})=0$ on $\parc \cU$. Since $\mathfrak{q}$ is holomorphic and $\cU$ is compact, we deduce by the maximum principle that \eqref{qcons} holds. In particular, $g(z)$ is a conformal parametrization by Hessian eigencurves of $\Omega$.

Consider now the case that $\Omega$ is a band domain in $\R^2$ that is \emph{periodic}, that is, $\Omega=\Omega+ {\bf v}_0$ for some ${\bf v}_0\in \R^2-\{{\bf 0}\}$. Then, there exists a holomorphic bijection 
\begin{equation}\label{demanflat}
g:[s_1,s_2]\times \R\flecha \overline\Omega\subset \R^2,
\end{equation} 
such that $g'(z)$ is well defined on a quotient band $\cU$ as in \eqref{quoban}. 

Suppose now that $\Omega$ is a periodic capillary band associated to a solution $u$ to $\Delta u+2=0$ in $\Omega$ that is also ${\bf v}_0$-periodic, i.e., $u({\bf x}+{\bf v}_0)= u({\bf x})$ for every ${\bf x}\in \Omega$. Then, it is direct that $\mathfrak{q}$ is well defined on the quotient $\cU$, and due to the capillarity boundary condition we have  ${\rm Im}(\mathfrak{q})=0$ along $\parc \cU$. As in the case of ring domains, we obtain that \eqref{qcons} holds, and $g(z)$ is a conformal parametrization by Hessian eigencurves of $\Omega$.


\begin{definition}\label{develop}
In the above conditions, we call $g(z)$ the \emph{developing map} of the capillary domain $\Omega$, and say that $(x,y)\in \cU$ are \emph{conformal eigenline parameters} of $\Omega$.
\end{definition}

%
%
%
%
%
%
%

\section{Foliation by capillary curves: local description}\label{sec:folicapi}
In this section we describe from a local perspective the domains $\Omega\subset\R^2$ that are foliated by capillary curves with respect to a solution $u$ to $\Delta u +2=0$, with a special focus on the case that all paraboloids associated to these capillary curves have the same rotation axis. 

\subsection{Local study}\label{sec:local}
Motivated by the discussion in Section \ref{sec:cacon}, we define next:

\begin{definition}\label{def:capring}
Let $u$ be a solution to $\Delta u+2=0$ on a smooth domain $\Omega\subset \R^2$, and let 
$$g(z):\cV:=(s_1,s_2)\times (t_1,t_2) \flecha\Omega$$ be a conformal eigenline parametrization of $\Omega$ for $u$. 
 We say that $\Omega$ is \emph{foliated by capillary curves} (with respect to $u$) if there exist real analytic functions $\alfa(x),\beta(x):(s_1,s_2) \flecha \R$ such that  
\begin{equation}\label{omex}
2\omega_x =- \alfa(x) e^{-\omega} - \beta(x) e^{\omega}
\end{equation}
holds on $\cV$, where $\omega=\log |g'|$. 

By abuse of language, we will also say that the solution $u$ is \emph{foliated by capillary curves}.
\end{definition}
\begin{remark}
Note that if \eqref{omex} holds, it follows by Proposition \ref{ekicapi} that each curve $y\mapsto g(x_0+iy)$ for which $\alfa(x_0)\neq 0$ is capillary for $u$. 
\end{remark}


Consider next a solution $u$ to $\Delta u+2=0$ foliated by capillary curves (Definition \ref{def:capring}). We will assume that $u$ is not a rotational polynomial $P(x_1,x_2)$ as in \eqref{cupo}. We will also assume for simplicity here that $\omega_y\not\equiv 0$. Then, with the notations of Section \ref{sec:confpa}, we have:
\begin{enumerate}
\item[(i)]
The Hopf differential $\mathfrak{q}:=v_{zz}-2\omega_z v_z$ is a real constant $\mathfrak{q}\neq 0$; see Remark \ref{rem:arriba}.
\item[(ii)]
If $\alfa(x)\neq 0$, the curve $y\mapsto (g(x+iy),v(x,y))$ lies in a rotational ``paraboloid'' $\cP(x)\subset \R^3$.
\end{enumerate}
Here, we recall that $\cP(x)$ can degenerate to a plane, see \eqref{cupo}. With this, we have:

\begin{proposition}\label{pro:sis}
The functions $\alfa(x),\beta(x)$ solve the system
\begin{equation}\label{system}
\left\{ \def\arraystretch{1.5}\begin{array}{lll} \alfa'' & = &  \delta \alfa - 2 \alfa^2 \beta, \\ \beta'' & = &  \delta \beta - 
2 \alfa \beta^2, \end{array} \right.
 \end{equation}
with respect to some constant $\delta\in \R$.
\end{proposition}
\begin{proof}
We first compute $4\omega_y \omega_{yy} = -4\omega_y (\omega_{x})_x$ using \eqref{omex} and $\Delta \omega=0$ to obtain 
$$4\omega_y \omega_{yy} = e^{-2\omega} \omega_y \left( \alfa^2 +2 \alfa' e^{\omega} + 2\beta' e^{3\omega}-\beta^2 e^{4\omega}\right).$$ 
Integrating with respect to $y$, this yields 
\begin{equation}\label{eq:wy2}
4\omega_y^2 = -e^{-2\omega} \left(\alfa^2 +4 \alfa' e^{\omega} -4\beta' e^{3\omega}+\beta^2 e^{4\omega}\right)+ A(x),
\end{equation} 
for some function $A(x)$. Next, we differentiate \eqref{eq:wy2} with respect to $x$ and use \eqref{omex} to obtain an expression for $
8\omega_y\omega_{yx}$. On the other hand, we can also obtain an expression for $8\omega_y\omega_{yx}$ by differentiation of \eqref{omex} with respect to $y$ and the use of \eqref{eq:wy2}. By comparing these two expressions, we arrive at
\begin{equation}\label{eqa}
A'+ e^{\omega}  \left(4 \beta'' + 2 \alpha \beta^2+A\beta\right)
=
6 \beta \alpha'+6 \alpha \beta' + 
e^{-\omega} \left(4 \alpha''+2 \alpha^2 \beta+A\alpha \right).
\end{equation}
Note that we can view \eqref{eqa}  as an equation of the form $P(x,Z)=0$, where $Z:=e^{-\omega}$ and $P(x,Z)$ is a second degree polynomial in $Z$ with coefficients depending solely on $x$. Since $\omega_y\not\equiv 0$, all these coefficients must vanish, and we obtain 
\begin{equation}\label{eqa2}
\def\arraystretch{1.4}\begin{array}{lll}4\alpha'' &=& - 2 \alpha^2 \beta - A\alpha, \\
4 \beta'' &=& - 2 \alpha \beta^2 - A\beta, \\
A &= & 6\alpha\beta -4 \delta,
\end{array}
\end{equation}
where $\delta\in\R$ is a constant.  This implies \eqref{system} and finishes the proof. 
\end{proof}

\begin{remark}\label{sicha}
The change $(\alfa,\beta)\mapsto (\landa \alfa, \landa^{-1} \beta)$ preserves solutions of \eqref{system}. It corresponds to adding a constant $\log \landa$ to $\omega$, that is, to a dilation $g\mapsto \landa g$ for the developing map $g(z)$.
\end{remark}

It is immediate to check that if $(\alfa(x),\beta(x))$ solves \eqref{system}, then
\begin{equation}\label{khami}
\kappa_1:=\alfa' (x)\beta(x) -\beta'(x) \alfa(x) 
\end{equation} 
and 
\begin{equation}\label{khami2}
\kappa_2 := \alfa'(x)\beta'(x)- \delta \alfa(x)\beta(x) + \alfa(x)^2\beta(x)^2
\end{equation}
are real constants. These first integrals come from the fact that \eqref{system} has a Hamiltonian structure, that will not be used explicitly in this paper.

\begin{remark}
From \eqref{eq:wy2} and \eqref{eqa2} we obtain the useful identity \begin{equation}\label{roy}
 4e^{-2 \omega} \omega_y^2 = 
 -\alfa^2 e^{-4\omega} 
 -4 \alpha' e^{-3\omega}
+(6 \alfa \beta -4 \delta)e^{-2 \omega} 
+4 \beta ' e^{- \omega}
-\beta^2. 
\end{equation}
\end{remark}

\noindent \underline{\bf Note:} From now on, along the rest of this section we will assume furthermore that all the ``paraboloids'' $\cP(x)$ (see item (ii) at the beginning of the section) \emph{have the same rotational axis}, equal to the $x_3$-axis. 

The above property means that, for any $x$ such that $\alfa(x)\neq 0$, the graphing function of $\cP(x)$ is of the form \eqref{cupo}, with $d_1=d_2=0$. Using now \eqref{ove1}, \eqref{ove2}, we obtain the existence of real analytic functions $a(x),b(x),c(x)$ such that
\begin{equation}\label{sobredet4}
v= a(x) + c(x) \sigma, \hspace{0.5cm} v_x = b(x) e^{\omega} + c(x) \sigma_x
\end{equation}
hold for every $(x,y)$ with $\alfa(x)\neq 0$; here, we are denoting as usual $\sigma:=-|g|^2/2$, as in \eqref{rep}. Moreover, in that situation $\cP(x)$ will be a (horizontal) plane if and only if $c(x)=0$.

We will assume next for simplicity that $\omega_y\not\equiv 0$ and work around points where $\alfa(x)\neq 0$. We note that if $\omega_y\equiv 0$, then $\omega$ is a linear function on $x$. This gives rise to radial solutions to $\Delta u+2=0$, as will be explained in Section \ref{sec:tau1}. 


System \eqref{sobredet4} has the compatibility condition
\begin{equation}\label{compati}
a'(x) + c'(x)\sigma =b(x) e^{\omega}.
\end{equation}
Also, by \eqref{relcon}, we have the relation
\begin{equation}\label{relcon2}
\alfa(x)= \frac{4\mathfrak{q}}{b(x)}, \hspace{0.5cm} \beta(x)= \frac{2(1-c(x))}{b(x)}.
\end{equation}
Let us now discuss how to obtain $a,b,c$ from $\alfa,\beta$. First, inverting \eqref{relcon2} we obtain 
\begin{equation}\label{relcon3}
b(x)=\frac{4 \mathfrak{q}}{\alfa(x)}, \hspace{0.5cm} c(x)=1-\frac{2\mathfrak{q}\beta(x)}{\alfa(x)}.
\end{equation}
Also, from \eqref{relcon3} and \eqref{khami} we have 
\begin{equation}\label{relcon4}
c'(x)= \frac{2\kappa_1 \mathfrak{q}}{\alfa(x)^2}.
\end{equation}
The way to obtain $a(x)$ from $\alfa(x),\beta(x)$ is a bit more indirect. First, we use \eqref{sobredet4} and \eqref{compati} to compute $0=\Delta v +2 e^{2\omega}=(v_x)_x+v_{yy}+2e^{2\omega}$. This gives $$\frac{a' c''}{c'}-a'' + e^{\omega} \left(2 b \omega_x+2 b'-\frac{b c''}{c'}\right)+2 e^{2\omega} (1-c) =0.$$ We now plug into this equation the values for $b,b'$ and $c$ in terms of $\alfa,\beta$ via \eqref{relcon3}, those of $c',c''$ by means of \eqref{relcon4}, and the value of $\omega_x$ in \eqref{omex}. Putting all together, we obtain $$a''+\frac{2\alfa'}{\alfa} \,a' +4 \mathfrak{q}=0,$$ which integrates to 
\begin{equation}\label{relcon5}
a'(x)=-\frac{4\mathfrak{q}}{\alfa(x)^2} \left( \hat{c}_0+\int_0^x \alfa(s)^2 ds\right).
\end{equation}
Here $\hat{c}_0$ is an integration constant.

Moreover, once here we can use the relations \eqref{relcon3}, \eqref{relcon4} and \eqref{relcon5} to compute the value of $\sigma(x,y)$ from the compatibility condition \eqref{compati}. We obtain 
\begin{equation}\label{ecusi}
\sigma(x,y)= \frac{2}{\kappa_1}\left(\alfa(x) e^{\omega(x,y)} + \hat{c}_0+ \int_0^x \alfa(s)^2 ds\right).
\end{equation} 
We note for later use that, by \eqref{ecusi}, if $\alfa(0)=0$, then 
\begin{equation}\label{hatcero}
\hat{c}_0=-\kappa_1 |g(0)|^2/4.
\end{equation}

\subsection{Moduli space}
The previous construction implies that the space of solutions to $\Delta u+2=0$ foliated by capillary curves that lie in rotational paraboloids with a common axis is finite dimensional. Indeed, first note that $\alfa(x),\beta(x)$ in \eqref{system} are determined by $5$ free parameters, namely, the four initial conditions at some $x_0$, and the value of $\delta$. Also, by \eqref{omex} and \eqref{roy}, $\alfa(x),\beta(x)$ determine uniquely the harmonic function $\omega(x,y)$ up to fixing the value of $\omega$ at a point $(x_0,y_0)$. Besides, the holomorphic map $g$ is linked to $\omega$ by $e^{2\omega}=|g'|^2$; thus, $g$ is uniquely determined from $\omega$ up to the change $g\mapsto e^{i\theta}g+g_0$, where $g_0\in \C$ and $\theta\in \R$. So, the space of holomorphic maps $g(x+iy)$ for which all $y$-curves are capillary is of dimension $9$. But on the other hand, the previous change for $g(z)$ amounts to translations and rotations of $\R^2$, and the additive constant for $\omega$ represents dilations of $g(z)$, see Remark \ref{sicha}. Finally, we recall that the conformal $(x,y)$ parameters are defined up to translations $(x,y)\mapsto (x+x_0,y+y_0)$ and dilations $(x,y)\mapsto (\landa x,\landa y)$, $\landa>0$. This leaves us, finally, with a $2$-dimensional moduli space, modulo similarities in $\R^2$, of solutions to $\Delta u+2=0$ foliated by capillary curves.

%
%

 Let us describe in more detail this space when $g:\cU\flecha \R^2$ parametrizes a Serrin ring domain $\overline\Omega=g(\cU)$ foliated by capillary curves, with $\cU$ as in \eqref{quoban}, and $\omega_y\not\equiv 0$. The boundary curves of $\parc \cU$ are thus given by $x=s_j$, $j=1,2$, and we will assume for simplicity that $\alfa(s_j)\neq 0$. So, each boundary curve $y\mapsto (g(s_j+iy),v(s_j,y))\in \R^3$ of the graph of $u$ lies in the rotational ``paraboloid'' $\cP(s_j)$, see (ii) at the beginning of Section \ref{sec:local}. But also, due the Serrin boundary condition, it lies in a horizontal plane $x_3=a_j$. Then, by the capillarity condition, either $\Omega$ is a radial annulus, or we must have $c(s_j)=0$, $j=1,2$, i.e., both $\cP(s_j)$ are horizontal planes.
 
We prove next that, in the above conditions, $\alfa:[s_1,s_2]\flecha \R$ must vanish somewhere. Indeed, observe that $c(x)$ in \eqref{relcon3} is monotonic, due to \eqref{khami}, and it is finite everywhere if $\alfa(x)$ does not vanish. In particular, if $\alfa(x)$ has no zeros, then $c(s_j)\neq 0$ for some $j\in\{1,2\}$, a contradiction.

This means that we can fix without loss of generality the initial condition $\alfa(0)=0$ in system \eqref{system}. By Lemma \ref{lem:curvatura}, this implies that $g(iy)$ is a regular parametrization of a closed curve of constant curvature $\beta(0)/2$. Thus, $\beta(0)\neq 0$, and after a dilation of $\Omega$ we can use the change explained in Remark \ref{sicha} and fix $\beta(0)=2$. Then, by \eqref{roy}, we have
\begin{equation}\label{roy2}
Z_y^2 =-\alfa'(0) Z^3 - \delta Z^2 + \beta'(0) Z - 1,
\end{equation}
where $Z(y):=e^{-\omega(0,y)}$. 
The only way that \eqref{roy2} can have periodic positive solutions $Z(y)$ is that the polynomial on the right side of \eqref{roy} has two positive, maybe coinciding, roots $0<\rho_1\leq \rho_2$, and in that case $Z(y)$ would takes its values in $[\rho_1,\rho_2]$. If $\rho_1=\rho_2$, then $Z(y)$ in \eqref{roy2} would be constant, of value $\rho_1(=\rho_2)$.


%

\section{Developing maps for foliations by capillary curves}\label{sec:gaussmaps}
In this section we construct a certain family of holomorphic maps $g(z)$ compatible with the foliation structure by capillary curves described in Section \ref{sec:folicapi}. In later sections we will show that some of these holomorphic maps are developing maps for Serrin ring domains in $\R^2$.


\subsection{Choice of parameters}\label{sec:param}
Taking into consideration the discussion at the end of Section \ref{sec:folicapi}, we define the parameter space 
\begin{equation}\label{def:theta0}
\Lambda:=\{(\eta,\tau): \eta>0, \ \tau\in (0,1]\},
\end{equation} 
and make the following choices for the initial data of \eqref{system}: 
\begin{equation}\label{abd}
\begin{array}{c}
\alfa(0)=0, \hspace{0.5cm} \beta(0)=2 ,  \hspace{0.5cm} \delta=\dfrac{1}{4} \left(\eta ^2 \tau ^2-\tau ^2-1\right),\\[15pt]
\alfa'(0)=\dfrac{\eta \tau^2}{8}, \hspace{0.5cm} \beta'(0)= \dfrac{\left(\eta ^2 \tau ^2+\eta ^2-1\right)}{2 \eta }, \hspace{0.3cm} 
 \end{array}
\end{equation}
In this way, the polynomial  in the right side of \eqref{roy2} is $p(Z)/4$, where
\begin{equation}\label{depe} 
p(s):=-\frac{\eta \tau^2}{2}\left(s-\frac{2}{\eta \tau^2}\right)\left(s -\frac{2}{\eta}\right) \big(s+2\eta \big).
\end{equation} 
Note that $p(s)$ has two positive roots, negative leading term, and the condition $\tau=1$ detects the coincidence of the two positive roots. So, we have the type of algebraic structure we were seeking.

Then, with the above notations, we have:

\begin{lemma}\label{lem:xaxb}
If $\tau\in (0,1)$, there exist unique $x_a^-<x_b^-<0<x_b^+<x_a^+$ such that: 
\begin{enumerate}
\item[i)] $\alpha(x_a^-)=\alpha(x_a^+)=0$, 
\item[ii)]  $\alpha<0$ in $ (x_a^-,0)$ and $\alpha>0$ in $ (0,x_a^+)$ , 
\item[iii)] $\beta(x_b^-)=\beta(x_b^+)=0$, 
\item[iv)]  $\beta<0$ in $ (x_a^-,x_b^-)\cup (x_b^+,x_a^+)$ and $\beta>0$ in $ (x_b^-,x_b^+)$. 
\end{enumerate}
If $\tau=1$, then $ii)$ and $iv)$ hold for $x_a^-=x_b^-=-\8$ and $x_b^+>0$, $x_a^+=\8$. That is, $\alpha(x)$ does not vanish for $x\neq 0$ and $\beta(x)$ has a unique zero at $x=x_b^+>0$.
\end{lemma}

\begin{proof} 
Consider the function $s(x)=\alpha(x)\beta(x)$. Note that $s(0)=0$ and $s'(0)=2\alpha'(0)>0$. By \eqref{system} we have 
$$s''(x)=-6s^2+4\delta s +2\kappa_2,$$
where $\kappa_2$ is the constant given by \eqref{khami2}.
Integrating the above equation, and using the constant $\kappa_1=2\alfa'(0)=s'(0)>0$ in \eqref{khami} we obtain 
$$s'(x)^2 = 4 \hat{p}(s), \hspace{0.5cm} \hat{p}(s):=-s^3 + \delta s^2+\kappa_2 s + \frac{\kappa_1^2}{4}.$$
By \eqref{abd} we see that $\hat p(s)$  has three real roots, at $$\{-1/4,-\tau^2/4,\eta^2\tau^2/4\}.$$ 

Let us assume first that $\tau\in (0,1)$. In this case, $\hat p$ has three simple roots, and $s(x)$ is a periodic function taking values in $[-\tau^2/4,\eta^2\tau^2/4]$. 

Let $x_0>0$ be the first positive zero of $s(x)$. Then, $x_0$ is the first positive zero for $\alpha(x)$ or $\beta(x)$. Assume that $\alpha(x_0)=0$. Using the constant $\kappa_1>0$ in \eqref{khami} we infer that $\beta(x_0)<0$. Since $\beta(0)>0$, this contradicts the fact that $x_0$ is the first zero for $s(x)$. Thus, $x_b^+:=x_0$ is the first positive zero of $\beta$. 

Let now $x_1>x_0$ be the second positive zero of $s(x)$. If $\beta(x_1)=0$, then we would have $\beta'(x_1)>0$ and $\alpha(x_1)>0$, contradicting that $\kappa_1>0$. Thus, $x_a^+:=x_1$ is the first positive zero of $\alpha$. 
Similarly, we deduce the existence of $x_a^- < x_b^-<0$ satisfying the desired conditions. 

Finally, assume $\tau=1$. In this case, $\hat p$ has a double root at $-1/4$ and a single root at $\eta^2/4$, and therefore $s(x)$ takes values in $(-1/4,\eta^2/4]$, reaching the value $-1/4$ in infinite time. In particular, $s(x)$ only vanishes at $x=0$ and at some other $x_0>0$ with $s'(x_0)<0$, that is attained after $s(x)$ \emph{hits} the boundary value $s=\eta^2/4$ (recall that $s'(0)>0$) and starts to decrease. Using again $\kappa_1$, we conclude that $\alpha$ cannot vanish at $x_0$, and so $x_b^+:=x_0$ is the unique zero of $\beta$, and the rest of the stated properties for $\tau=1$ hold.
\end{proof}

\subsection{Construction of the harmonic function $\omega$}\label{sec:omega}

Take $(\eta,\tau)\in \Lambda$ and let $(\alfa(x),\beta(x))$ be the unique solution to system \eqref{system} with the initial conditions \eqref{abd}.

Consider now the unique solution $z(y):\R\flecha \R$ to the differential equation 
\begin{equation}\label{edoz}
8 z''(y) = p'(z(y))
\end{equation} 
with initial conditions $z(0)=2/\eta $, $z'(0)=0$, where $p(z)$ is the polynomial in \eqref{depe}. 
If $\tau=1$, then $z=2/\eta$ is a double root of $p(z)$ and therefore $z(y)\equiv 2/\eta$ is constant. If $\tau\in (0,1)$, then $z(y)$ is non constant and it satisfies 
\begin{equation}\label{eq:zy} 
\left\{ 
\begin{array}{rll}
4z'(y)^2 &=& p(z(y)),\\
z(0) &=& 2/\eta. 
\end{array}
\right. 
\end{equation}
It is clear from \eqref{eq:zy} that $z(y)$ is a periodic, real analytic function that takes values in the interval $[2/\eta, 2/(\eta\tau^2)]$, and satisfies $z(-y)=z(y)$.
 
Now for each $y\in\R$, let $x\mapsto \omega(x,y)$ be the unique solution to the differential equation in $x$
\begin{equation}\label{rox}
2\omega_x = - \alpha(x) e^{-\omega} - \beta(x) e^{\omega} 
\end{equation}
with initial condition 
\begin{equation}\label{iniconro}
\omega(0,y) = -\log z(y).
\end{equation} 
This process creates a real analytic function $\omega(x,y)$. We note that if $\tau = 1$, then $\omega = \omega(x)$ depends only on $x$, as $z(y)$ is constant. Recalling the notations of Lemma \ref{lem:xaxb} we have:
\begin{lemma}\label{lem:lio} 
For any $(\eta,\tau)\in\Lambda$, the function $\omega(x,y)$ is harmonic, and well defined in $(x_b^-,x_a^+)\times\R$. 
Moreover, $\omega(x,y)$ cannot be extended to $[x_b^-,x_a^+]\times\R$.
\end{lemma}
\begin{proof}
Let us see first that $\omega(x,y)$ satisfies \eqref{roy}. Let $Z(x,y):=e^{-\omega(x,y)}$ and define 
\begin{equation}\label{eq:roypol}
Q(x,Z) = - \alfa(x)^2 Z^{4}  -4 \alpha'(x) Z^{3} +(6 \alfa(x) \beta(x) -4 \delta)Z^{2}  +4 \beta'(x) Z -\beta(x)^2. 
\end{equation} 
Then, \eqref{roy} can be rewritten as  
\begin{equation}\label{eq:Zy}
4Z_y^2=Q.
\end{equation}
By \eqref{rox}, $Z(x,y)$ satisfies 
\begin{equation}\label{eq:Zx}
2Z_x=\alpha Z^2 + \beta,
\end{equation}
and therefore, $(4Z_y^2)_x = 2\alpha Z (4Z_y^2).$
On the other hand, direct computations for \eqref{eq:roypol} using \eqref{system} and \eqref{eq:Zx} show that 
\begin{equation}\label{ec:au1}
Q_x = 2\alpha Z Q.
\end{equation}
Thus, there exists a function $F=F(y)$ such that 
$$4Z_y^2=F(y)Q.$$
Evaluating at $x=0$ and using \eqref{abd} and \eqref{eq:zy},  we obtain $$4Z_y(0,y)^2=p(Z(0,y))=Q(0,Z(0,y)).$$
If $\tau = 1$, we have that $Z_y\equiv 0$, so in particular $Q(0,y) \equiv 0$. By \eqref{ec:au1}, $Q(x,Z) \equiv 0$, so \eqref{eq:Zy} holds trivially. If $\tau \in (0,1)$, then $Z_y(0,y) \not \equiv 0$, so we deduce that $F(y)\equiv 1$. This proves that \eqref{eq:Zy}, and thus \eqref{roy}, holds. 

Let us now define
\begin{equation}\label{eq:maQ}
\widetilde{Q}(x,Z):= -4\alpha(x)^2Z^3 - 12\alpha'(x)Z^2 + 2(6\alpha(x)\beta(x) - 4\delta)Z + 4\beta'(x).
\end{equation}
It follows from \eqref{system} and \eqref{eq:Zx} that $\widetilde{Q}$ satisfies
$$\widetilde{Q}_x = 2\alpha Q + \alpha Z \widetilde{Q}.$$
In particular, if $\tau = 1$, we have that $Q \equiv 0$ and $\widetilde{Q}(0,Z(0,y)) = \widetilde{Q}(0,2/\eta) \equiv 0$, so $\widetilde{Q}(x,Z) \equiv 0$.

We will now prove that 
\begin{equation}\label{eq:armonicidadZ}
Z \Delta Z = Z_x^2 + Z_y^2.
\end{equation}
Since $Z=e^{-\omega}$, the above equation is equivalent to the harmonicity of $\omega$. If $\tau \in (0,1)$, we differentiate in \eqref{eq:Zx} with respect to $x$ to obtain an expression for $Z_{xx}$, and in \eqref{eq:Zy} with respect to $y$ to obtain an expression for $Z_{yy}$, both in polynomial terms with respect to $Z$. We compute on the other hand the expressions for $Z_x^2$ and $Z_y^2$ in \eqref{eq:Zx} and \eqref{eq:Zy} respectively. Comparing the resulting terms, we arrive at \eqref{eq:armonicidadZ}. In the case $\tau = 1$, we have that $Z_y = Z_{yy} \equiv 0$. By \eqref{eq:Zx}, we deduce that
$$ZZ_{xx} - Z_x^2 = -\frac{Z}{8}\widetilde Q + \frac{1}{4}Q = 0,$$
showing \eqref{eq:armonicidadZ} in this case.

Next, we describe the maximal domain of $\omega(x,y)$.  We first assume that $\tau\in (0,1)$.

Recall that $Z(0,y)=z(y)$ takes values in $[2/\eta,2/(\eta\tau^2)]$ and is defined for all $y$. In this case, $Q(0,Z)$ in \eqref{eq:roypol} has degree three and negative leading term, since $\alfa(0)=0$ and $\alfa'(0)>0$. For $x_0$ small enough, the polynomial $Q(x_0,Z)$ has then degree four, with two positive roots $\rho_1(x_0),\rho_2(x_0)$ and two negative roots. 
By \eqref{eq:Zy}, for $x_0$ small enough $Z(x_0,y)$ takes values in $[\rho_1(x_0),\rho_2(x_0)]$, and is defined for all $y$. This means that $\omega(x,y)$ is well defined on the strip $(-x_0,x_0)\times \R$.

This situation is maintained as long as the positive roots $\rho_1(x_0),\rho_2(x_0)$ remain that way, that is, until one of the next three situations happen:
\begin{enumerate}
\item[(1)]
$\rho_1(x_0)=\rho_2(x_0)$.
\item[(2)]
$\rho_1(x_0)=0$.
\item[(3)]
 $\rho_2(x_0)=\8$.
\end{enumerate}

Case (1) cannot happen. Indeed, should (1) hold for some $x_0$, the function $Z(x_0,y)$ would be constant, and so, by the uniqueness of solutions to \eqref{rox} with respect to the initial conditions, this would imply that $Z(x,y)$ does not depend on $y$. This is a contradiction, as we know that $z(y):=Z(0,y)$ is not constant (recall that $\tau\neq 1$).

For the remaining two cases (2), (3) above, notice that, by \eqref{eq:roypol}, $\rho_1(x_0)=0$ implies $\beta(x_0)=0$ and $\rho_2(x_0)=\8$ implies $\alpha(x_0)=0$. Thus, by Lemma \ref{lem:xaxb},  $\omega(x,y)$ is defined in $(x_b^-,x_a^+)\times \R$.

We next prove that $\omega(x,y)$ cannot be extended to a larger band than $(x_b^-,x_a^+)\times \R$. For this, we will use that $\omega(x,y)=\omega(x,-y)$, which is proved in more generality in Lemma \ref{lem:sigma} below. This implies that $\omega_y(x,0)=0$, and so, $Z(x,0)=\rho_1(x)$.

If $x=x_b^-$, by \eqref{eq:roypol}, the polynomial $Q(x_b^-,Z)$ has a simple root at $Z=0$ with positive derivative, since $\beta'(x_b^-)>0$. This implies that $\rho_1(x)\to 0$ as $x\to x_b^-$. Since $e^{-\omega(x,0)}=Z(x,0)=\rho_1(x)$, this shows that $\omega$ cannot be extended to $x=x_b^-$.

If $x=x_a^+$, then $Q(x_a^+,Z)$ has degree three, positive leading coefficient, two negative roots and one positive root at $\rho_1(x_a^+)$. In that situation, $y\mapsto Z(x_a^+,y)$ takes its values in $[\rho_1(x_a^+),\8)$ and, by \eqref{eq:Zy}, there is some $y_0>0$ such that $Z(x_a^+,y)\to \8$ as $y\to y_0^-$. This again shows that $\omega$ cannot be extended to $x=x_a^+$ and concludes the proof in the case $\tau\in (0,1)$.

We finally consider the case $\tau=1$. For later use, we state the argument as a remark.
\begin{remark}\label{re:omegatau1}
If $\tau=1$, then $z(y)$ in \eqref{edoz} is constant. Then, by \eqref{rox} and \eqref{iniconro}, we obtain that $\omega$ depends only on $x$. Since $\Delta \omega=0$, then $\omega(x)$ is linear.
\end{remark}
Recalling that when $\tau=1$ we have $x_b^-=-\8$ and $x_a^+=\8$ (see Lemma \ref{lem:xaxb}), this concludes the proof of Lemma \ref{lem:lio}.
\end{proof}

The next lemma summarizes the symmetries of $\omega(x,y)$: 

\begin{lemma}\label{lem:sigma}
Assume that $\tau \in (0,1)$. Then, there exists $\vart=\vart(\tau,\eta)>0$ such that $\omega(x,y)$   satisfies 
\begin{equation}\label{simroo}
 \omega(x,k\vart+y)=\omega(x, k\vart - y)
\end{equation}
for every $k\in\Z$.  In particular, $\omega(x,y+2\vart)=\omega(x,y)$. 
Moreover, $\vart$ is the first positive value for which $e^{-\omega(0,\vart)}=2/(\eta\tau^2)$, and can be computed as 
\begin{equation}\label{eq:sigma}
\vart = \int_{2/\eta }^{2/(\eta \tau^2)} \frac{2}{\sqrt{p(z)}} dz
\end{equation}
where $p(z)$ is the polynomial defined in \eqref{depe}.  
\end{lemma}
\begin{proof} 
Let us assume first that $x=0$, and so $e^{-\omega(0,y)}=z(y)$, where $z(y)$ is the solution to \eqref{eq:zy}. By \eqref{eq:zy}, there exists $\vart>0$ such that $z(y)$ is strictly increasing from $[0,\vart]$ to $[2/\eta, 2/(\eta\tau^2)]$, and $z(y):\R\flecha [2/\eta, 2/(\eta\tau^2)]$ is $2\vart$-periodic, with $z(k\vart +y)=z(k\vart-y)$. Then, by \eqref{eq:zy}, we obtain \eqref{eq:sigma} and by \eqref{iniconro} we have
\begin{equation}\label{oecuu}
\omega(0,k\vart+y)=\omega(0, k\vart - y).
\end{equation}
Observe also that $z'(k\vart)=0$ for any $k\in\Z$. 

Consider now the functions $f_1(x,y):=\omega(x,k\vart-y)$ and $f_2(x,y):=\omega(x,k\vart+y)$, which are harmonic functions that coincide along $y=0$. We will next prove that they also have the same $y$-derivatives along $y=0$, what shows that they coincide, and so \eqref{simroo} holds. First, differentiating \eqref{rox} we have $$2\omega_{xy} = (\alfa(x) e^{-\omega} - \beta(x) e^{\omega} ) \omega_y.$$ Viewing this expression as an ODE along the horizontal line $(x,k\vart)$, and using that $\omega_y(0,k\vart)=-z'(k\vart)/z(k\vart)=0$, we deduce that $\omega_y(x,k\vart)=0$. This implies that $(f_1)_y=(f_2)_y$ along $y=0$ and finishes the proof. 
\end{proof}
\begin{remark}\label{re:sigma1}
The change of variable 
\begin{equation}\label{chava}
z=\frac{1}{h(t)}, \hspace{0.5cm} h(t):=\frac{\eta\left((\tau^2-1)t+1\right)}{2},
\end{equation} 
in the integral expression in \eqref{eq:sigma} allows to express $\vart$ in \eqref{eq:sigma} as 
\begin{equation}\label{eq:cdvvart}
\vart =\vart(\eta,\tau) = \int_0^1 \frac{2 dt}{ \sqrt{ t \left(1-t\right) \left( 1+t(\tau^2-1) \right) \left(1+\eta^2 ( 1+t(\tau^2-1) ) \right)    }  }.
\end{equation}
Thus, $\vart(\eta,\tau)$ extends analytically to $\tau=1$ with 
\begin{equation}\label{eq:sigma1}
\vart(\eta,1)=  \dfrac{2\pi}{\sqrt{1+\eta^2}}.
\end{equation}
Moreover, as $\omega$ depends only on $x$ in this case (see Remark \ref{re:omegatau1}), then \eqref{simroo} trivially holds for $\tau=1$ and $\vart$ as in \eqref{eq:sigma1}. 

\end{remark}

\subsection{Symmetries of $g(z)$}\label{sec:g}

From the harmonic function $\omega:(x_b^-,x_a^+)\times \R\flecha \R$, we can construct a holomorphic map $g(x+iy):(x_b^-,x_a^+)\times \R\flecha \C$ given by $e^{2\omega}=|g'|^2$. We remark that $g(z)$ is determined by the above process from the constants $(\eta,\tau)\in \Lambda$, but only up to translations $g(z)+{\bf v}$ and rotations $e^{i\theta}g(z)$, i.e., up to Euclidean isometries. In particular, it is unique once we fix $g(0)$ and $g'(0)$. We will chose for convenience the initial conditions 
\begin{equation}\label{inig}
g(0)=-1, \hspace{0.5cm} g'(0)=\frac{\eta}{2}.
\end{equation}
Note that $|g'(0)|=e^{\omega(0,0)}$, i.e., our choice of $g'(0)$ in \eqref{inig} is consistent with the initial value chosen in the construction $\omega(x,y)$ in our previous computations.

We next study the symmetries of $g(z)$.
\begin{lemma}\label{lem:sim}
Given $(\eta,\tau)\in \Lambda$ and $g(z)$ as above, then:
\begin{enumerate}
\item The curve $g(iy):\R\flecha \C$ infinitely covers the unit circle $\S^1$ with negative orientation.
\item Let  $\vart=\vart(\eta,\tau)>0$ be as in Lemma \ref{lem:sigma} if $\tau \in (0,1)$ and as in Remark \ref{re:sigma1} if $\tau=1$. Then, for any $k\in\Z$ the
length $\Theta=\Theta(\eta,\tau)$ of the arc of $g(iy)$ between $y=k\vart$ and $y=(k+1)\vart$ is independent from $k\in\Z$, and can be computed as  
\begin{equation}\label{eq:Theta}
\begin{array}{lll}
\Theta(\eta,\tau) &=& \displaystyle\int_{2/\eta}^{2/\eta\tau^2}  \dfrac{2}{z \sqrt{p(z)} }dz \quad \rm{if}\; \tau\in (0,1), \\
\\
\Theta(\eta,1)&=& \dfrac{\pi\eta}{\sqrt{1+\eta^2}} 
\end{array}
\end{equation}
where $p(z)$ is the polynomial in \eqref{depe}. Moreover, $\Theta\in (0,\pi)$, 
\item
For $k\in \Z$, let $L_k$ be the line passing through the origin and $g(ik\vart)\in \S^1$, and let $\Psi_k$ denote the symmetry in $\R^2$ with respect to $L_k$. Then, denoting $g(x,y)=g(x+iy)$, we have
\begin{equation}\label{sim:g}
g(x,k\vart-y)=\Psi_k (g(x,k\vart +y)).
\end{equation}
\end{enumerate}
\end{lemma}

\begin{proof}
Since $\alfa(0)=0$ and $\beta(0)=2$, it follows from Lemma \ref{lem:curvatura} that $\Gamma(y):=g(iy)$ covers a circle of radius one with negative orientation. The initial conditions \eqref{inig} imply that this circle is precisely $\S^1$. This proves item (i).

To prove item (ii), we compute the length of $\Gamma(y)$ from $y=k\vart$ to $y=(k+1)\vart$ as 
\begin{equation}\label{tein}
\Theta=\int_{k\vart}^{(k+1)\vart} |\Gamma'(y)| dy=  \int_{k\vart}^{(k+1)\vart} e^{\omega(0,y)}dy = \int_0^{\vart} e^{\omega(0,y)}dy,
\end{equation}
where the last equality comes from \eqref{simroo}. In particular, $\Theta$ does not depend on $k$.

Equation \eqref{eq:Theta} for $\tau=1$ is a consequence of \eqref{eq:sigma1} and the fact that $e^{\omega(0,y)} \equiv \eta/2$, as discussed in Section \ref{sec:omega}. For  $\tau\in (0,1)$, \eqref{eq:Theta} is obtained as a consequence of \eqref{eq:zy} after making in the integral of equation \eqref{tein} the change $z(y)=e^{-\omega(0,y)}$; one should recall here that we already discussed how $z(y)$ is a strictly increasing map from $[0,\vart]$ into $[2/\eta,2/(\eta\tau^2)]$.

Finally, to see that $\Theta\in (0,\pi)$, we consider the change of variable $z=1/h(t)$ of \eqref{chava} in the integral of  \eqref{eq:Theta}. This gives 
\begin{equation}\label{eq:cdv}
\Theta=\Theta(\eta,\tau)=\int_0^1 \frac{ \eta  \sqrt{\left(\tau ^2-1\right) t+1}}{ \sqrt{ 1+\eta^2\left(    (\tau^2-1)t + 1   \right)  } } 
\frac{dt}{\sqrt{t (1-t) } }.
\end{equation}
The integral of the second factor in the above product is equal to $\pi$, whereas straightforward computations show that the first factor lies in $(0,1)$. Therefore, $\Theta\in (0,\pi)$.  This completes the proof of item (ii).

To prove item (iii), we first note that, by the uniqueness of $g(z)$ from $\omega$ up to isometries of $\R^2$ and the symmetry equation \eqref{simroo} for $\omega$, we obtain directly
$$g(x,k\vart-y)=\Psi_k (g(x,k\vart+y)),$$ where $\Psi_k$ is the symmetry in $\R^2$ with respect to the line $L_k$ that passes through $g(0,k\vart)$ with tangent vector $g_x(0,k\vart)$. Since $g(iy)\in \S^1$ for every $y$, this line $L_k$ passes through the origin. This proves item (iii).
\end{proof}

\subsection{Periodicity of $g(z)$}\label{sec:per} 
Let $\Theta=\Theta(\eta,\tau)$ be given by \eqref{eq:Theta}, and define $\rm{Per}:\Lambda  \to (0,1)$ by
\begin{equation}\label{def:per}
\rm{Per}(\eta,\tau)=\dfrac{1}{\pi}\Theta(\eta,\tau).
\end{equation}
Note that ${\rm Per}(\eta,1):\R\flecha (0,1)$ is bijective, by \eqref{eq:Theta}. Also, from the expression for $\Theta$ in \eqref{eq:cdv} we have
\begin{equation}\label{eq:Per0}
\lim_{\tau\to 0} \rm{Per}(\eta,\tau) = \frac{2}{\pi}\arctan(\eta).
\end{equation}
The map $\rm{Per}(\eta,\tau)$ controls the periodicity of the holomorphic map $g(x,y)\equiv g(x+iy)$ in the $y$-variable. Indeed, assume that ${\rm Per}(\eta,\tau)=1/n$ for some $n\in\mathbb{N}$, $n\geq 2$, and denote $\Gamma(y):=g(iy)$. Then, as explained in the proof of Lemma \ref{lem:sim}, the restriction of $\Gamma(y)$ to any interval of the form $[k\vart,(k+1)\vart]$ describes an arc of the unit circle of length $\Theta= \pi/n$. Thus, the angle between two consecutive lines $L_k, L_{k+1}$ of item (iii) in Lemma \ref{lem:sim} is $\pi/n$, independent of $k$. This means that $g(0,y)=g(0,y+2n\vart)$, and so, since $g$ is holomorphic, 
%
\begin{equation}\label{eq:periodic}
g(x,y+2n\vart)=g(x,y),
\end{equation} 
i.e., $g(x,y)$ is defined on the quotient vertical strip $((x_b^-,x_a^+)\times \R,\sim)$, with $(x,y)\sim (x,y+2n\vart)$. 

Note that, if ${\rm Per}(\eta,\tau)=1/n$, the previous arguments show that set $\{L_k :k\in \Z\}$ of symmetry lines is just a finite set $\{L_0,\dots, L_{n-1}\}$, with $L_n=L_0$. Since $g(0)=-1$ (see \eqref{inig}), $L_0$ is the $x_2=0$ axis, and each $L_k$ is a rotation of $L_0$ of angle $k\pi/n$, for $k=1,\dots, n-1$.

We now study the behavior of the level curves of the map $\rm{Per}(\eta,\tau)$. 

\begin{lemma}\label{lem:period}
For each $q_0\in (0,1)$, the level set ${\rm Per}^{-1}(q_0)\subset \Lambda$ is a graph $\eta=\eta_{q_0}(\tau)$, where $\eta_{q_0}: (0,1]\flecha (0,\8)$ is a strictly decreasing analytic function. In particular, for each $n\in\N$, $n\geq 2$, there exists such a function
\begin{equation}\label{percurva}
\eta_n:(0,1]\to\R^+
\end{equation}
such that ${\rm Per}(\eta_n(\tau),\tau)=1/n$. This curve satisfies 
\begin{equation}\label{eq:etalimit}
\eta_n(1)= \frac{1}{\sqrt{n^2-1}}, 
\qquad 
\lim_{\tau\to 0} \eta_n(\tau) =\tan\left(\frac{\pi}{2n}\right).
\end{equation}
Moreover, for any $\tau\in (0,1)$ the holomorphic map $g(z)$ associated to $\eta=\eta_n(\tau)$ and $\tau$ satisfies \eqref{eq:periodic}, where $\vartheta=\vartheta(\eta_n(\tau),\tau)$ is given by \eqref{eq:sigma} (see also \eqref{eq:cdvvart}). 
\end{lemma}

\begin{proof}
By \eqref{eq:cdv}, the $\partial_\eta$ and $\partial_\tau$-derivatives of $\rm{Per}(\eta,\tau)$ are positive. So, by the implicit function theorem, the level sets of $\rm{Per}(\eta,\tau)$ can be written as graphs over the $\tau$ variable: $\eta=\eta(\tau)$. To check that these curves are decreasing it suffices to derive in the expression ${\rm Per}(\eta(\tau),\tau)=q_0$ to infer that $\eta'(\tau)<0$. 
The limit values of $\eta_n(\tau)$ at $\tau=1$ and $\tau\to 0$ are a consequence of \eqref{eq:Theta} and \eqref{eq:Per0} respectively, using that $\Theta=\pi/n$ Finally, \eqref{eq:periodic} is a consequence of Lemma \ref{lem:sim}, as explained in the paragraph before Lemma \ref{lem:period}.
\end{proof}


\subsection{Summary of the construction}\label{sec:summary}

Fix $n\in \N$, $n\geq 2$. Then, for any $\tau\in (0,1]$ we have constructed a holomorphic map 
\begin{equation}\label{getau}
g(z): \cA\flecha \C,
\end{equation} 
where $\cA=\cA_\tau$ is a quotient vertical strip of the form
\begin{equation}\label{ec:banda}
\cA=\left( (x_b^-,x_a^+)\times \R, \sim \right), \hspace{0.5cm} (x,y)\sim (x,y + 2 n\vart).
\end{equation}
Here $g(z)=g(z;\tau)$, $\vart=\vart(\tau)>0$ as well as $x_b^-=x_b^-(\tau)<0$ and $x_a^+=x_a^+(\tau)>0$ all depend analytically on $\tau$. The fact that $g(z)$ is well defined on $\cA$ comes from \eqref{eq:periodic}.

The holomorphic map $g$ satisfies $g'(z)\neq 0$ for any $z\in \cA$, and $g(iy)\in \S^1$ for every $y$. Also, $g(z)$ is symmetric with respect to $n$ equiangular lines passing through the origin, as expressed by \eqref{sim:g}. Moreover, the harmonic function $\omega:= \log |g'|$ on $\cA$ satisfies the overdetermined \emph{capillarity} foliation structure described by \eqref{rox} with respect to adequate real analytic functions $\alfa(x),\beta(x)$.

In this way, $g$ is locally injective, and its restriction to any closed quotient band of the form $\cU:= ([s_1,s_2]\times \R,\sim)$ defines an \emph{immersed} (i.e. possibly with self-intersections) bounded doubly connected domain $\Omega$ in $\R^2$, given by $\overline\Omega:=g(\cU)$. This domain $\Omega$ has dihedral symmetry $D_{n}$ of order $2n$ with respect to the origin. Moreover, $\Omega$ is foliated by capillary curves in the sense of Definition \ref{def:capring}.  


\subsection{The radial case ($\tau=1$)}\label{sec:tau1} 
We will now prove that the holomorphic map $g$ constructed above is rotationally symmetric when $\tau=1$. Fix $n\in\N$, $n\geq 2$, and $(\tau,\eta_1)=(1,\eta_n(1) ) \in\Lambda$.  Let us recall that $\eta_1=1/\sqrt{n^2-1}$, see Lemma \ref{lem:period}.

As explained in Remark \ref{re:omegatau1}, when $\tau=1$ the harmonic map $\omega(x,y)$ is linear and depends only on $x$. Also, it satisfies  $e^{\omega(0,y)}\equiv \eta_1/2$. By  \eqref{rox} and \eqref{abd}, we also have $\omega_x(0,y)\equiv -\eta_1/2$. Thus,  
\begin{equation}\label{eq:omega1}
\omega(x,y)=-c_1 x + \log(c_1),\qquad c_1:=\dfrac{\eta_1}{2} =\frac{1}{2\sqrt{n^2-1}}.
\end{equation}
Since $|g'|=e^\omega$, we see from the initial conditions \eqref{inig} that $g(z)=-e^{-c_1 z}$. In particular, $g$ is well defined and holomorphic on the quotient $(\C,\sim)$, with $(x,y)\sim (x,y+4 \pi \sqrt{n^2-1})$.  

Note that each curve $y\mapsto g(x_0+iy)$ covers (with negative orientation) the circle of radius $R=e^{-c_1 x_0}$ around the origin, and the restriction of $g(z)$ to any domain $([s_1,s_2]\times \R,\sim)$ in the above quotient is a bijection onto a radial annulus.




\subsection{The limit case $\tau = 0$}\label{sec:tau0} 

The construction of the map $g(z)$ can also be carried out for $\tau=0$, with some modifications, as we travel along the level curve ${\rm Per}(\eta,\tau)=1/n$. 
First, notice that the initial conditions in  \eqref{abd} give $\alpha(x)\to 0$ as $\tau\to 0$. For our purposes, it will then be useful to work instead with 
$$\tilde{\alfa}(x):=\frac{\alfa(x)}{\tau^2}.$$ 
By \eqref{system}, $(\tilde\alfa,\beta)$ satisfy 
\begin{equation}\label{systemtil}
\left\{ \def\arraystretch{1.5}\begin{array}{lll} \tilde\alfa'' & = &  \delta \tilde\alfa - 2\tau^2 \tilde\alfa^2 \beta, \\ \beta'' & = &  \delta \beta - 
2 \tau^2\tilde\alfa \beta^2, \end{array} \right.
 \end{equation}
with $\tilde\alfa(0)=0$ and $\tilde\alfa'(0)= \eta/8>0$, by \eqref{abd}.
Assume now that $\tau= 0$. In this case, \eqref{abd} gives $\delta= -1/4$, and from \eqref{systemtil} we have
\begin{equation}\label{eq:alfatilde}
 \tilde\alfa(x) = \frac{\eta_0}{4}\sin\left({x}/{2}\right), 
\quad \beta(x) = 2\cos(x/2)+c_2 \sin(x/2), 
\end{equation}
for $c_2:=\eta_0-\frac{1}{\eta_0}$ and $\eta_0=\tan(\frac{\pi}{2n})$ (see \eqref{eq:etalimit} in Lemma \ref{lem:period}). So, $c_2=-2\cot(\pi/n)$.

The functions $\omega$ and $g$ defined in Sections \ref{sec:omega} and \ref{sec:g} can also be explicitely computed. Indeed, the solution $\omega(x,y)$ to \eqref{rox} and \eqref{iniconro} is given in our case by 
$$e^{-\omega} =\frac{2}{\sin(\frac{\pi}{n})} \left( \cos \left(\frac{x}{2}-\frac{\pi}{n}\right)+\cosh \left(\frac{y}{2}\right)\right),$$ 
and the associated holomorphic map $g$ is  
\begin{equation}\label{gt0}
g(z)=-\cos\left( \frac{\pi}{n} \right) +\sin \left(\frac{\pi }{n}\right) \tan \left(\frac{z}{4}-\frac{\pi }{2 n}\right).
\end{equation}
Note that $g(z)$ is well defined in the domain 
$$ \mathcal{I}_1\times\R= \left(-2\pi + \frac{2\pi}{n},2\pi + \frac{2\pi}{n}\right) \times \R.$$
Also, observe that $g(z)$ in \eqref{gt0} is not periodic in the $y$-variable in this case (indeed, $\vart\to\8$  as $\tau\to 0$, see \eqref{eq:cdvvart}), and it satisfies 
\begin{equation}\label{eq:p0p1}
\def\arraystretch{1.8}\begin{array}{lll}
\displaystyle \lim_{y\to-\8} g(x+iy) &=& -\cos\left( \frac{\pi}{n} \right) - i\sin\left( \frac{\pi}{n}\right) =:{\bf p}_0 \\
\displaystyle \lim_{y\to\8} g(x+iy) &=&  -\cos\left(\frac{\pi}{n} \right) + i\sin\left( \frac{\pi}{n}\right)=:{\bf p}_1
\end{array}
\end{equation}
Moreover, for each $x\in \mathcal{I}_1$, the curve $y\mapsto g(x+iy)$ is a circle arc joining ${\bf p}_0$ and ${\bf p}_1$.  Following \eqref{curcapi}  the radius of the circle is $R(x)=1/|\kappa(x)|$, where 
\begin{equation}\label{curt0}
\kappa(x) = \frac{\sin\left(\frac{x}{2}-\frac{\pi}{n}\right)}{\sin(\frac{\pi}{n})}.
\end{equation}
Thus,  in the limit case $\tau=0$ we obtain a holomorphic bijection 
$$g:  \left(-2\pi + \frac{2\pi}{n},2\pi + \frac{2\pi}{n}\right) \times \R \to \C\setminus L$$
where $L=\{-\cos(\pi/n)+ i y \in\C : |y|> \sin(\pi/n)\}$.


\section{Immersed Serrin ring domains}\label{sec:immersed}

In this section we will show how the holomorphic maps $g(z)$ constructed in Section \ref{sec:gaussmaps} define a $2$-parameter analytic family of doubly connected \emph{immersed} (i.e. maybe self-intersecting) domains of Serrin type. This concept is made precise by the definition below.


\begin{definition}\label{imser}
Let $g(z):\cU\flecha \C\equiv \R^2$ be a holomorphic map from a vertical quotient strip $$\cU:=([s_1,s_2]\times \R, \sim), \hspace{0.5cm} (x,y)\sim (x,y+T),$$ and assume that $g'(z)\neq 0$ for any $z=x+iy\in \cU$. We say that $\overline\Omega:=g(\cU)\subset\R^2$ is an \emph{immersed Serrin ring domain} if there exists $v(x,y):\cU\flecha \R$ such that:
\begin{enumerate}
\item
$v$ is a solution to $\Delta v + 2e^{2\omega}=0$, where $\omega:= \log |g'|$.
\item
Along each boundary curve $x=s_j$, $j=1,2$, there exist $a_j,b_j\in \R$ such that $$v(s_j,y)=a_j, \hspace{0.5cm} v_x(s_j,y)=b_j\, e^{\omega(s_j,y)}$$ hold for every $y\in \R$.
\end{enumerate}
\end{definition}

\begin{remark}\label{cuandoserrin}
Assume that $g(z)$ is injective in Definition \ref{imser} above. Then, the exterior unit normal $\nu$ of $\parc \Omega$ is given by $\nu=-e^{-\omega}\parc_x$ (resp. by $\nu=e^{-\omega}\parc_x$) at the boundary curve $x=s_1$ (resp. $x=s_2$). From here, we easily see that $u:=v\circ g^{-1}$ solves \eqref{overeq00} for $a_1,a_2,-b_1,b_2$. 
Therefore, $\Omega:= {\rm int}(\overline\Omega)$ is a Serrin ring domain in $\R^2$; see the discussion at the beginning of Section \ref{sec:confpa}.
\end{remark}

Fix $n\in \N$, $n\geq 2$. In Section \ref{sec:gaussmaps} we constructed, for each $\tau \in (0,1]$,  a holomorphic map $g(z)$ defined on a vertical strip $(x_b^-,x_a^+)\times \R$, with $g'(z)\neq 0$.  By \eqref{eq:periodic}, this holomorphic map $g$ is periodic and defined on the quotient band $\cA$ in \eqref{ec:banda}; 
see the summary of these properties in Section \ref{sec:summary}. 

We are going to show that from these holomorphic maps arises a real analytic, $2$-parameter family of immersed Serrin ring domains. For that, we start by defining the parameter space
\begin{equation}\label{wimm}
\cW^n_{\rm imm} :=\{(s,\tau) : \tau \in (0,1), x_b^-(\tau)<s<0\}.
\end{equation}
For any $\tau\in (0,1]$, we let $(\alpha(x),\beta(x)):\R\to\R^2$ be the solution to system \eqref{system} with initial conditions \eqref{abd}; here, $(\eta,\tau)=(\eta,\eta_n(\tau))$, with $\eta_n(\tau)$ as in \eqref{percurva}. The zeros of $\alpha(x)$ and $\beta(x)$ were described in Lemma \ref{lem:lio}. We have then:

\begin{lemma}\label{lem:t}
For each $s\in (x_b^-,0)$ there exists a unique $s^*\in (x_b^+,x_a^+)$ such that 
\begin{equation}\label{eseses}
\frac{\beta(s)}{\alfa(s)}=\frac{\beta(s^*)}{\alfa(s^*)}.
\end{equation}
Moreover, the map $s\mapsto s^*$ is strictly increasing.
\end{lemma}
\begin{proof}
Denote 
\begin{equation}\label{defit}
t(x):=\frac{\beta(x)}{\alfa(x)}:(x_b^-,0)\cup(0,x_a^+)\flecha \R.
\end{equation} 
From \eqref{khami}, \eqref{abd} we have $t'(x)=-2\alfa'(0)/\alfa(x)^2=-\kappa_1/\alfa(x)^2<0$, and so $t(x)$ is strictly decreasing on each connected component of its domain. Also, by Lemma \ref{lem:lio}, $t(x_b^-)=t(x_b^+)=0$, and $t(x)\to -\8$ as $x\to 0^-$ or as $x\to x_a^+$. This proves the existence of $s^*$. The fact that the map $s\mapsto s^*$ is increasing follows directly from the monotonicity of $t(x)$.
\end{proof}

\begin{theorem}\label{th:imm}
For any $(s,\tau)\in \cW_{\rm imm}^n$, let
\begin{equation}\label{ubueno}
\cU=([s,s^*]\times\R,\sim)  
\end{equation}
where  $s^*>0$ is defined in Lemma \ref{lem:t} and $(x,y)\sim (x,y+2n\vart)$, as in \eqref{ec:banda}. Then,  $\overline\Omega_{(s,\tau)}:=g(\cU)$ is an immersed Serrin ring domain, as in Definition \ref{imser}.
\end{theorem}
\begin{proof}
To prove the theorem we will construct a solution $v(x,y)$ to \eqref{pdeom} in the quotient band \eqref{ubueno} 
that satisfies $$v(s,y)=a_1, \hspace{0.5cm} v_x(s,y)=b_1\, e^{\omega(s,y)}$$ and $$v(s^*,y)=a_2, \hspace{0.5cm} v_x(s^*,y)=b_2\, e^{\omega(s^*,y)}$$ for any $y\in \R$, with respect to adequate constants $a_j, b_j$.



For that, we will rely on our local computations of Section \ref{sec:folicapi} regarding solutions to $\Delta u+2=0$ foliated by capillary curves, and follow a reverse path. First, we denote 
\begin{equation}\label{eq:q}
\mathfrak{q}:=\frac{1}{2t(s)}=\frac{1}{2 t(s^*)}<0,
\end{equation}
with $t(x)$ as in \eqref{defit}, and define in terms of $\alfa(x),\beta(x)$ the functions $a(x),b(x),c(x)$ given by \eqref{relcon3} and \eqref{relcon5}, that is, 
\begin{equation}\label{abcbis}
b(x)=\frac{4 \mathfrak{q}}{\alfa(x)}, \hspace{0.5cm} c(x)=1-\frac{2\mathfrak{q}\beta(x)}{\alfa(x)}, \hspace{0.5cm} a'(x)=-\frac{4\mathfrak{q}}{\alfa(x)^2} \left( \hat{c}_0+\int_0^x \alfa(s)^2 ds\right).
\end{equation}
Here, $\hat{c}_0=-\kappa_1 |g(0)|^2/4$, see \eqref{hatcero}, which in our case is, by \eqref{abd}, \eqref{inig},
\begin{equation}\label{hace0}
\hat{c}_0 := -\frac{\eta \tau^2}{16} \hspace{0.5cm} \left(= -\frac{\alfa'(0)}{2}\right).
\end{equation} 
The function $a(x)$ is defined up to an additive constant that does not affect the arguments below. 
Note that $a(x),b(x),c(x)$ are only defined in $(x_b^-,x_a^+)\setminus\{0\}$, since $\alfa(x)\neq 0$ in that set. At $x=0$ all of them blow up (recall that $\alfa(0)=0$). Also, note that $c(s)=c(s^*)=0$, by \eqref{eq:q}.

In addition, we define a map $\sigma(x,y)$ in terms of $\alfa$ and $\omega$ as in \eqref{ecusi}, that is,
\begin{equation}\label{ecusi2}
\sigma(x,y)= \frac{2}{\kappa_1}\left(\alfa(x) e^{\omega(x,y)} + \hat{c}_0+ \int_0^x \alfa(s)^2 ds\right).
\end{equation} 
It is clear that $\sigma(x,y)$ is real analytic and well defined on the quotient strip $\cU$. From here, we define 
\begin{equation}\label{eq:v}
v(x,y):=a(x) +c(x)\sigma(x,y).
\end{equation}
Let us check that $v$ is a real analytic function on $\cU$. For that, it suffices to prove that $v(x,y)$ is analytic when $x=0$. Since $\alfa(x)\equiv x(\kappa_1/2 + \cdots)$, we see from \eqref{abcbis} that $$a(x)= \frac{16\mathfrak{q} \hat{c}_0}{\kappa_1^2} \, \frac{1}{x} + \text{analytic}(x).$$ A similar argument for $c(x)\sigma(x,y)$ using \eqref{abcbis} and \eqref{ecusi2} gives
$$c(x)\sigma(x,y)= \frac{-16\mathfrak{q} \hat{c}_0}{\kappa_1^2} \, \frac{1}{x} + \text{analytic}(x,y).$$ Thus, these singularities cancel out, and $v(x,y)$ is real analytic around $x=0$.

Also, the above definitions in \eqref{abcbis}, \eqref{ecusi2} make the compatibility condition \eqref{compati} hold, i.e.,
$$a'(x) + c'(x)\sigma =b(x) e^{\omega}.$$ Therefore, $v(x,y)$ is a solution to system \eqref{sobredet4}, that is,
\begin{equation}\label{dosev}
v= a(x) + c(x) \sigma, \hspace{0.5cm} v_x = b(x) e^{\omega} + c(x) \sigma_x.
\end{equation}

On the other hand, the function $\sigma(x,y)$ in \eqref{ecusi2} satisfies 
\begin{equation}\label{edpsi}
\Delta \sigma +2e^{2\omega}=0.
\end{equation} 
This is a rather long but direct computation that we omit. In it, when computing the Laplacian of $\sigma(x,y)$ in \eqref{ecusi2}, we need to take into account that $\omega$ is harmonic, that $\omega_x$ satisfies \eqref{rox} and $\omega_y^2$ satisfies \eqref{roy}, as well as the differential system \eqref{system} for $(\alfa,\beta)$.

We can now compute, using \eqref{dosev}, 
$$\Delta v= (v_x)_x+v_{yy} = b' e^{\omega} + b \omega_x e^{\omega} + c' \sigma_x + c\Delta \sigma.$$
Plugging now into this expression \eqref{rox}, \eqref{abcbis}, \eqref{ecusi2} and \eqref{edpsi}, we deduce that $$\Delta v+2e^{2\omega}=0.$$
Finally, since $c(s)=c(s^*)=0$, we see that the conditions in Definition \ref{imser} are fulfilled. Thus, $\overline\Omega:=g(\cU)$ is an immersed Serrin ring domain, what completes the proof.
%
%
%
\end{proof}

\begin{corollary}\label{re:symmetries}
Any of the immersed Serrin domains $\overline\Omega_{(s,\tau)}$ of Theorem \ref{th:imm} has a dihedral symmetry group $D_n$ of order $2n$ centered at the origin. 
\end{corollary}
\begin{proof}
We start by recalling the symmetries of $g(z)$ in Lemma \ref{lem:sim} and the fact that $\Theta(\eta,\tau)=\pi/n$ for $\eta=\eta_n(\tau)$. As discussed after \eqref{eq:periodic}, $g(z)$ is then invariant under the reflections with respect to $n$ lines $L_0,\ldots,L_{n-1}$, where $L_0$ is the horizontal axis in $\R^2$, and $L_k$  is the rotation of $L_0$ of angle $k \pi /n$ with respect to the origin, $k=1,\ldots,n-1$. Thus, $\overline\Omega:=\overline\Omega_{(s,\tau)}$ has these symmetries. In particular, since $\overline\Omega$ is not radial, the isometry group of $\overline\Omega$ is a dihedral group $D_{n'}$ of order $2n'$ for some $n'\geq n$, and contains the dihedral subgroup $D_n$ generated by all these previous reflections. We want to show that $n=n'$.

Arguing by contradiction, assume $n'>n$, and let $\gamma$ be, for example, the component of $\parc \Omega$ corresponding to $x=s$ in \eqref{ubueno}. Thus, there exists an additional symmetry line $L'$ for $\gamma$. At the intersection points of $\gamma$ and $L'$, the curvature $\kappa_\gamma$ of $\gamma$ should have a maximum or a minimum. Nonetheless, by \eqref{curcapi}, the critical points of $\kappa_{\gamma}$ are those of $y\mapsto \omega(s,y)$, and these placed exactly at the intersection points of $\gamma$ with the symmetry lines $L_0,\dots, L_{n-1}$. This means that $L'$ agrees with one of the lines $L_j$, which is a contradiction that yields $n=n'$.
\end{proof}

%
%
%
%

We next show that the ring domains $\overline\Omega_{(s,\tau)}$ of Theorem \ref{th:imm} are foliated by capillary curves lying in paraboloids, with the $x_3$-axis as their common rotational axis:
\begin{proposition}
For any $x_0\in (s,s^*)-\{0\}$, the regular space curve 
\begin{equation}\label{curpara}
y\mapsto (g(x_0+iy),v(x_0+iy))
\end{equation} 
associated to the immersed ring domain $\Omega_{(s,\tau)}$ lies in the rotational paraboloid of $\R^3$ with equation 
\begin{equation}\label{parabo}
x_3=a(x_0)-\frac{c(x_0)}{2} (x_1^2+x_2^2).
\end{equation}
\end{proposition}
\begin{proof}
By the first equation in \eqref{dosev}, to prove the result it suffices to show that $\sigma$ in \eqref{ecusi2} is given by $\sigma(x,y)=-|g(z)|^2/2$, with $z=x+iy$. We do this next.

Recall that $\sigma(x,y)$ solves \eqref{edpsi}. By a similar direct but long computation to the one leading to \eqref{edpsi}, that we also omit, one obtains that the Hopf differential of $\sigma(x,y)$ is zero, that is, $\sigma$ satisfies \eqref{sizz}.
In this way, as explained in our discussion below Remark \ref{rem:arriba}, $\sigma(x,y)$ must be of the form $$\sigma(x,y)= - \frac{|A g(z)+B|^2}{2},$$ for $A,B\in \C$, $A\neq 0$. We now use \eqref{rox} and the expression for $\alfa'(0)=\kappa_1/2$ in \eqref{abd} to obtain from \eqref{hace0} the following initial values for $\sigma$ in \eqref{ecusi2}: $$\sigma(0,0)=\frac{2\hat{c}_0}{\kappa_1}=-\frac{1}{2}, \hspace{0.5cm} \sigma_x(0,0)=\frac{\eta}{2}, \hspace{0.5cm }\sigma_y (0,0)=0.$$
Since $g(0)=-1$ and $g'(0)=\eta/2$, see \eqref{inig}, we have $\sigma(x,y)=-|g(z)|^2/2$, as claimed.
\end{proof}

\begin{remark}
Let $x_0=0$. Then, then the curve \eqref{curpara} lies in a cylinder over the unit circle, which appears as the limit of the paraboloids in \eqref{parabo} as $x_0\to 0$. So, strictly speaking, the curve $g(iy)$ is only a limit of capillary curves.
\end{remark}


\section{Embeddedness}\label{sec:embedded}

In this section we determine which among the immersed Serrin domain $\overline\Omega_{(s,\tau)}$ of Theorem \ref{th:imm} are actually embedded. That is, we describe the values of the parameters $(s,\tau)\in \cW^n_{\rm imm}$ for which $g(z)$ is a bijection from the conformal quotient strip $\cU$ of \eqref{ubueno} into $\overline\Omega_{(s,\tau)}$, and so $\Omega:={\rm int}(\overline\Omega_{(s,\tau)})$ defines a Serrin ring domain in $\R^2$; see Remark \ref{cuandoserrin}. We first study the embeddedness property for curves $y\mapsto g(x_0+iy)$, and then for the whole domains.

\subsection{Embeddedness of curves}\label{sec:embecur}

In what follows we fix $n\geq 2$ and $\tau\in (0,1)$, and use the notations in Sections \ref {sec:gaussmaps} and \ref{sec:immersed}.

Given $x_0\in (x_b^-,x_a^+)$, denote $\gamma(y):=g(x_0+iy)$, and $\S^1:=\R/\sim$, with $y\sim y+2n\vart$, where $\vart>0$ is defined in \eqref{eq:sigma}. In this way, $\gamma:\S^1\flecha \C$ is a closed regular real analytic curve. We study next its embeddedness, i.e., the injectivity of $\gamma$.

\begin{figure}[h]
  \centering
  \includegraphics[width=0.65\textwidth]{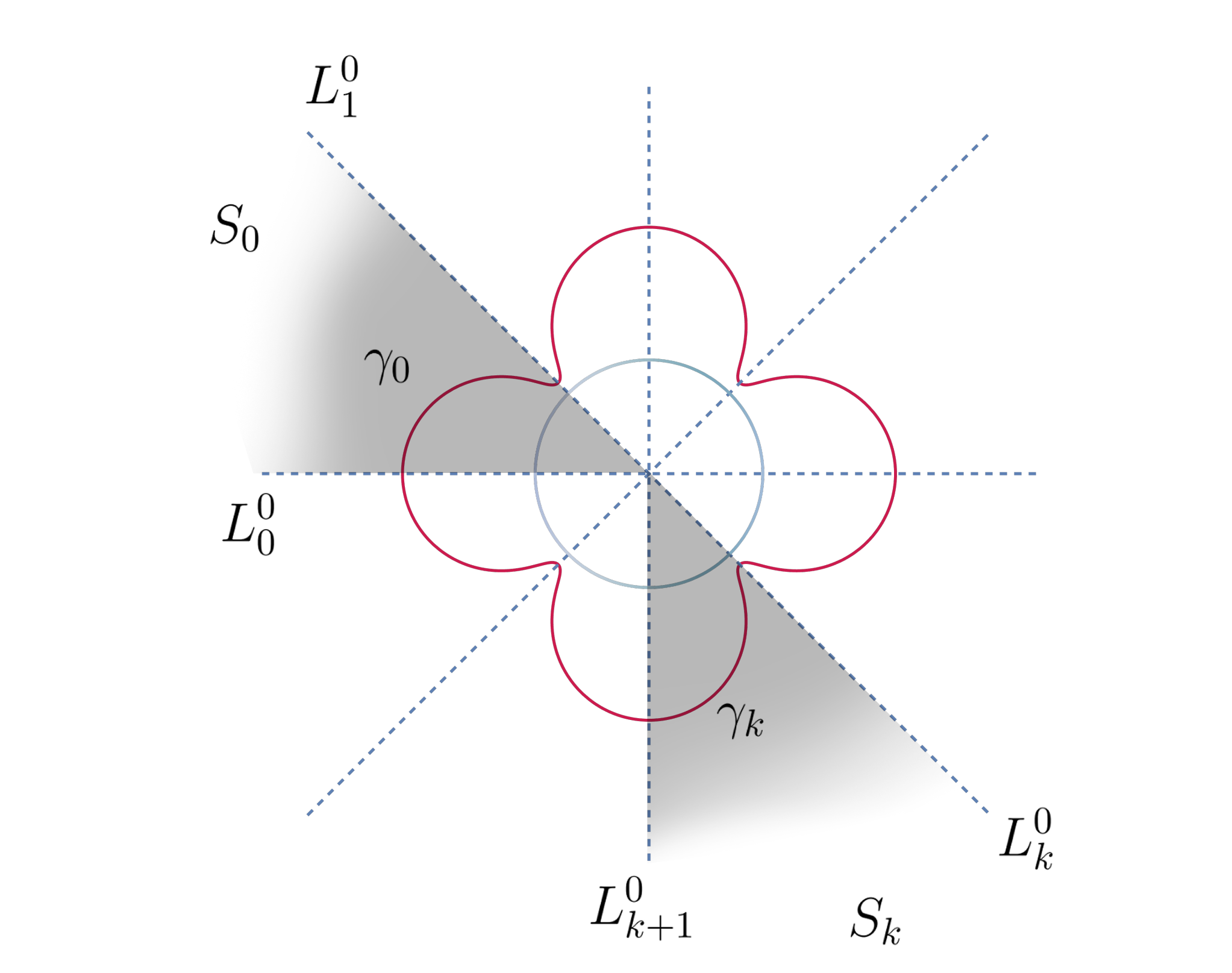} 
 \caption{The sectors $S_k$ and the arcs $\gamma_k$ of the closed curve $\gamma$. \label{sectores}}
\end{figure}

For that, we fix some notation. Let $L_0, \dots, L_{n-1},$ with $L_n=L_0$, $L_{n+1}=L_1$, etc., denote the symmetry lines of $g$ according to Lemma \ref{lem:sim}. We denote  by $L_k^0$ the half-line in $L_k$ starting at the origin and passing through $g(i k \vart)\in \S^1$. Since ${\rm Per}(\eta,\tau)=1/n$ and $g(0)=-1$, $L_k^0$ is given by ${\rm arg}(x_1+ix_2)=\pi- k\pi/n$; see our discussion below \eqref{eq:periodic}. We will also denote by $S_k$ the open sector in $\C$ with boundary $L_k^0\cup L_{k+1}^0$, and by $\gamma_k$ the open arc $\{\gamma(y): y\in (k\vart,(k+1)\vart)\}$. Note that $\gamma(k\vart)\in L_k^0$ for every $k$, and so the endpoints of $\gamma_k$ lie in $L_k^0$ and $L_{k+1}^0$, respectively.

In Lemma \ref{lem:embe1} below we use a description of $g(z)$ in terms of Weierstrass elliptic $\wp$-functions that is explained in detail in the appendix.

\begin{lemma}\label{lem:embe1}
Each arc $\gamma_k$ is embedded. Moreover, if $x_0\neq 0$, the distance function to the origin along each $\gamma_k$ is strictly monotonic. 

As a consequence, $\gamma:\S^1\flecha \C$ is embedded if and only if $\gamma_0\subset S_0$.
\end{lemma}
\begin{proof}
The critical points of $y\mapsto |\gamma(y)|^2=|g(x_0+iy)|^2$ are given by $\esiz \gamma',\gamma\esde=0$, that is, by 
\begin{equation}\label{disto}
{\rm Im} \left(\frac{g'(x_0+iy)}{g(x_0+iy)}\right) =0.
\end{equation} 
On the other hand, the function $g'(z)/g(z)$ admits a simple expression in terms of a Weierstrass elliptic function $\wp(z)$, see \eqref{fgp}, \eqref{cteswei} in the appendix. It follows from \eqref{fgp} that \eqref{disto} holds at $z_0=x_0+iy_0$ if and only if $\wp(z_0+\omega_1+\omega_2)\in \R$, where $\omega_1>0$ and $\omega_2\in i\R_+$ denote the half-periods of $\wp(z)$. By basic properties of Weierstrass functions, this happens if and only if $z_0$ lies in the lattice spanned by the half-periods $\omega_1,\omega_2$, i.e. either $x_0= m \omega_1$ or $iy_0=m \omega_2$ for some $m\in \Z$. Now, by the final remark (2) in the Appendix, we have that $x_0\in (x_b^-,x_a^+)\subset (-\omega_1,\omega_1)$. Also, since $g(z_0)\in \gamma_k$, we have $y_0\in (k\vart,(k+1)\vart)$. This means by the final remark (3) in the Appendix that $iy_0\in (k\omega_2,(k+1)\omega_2)$. Therefore, \eqref{disto} can only happen along $\gamma_k$ if $x_0
=0$, what proves the stated monotonicity property, as well as the embeddedness of $\gamma_k$ if $x_0\neq 0$. 

In the case $x_0=0$, we know that $g(iy)$ parametrizes the unit circle of $\C$, and each $\gamma_k$ corresponds to the segment $S_k\cap \S_1$. So, each $\gamma_k$ is also embedded in that case.

We finally consider the global embeddedness of $\gamma:\S^1\flecha \C$. If $\gamma_0\subset S_0$, then by symmetry $\gamma_k\subset S_k$ for all $k$, see \eqref{sim:g}. Since the $S_k$'s are disjoint, $\gamma$ is embedded. Conversely, assume that $\gamma_0$ is not contained in $S_0$. So, $\gamma(y_0)\in L_0^0\cup L_1^0$ for some $y_0\in (0,\vart)$. Then, again by the symmetry relation \eqref{sim:g} for $k=0$ or $k=1$, then either: (i) $\gamma(y_0)=\gamma(-y_0)$ if $\gamma(y_0)\in L_0^0$; or (ii) $\gamma(y_0)=\gamma(2\vart-y_0)$ if $\gamma(y_0)\in L_1^0$. In any case, $\gamma$ is not embedded.
\end{proof}
\begin{figure}[h]
  \centering
  \includegraphics[width=0.45\textwidth]{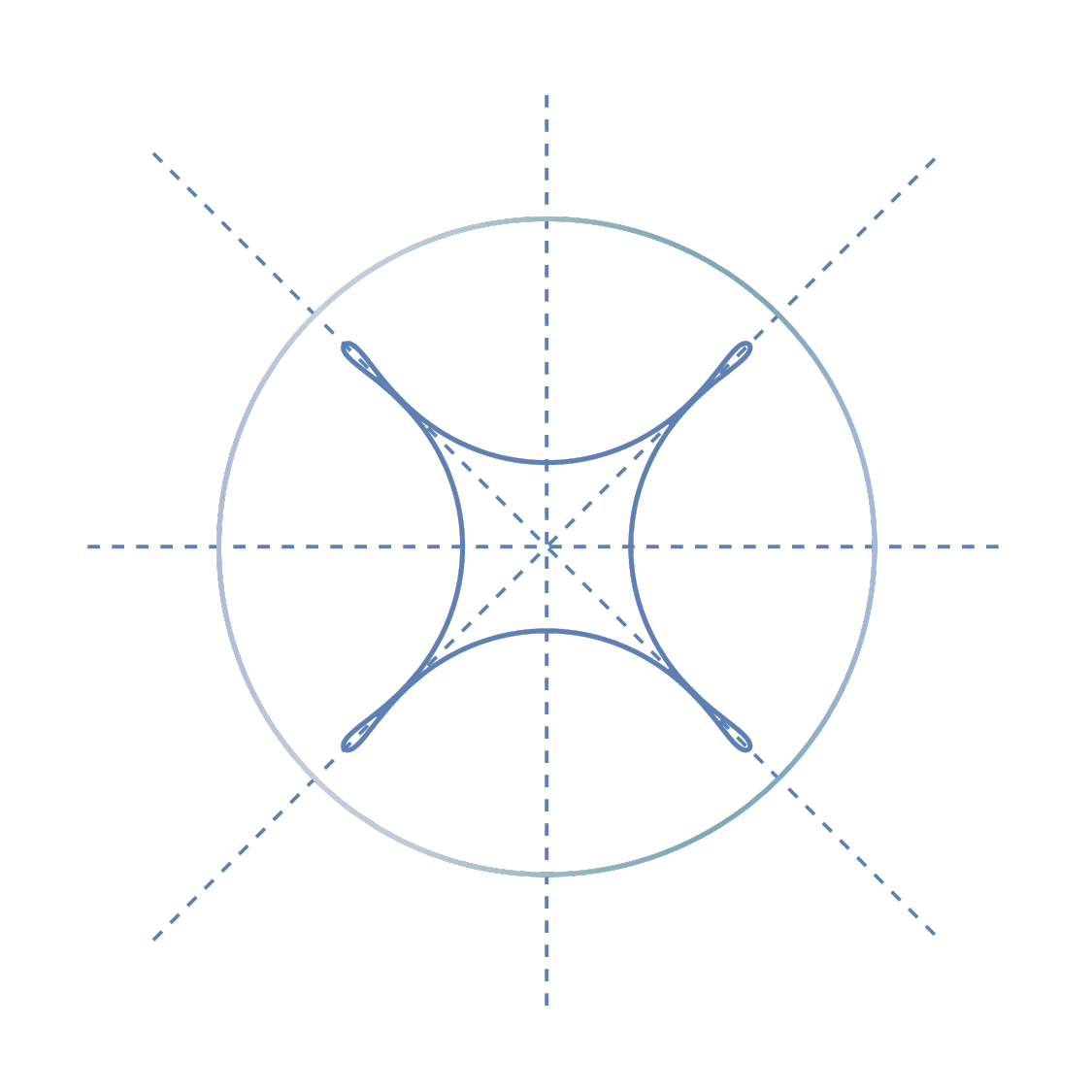}  \hspace{0.4cm}   \includegraphics[width=0.45\textwidth]{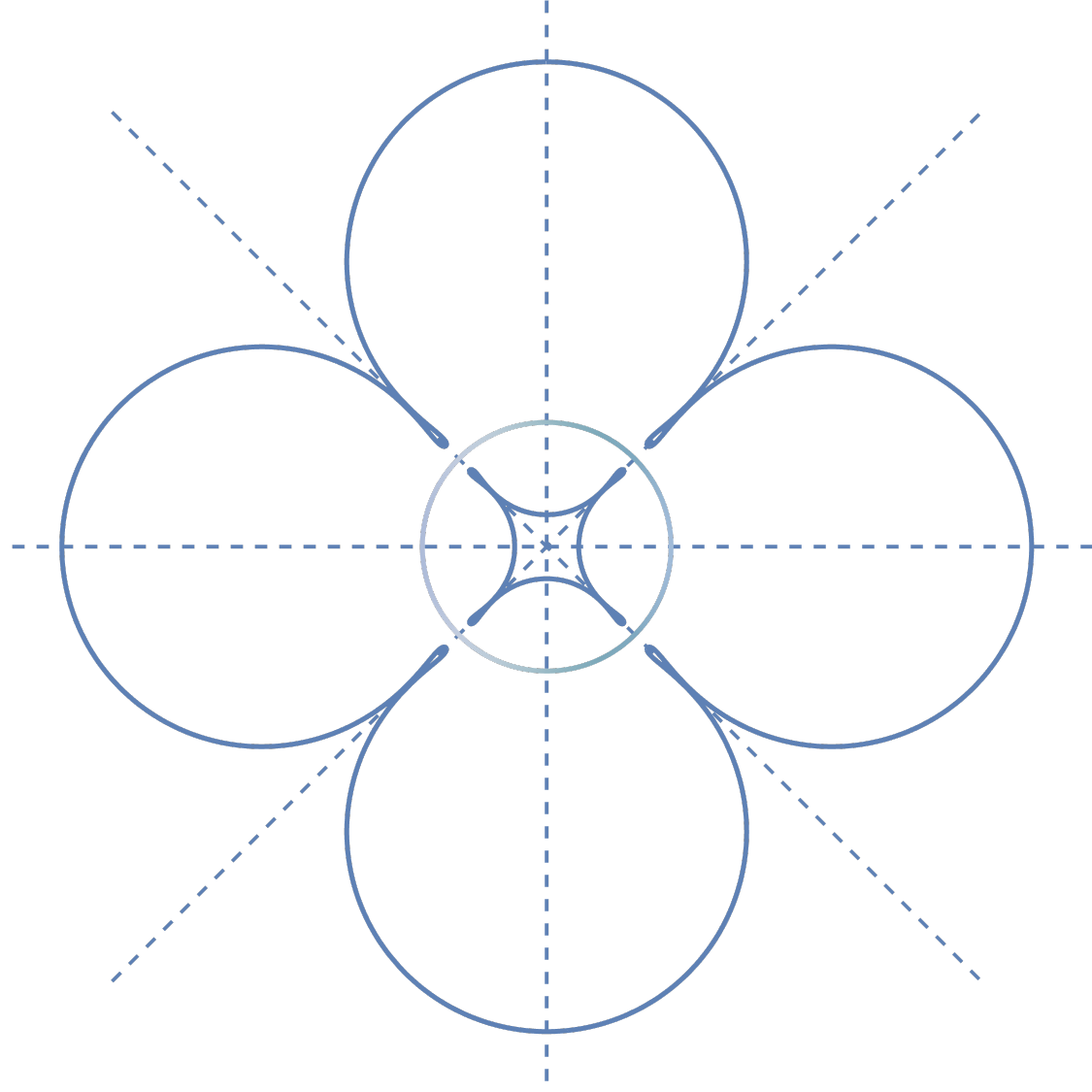}  
  \caption{Left: the moment $\hat{x}$ in Proposition \ref{noem} where the embeddedness of $y\mapsto g(x+iy)$ is lost, for an example with $n=4$. Right: The interior curve $g(\hat{x}+iy)$ and the exterior curve $g(-\hat{x}+iy)$ differ by an $\S^1$-inversion.\label{fihat}}
\end{figure}
\begin{proposition}\label{noem}
There exists $\hat{x}\in (x_b^+,x_a^+)$ such that $\gamma(y):=g(x_0+iy):\S^1\flecha \C$ is embedded if and only if $x_0\in (-\hat{x},\hat{x})$.
\end{proposition}
\begin{proof}
Denote $\cD:=[0,x_a^+)\times [0,\vart]$, and let $\varphi_1$ be the harmonic function on $\cD$ $$\varphi_1(x,y):= {\rm Im}\left( e^{-i(\pi-\pi/n)}g(x+iy)\right).$$ Note that $\varphi_1(x,y)=0$ if and only if $g(x+iy)\in L_1$, because $e^{i(\pi-\pi/n)}\in L_1$. As $g'\neq 0$ at every point, the nodal set $\varphi_1^{-1}(0)$ is a disjoint union of regular analytic curves in $\cD$. One of such curves is $\cN_1:= \cD\cap \{y=\vart\}$, since $g(x+i\vart)\in L_1$ for every $x$.

Assume next that $x\in [0,x_b^+]$. Then, by \eqref{curcapi}, we see that $\gamma(y):=g(x+iy)$ is convex (recall that $\alfa(x),\beta(x)\geq 0$ in $[0,x_b^+]$). So, by Lemma \ref{lem:embe1}, $g(x+iy)$ lies in the open sector $S_0$, for any $y\in (0,\vart)$ and any $x\in [0,x_b^+]$. Therefore the set $\cN-\cN_1$ of nodal curves of $\varphi_1$ different from $\cN_1$ is at a positive distance from $\cD\cap \{x\leq x_b^+\}$.

\begin{figure}[h]
  \centering
  \includegraphics[width=0.55\textwidth]{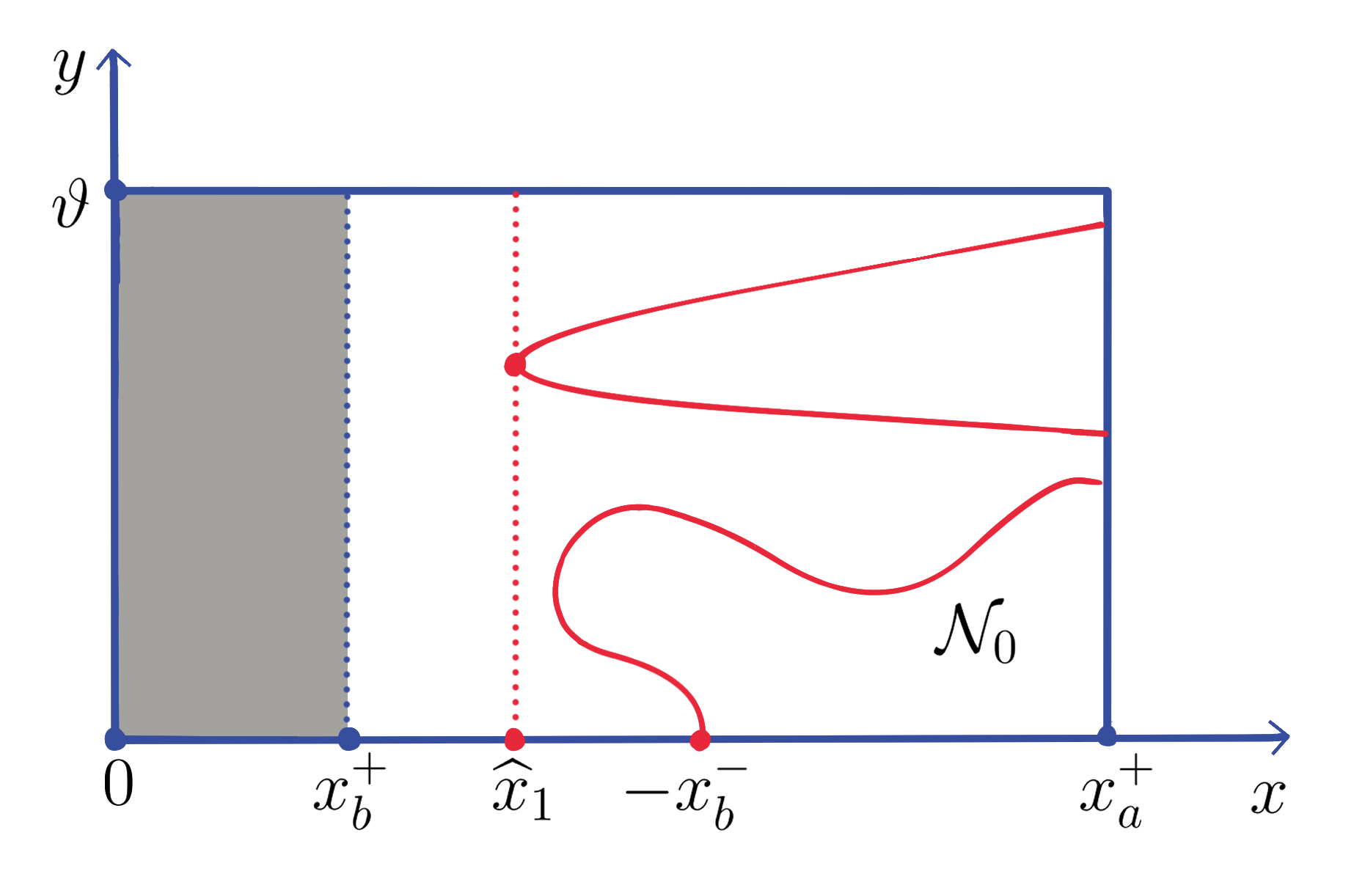} 
 \caption{The nodal curve $\cN_0$ and the value $\hat{x}_1$ in the proof of Proposition \ref{noem}. \label{rectang}}
\end{figure}


We next look at $\cD\cap \{y=0\}$. Since $g(x)\in \R$ with $g'(x)>0$ for every $x\in [0,x_a^+)$ and $g(-x_b^-)=0$ (see the final remark (1) at the end of the Appendix), we deduce that there exists a unique nodal curve $\cN_0\subset \cN$ that meets $\cD\cap \{y=0\}$. See Figure \ref{rectang}. This curve $\cN_0$ starts at $(-x_b^-,0)\in \cD$ and does not return to that point, by the maximum principle. Also, by the previous discussion, $\cN_0$ does not end at any point in $\cD\cap \{x\leq x_b^+\}$, or in $\cD\cap \{y=\vart\}$, or in $\cD\cap \{y=0\}$. That is, $\cN_0$ must end at some point of $x=x_a^+$, and in particular this means that $\cN_0\cap \{x=x_0\}$ is non-empty for any $x_0\geq -x_b^-$. Note that there could exist other nodal curves in $\cN$ other than $\cN_0, \cN_1$, but none of them approaches $x=x_b^+$. This implies that the set of values $x_0$ for which $\cD\cap \{x=x_0\}$ intersects the nodal set $\cN-\cN_1$ is an interval of the form $[\hat{x}_1,x_a^+)$, with $\hat{x}_1>x_b^+$.

We now apply the same argument to $\varphi_0(x,y):={\rm Im} (g(x+iy)):\cD\flecha \R$, which detects the points where $g(x+iy)\in L_0$. We obtain that either the nodal set $\varphi_0^{-1}(0)$ is composed only by $\cD\cap \{y=0\}$, or there exists some $\hat{x}_0>x_b^+$ such that $\cD\cap \{x=x_0\}$ intersects  $\varphi_0^{-1}(0)$ if and only if $x_0\in [\hat{x}_0,x_a^+)$.

Once here we define $\hat{x}:={\rm min}(\hat{x}_0,\hat{x}_1)$ if $\hat{x}_0$ exists, or merely $\hat{x}:=\hat{x}_1$ otherwise. Then, for $x_0\in [0,x_a^+)$ the arc $\gamma_0=g(x_0+iy)$ with $y\in (0,\vart)$ lies in the sector $S_0$ if and only if it never intersects $L_0\cup L_1$, i.e., if and only if $x_0< \hat{x}$. Applying Lemma \ref{lem:embe1} we prove the desired statement if $x_0\geq 0$.

When $x_0\in (x_b^-,0)$, the result follows directly from the case $x_0\geq 0$ we just proved, since $g(x+iy)$ and $g(-x+iy)$ always differ by an inversion with respect to the unit circle $\S^1\subset \C$. This property follows by Schwarz's reflection, since $g(iy)\in \S^1$ for every $y$. This completes the proof.
\end{proof}

\begin{remark}\label{xgoran}
By construction, the map $\hat{x}(\tau):(0,1)\flecha (0,\8)$ associated to the value $\hat{x}$ in Proposition \ref{noem} is continuous and piecewise real analytic in $\tau$, since it is given by the maximum or minimum of a finite collection of real analytic maps depending on $\tau\in 0,1)$.
\end{remark}

\begin{remark}\label{re:bordeOmega}
By construction, $\hat{x}$ is the smallest $x_0>0$ for which the arc $\gamma_0(y):= g(x_0+iy):[0,\theta]\flecha \C$ intersects $L_0^0\cup L_1^0$ at an interior point $y^0\in (0,\theta)$. Using the dihedral symmetry of the curve $g(\hat{x}+iy):\S^1\flecha \C$, we deduce that $g(\hat{x}+iy)$ has $n$ points of self-intersection, all lying in the union of the symmetry lines $L_j$. See Figure \ref{fihat}. By inversion with respect to $\S^1$, the same obviously holds for $g(-\hat{x}+iy)$.
\end{remark}

\subsection{Embeddedness of the domains $\Omega_{(s,\tau)}$}

\begin{theorem}\label{embe:domain}
For any $\tau\in (0,1)$ there exist values $\mathfrak{h}_0(\tau)<\mathfrak{h}_1(\tau)<0$ such that the immersed Serrin domain $\overline\Omega_{(s,\tau)}$ of Theorem \ref{th:imm} is embedded if and only if $s\in (\mathfrak{h}_0(\tau),\mathfrak{h}_1(\tau))$.

Moreover, the mappings $\mathfrak{h}_0(\tau),\mathfrak{h}_1(\tau):(0,1)\flecha (-\8,0)$ are continuous and piecewise real analytic in $\tau$.
\end{theorem}
\begin{proof}
Recall first of all that $g(x_b^+ + iy):\S^1\flecha \C$ is a regular convex Jordan curve, and that, by Proposition \ref{noem} there exists $\hat{x}\in (x_b^+,x_a^+)$ such that $y\mapsto g(x_0+iy)$ parametrizes an embedded curve if and only if $x_0\in (-\hat{x},\hat{x})$. 

We now look at the property that $\gamma(y):=g(x_0+iy):\S^1\flecha \C$ is a \emph{radial graph} with respect to the origin, that is, the property that $\gamma'(y)$ is transverse to $\gamma(y)$ for all $y$. This property holds if and only if ${\rm Re}(\gamma'/\gamma)\neq 0$ on $\S^1$, and it trivially implies that $\gamma(\S^1)$ is a starshaped curve in $\C$.

Recall that, by \eqref{curcapi}, the curves $g(x_0+iy):\S^1\flecha \C$ are stricly convex if $x_0\in(0,x_b^+)$, since $\alfa(x_0), \beta(x_0)>0$. On the other hand, if $\hat{x}>0$ is the value given by Proposition \ref{noem}, then $g(\hat{x}+iy)$ is not embedded, and in particular it is not a radial graph. So, there exists a smallest value $\hat{r}\in (x_b^+,\hat{x})$ at which $g(\hat{r}+iy)$ loses the property of being a radial graph. By Schwarz reflection using that $g(iy)\in \S^1$, we deduce then that $g(x_0+iy):\S^1\flecha \C$ is a radial graph for all $x_0\in (-\hat{r},\hat{r})$, but fails to have that property at $x_0=\pm \hat{r}$. 

Equivalently, $\hat{r}$ is the first positive value for which there exists some $y_0\in \R$ such that 
\begin{equation}\label{couga}
{\rm Re} \left(\frac{g'(\hat{r}+iy_0)}{g(\hat{r}+iy_0)}\right) =0.
\end{equation} 
By the dihedral symmetry of $g(\hat{r}+iy)$, we can assume that $y\in (0,\vart)$. 

From \eqref{couga}, the tangent line of $g(\hat{r}+iy)$ at $y=y_0$ is collinear with $g(\hat{r}+iy_0)$. This means that $g(\hat{r}+iy)$ does not lie around $y=y_0$ on one side of its tangent line (if it did, $g(x_0+iy)$ would not be a radial graph around $y=y_0$ for $x_0<\hat{r}$ close to $\hat{r}$). So, we deduce that the curvature of the planar curve $g(\hat{r}+iy)$ vanishes at $y=y_0$. Thus, by \eqref{curcapi} we have 
\begin{equation}\label{curga}
-\frac{\beta(\hat{r})}{\alfa(\hat{r})} = e^{-2\omega(\hat{r},y_0)}.
\end{equation}
The same argument obviously applies to $-\hat{r}$ when dealing with negative values of $x$; that is, there exists some $y_1\in (0,\vart)$ such that equations \eqref{couga} and \eqref{curga} hold at $(-\hat{r},y_1)$. On the other hand, it is direct from \eqref{fgp} that $g'/g$ has the same values at $x+iy$ and $-x+iy$, for any $x,y$. Thus, by comparing \eqref{couga} at $(\hat{r},y_0)$ and $(-\hat{r},y_1)$, we obtain $y_1=y_0$. So, by \eqref{curga}, 
\begin{equation}\label{curga2}
-\frac{\beta(\hat{r})}{\alfa(\hat{r})} = e^{-2\omega(\hat{r},y_0)}, \hspace{0.5cm} -\frac{\beta(-\hat{r})}{\alfa(-\hat{r})} = e^{-2\omega(-\hat{r},y_0)}.
\end{equation}

We next recall that, by Lemma \ref{lem:t}, for any $s\in (x_b^-,0)$ there is a unique $s^*\in (0,x_a^+)$ such that  \eqref{eseses} holds, and the immersed Serrin domain $\overline\Omega_{(s,\tau)}$ is given by $\overline\Omega_{(s,\tau)} = g([s,s^*]\times \S^1)$. 

Given $\tau\in (0,1)$, we let $\mathfrak{h}_0:=-\hat{x}$, and $\mathfrak{h}_1<0$ be given by $(\mathfrak{h}_1)^* =\hat{x}$, where $\hat{x}>0$ is defined in Proposition \ref{noem}. Note that $\mathfrak{h}_0(\tau)$ and $\mathfrak{h}_1(\tau)$ are piecewise real analytic, by their definition and Remark \ref{xgoran}.

Similarly, we denote $\varrho_0:= -\hat{r}<0$ and $\varrho_1<0$ be defined by $(\varrho_1)^*=\hat{r}$. We claim that $$\mathfrak{h}_0(\tau)<\varrho_0(\tau), \hspace{0.5cm} \varrho_1(\tau) <\mathfrak{h}_1(\tau).$$ Indeed, the first inequality is direct by definition, since $\hat{r}<\hat{x}$, while the second one is a consequence of the fact that the map $s\mapsto s^*$ is strictly increasing (Lemma \ref{lem:t}).

We show next that $\varrho_0(\tau)<\varrho_1(\tau)$ for every $\tau\in (0,1)$, what implies $\mathfrak{h}_0(\tau)<\mathfrak{h}_1(\tau)$.

When $\tau=1$, the curves $g(x_0+iy):\S^1\flecha \C$ are circles centered at the origin, for any $x_0\in \R$; see Section \ref{sec:tau1}. In particular, these curves are radial graphs. Thus, $\varrho_0(\tau)\to -\8$ and $\varrho_1(\tau)\to 0$ as $\tau\to 1$. This means that $\varrho_0(\tau)<\varrho_1(\tau)$ for $\tau$ close to $1$. 
Arguing by contradiction, assume that $\varrho_0(\tau_0)=\varrho_1(\tau_0)$ for some $\tau_0\in (0,1)$, and that $\tau_0$ is the largest value with such property. This means that, for $\tau_0$, one has $(-\hat{r})^*=\hat{r}$ and therefore by \eqref{eseses},
\begin{equation}\label{curga3}
\frac{\beta(\hat{r})}{\alfa(\hat{r})} =\frac{\beta(-\hat{r})}{\alfa(-\hat{r})} <0.
\end{equation}
By \eqref{curga2}, this implies $e^{2\omega(\hat{r},y_0)} = e^{2\omega(-\hat{r},y_0)}$. Since $e^{2\omega}= |g'|^2$ and $g'/g$ has the same values at $x+iy$ and $-x+iy$, we deduce that $$|g(-\hat{r}+iy_0)|= |g(\hat{r}+iy_0)|.$$ This is a contradiction, since $|g(x+iy)|>1$ for $x\in (x_b^-,0)$ while $|g(x+iy)|<1$ for $x\in (0,x_a^+)$.

This contradiction proves that $\varrho_0(\tau)<\varrho_1(\tau)$ for every $\tau\in (0,1)$, and in particular that $\mathfrak{h}_0(\tau)<\mathfrak{h}_1(\tau)$.

We are now ready to complete the proof. Consider the immersed Serrin domain $\overline\Omega_{(s,\tau)} = g(\cU)$, with $\cU$ as in \eqref{ubueno}. If $s\in (x_b^-(\tau),\mathfrak{h}_0(\tau)]$, the closed curve $g(s+iy):\S^1\flecha \C$ is not embedded by   Proposition \ref{noem}, and thus neither is $\overline\Omega_{(s,\tau)}$. For $s\in [\mathfrak{h}_1(\tau), 0)$, we have $s^*\geq (\mathfrak{h}_1)^* = \hat{x}$, see Lemma \ref{lem:t}. Thus, both $g(s^*+iy):\S^1\flecha \C$ and $\overline\Omega_{(s,\tau)}$ are, again, not embedded. 

The same argument shows that the closed curves $\gamma:=g(s+iy):\S^1\flecha \C$ and $\gamma^*:=g(s^*+iy):\S^1\flecha \C$ are both embedded if $s\in (\mathfrak{h}_0(\tau),\mathfrak{h}_1(\tau))$. Also, $\gamma$ and $\gamma^*$ are disjoint, since each of them lies in a different connected component of $\C\setminus \S^1$. A direct application of the fact that $g'\neq 0$ at every point shows then that $\parc \overline\Omega_{(s,\tau)} =\gamma\cup \gamma^*$ and that $g:\cU\flecha \overline\Omega_{(s,\tau)}$ is a bijection. This completes the proof. 
\end{proof}

%

\section{Limit necklaces} \label{sec:necklaces}

We discuss here the limit case $\tau=0$ of our construction. We fix $n\in\N$, $n\geq 2$, and consider the Serrin domains $\overline\Omega_{(s,\tau)}=g(\cU)$ in Theorem \ref{embe:domain} and the associated solutions $v(x,y)$ to \eqref{pdeom} constructed in  Theorem \ref{th:imm}.

\subsection{Description of the domains}\label{sec:domain0} 
Let us recall that $\overline\Omega_{(s,\tau)}=g(\mathcal{U})$, where $g(z)=g(z;\tau)$ is the holomorphic map defined in Section \ref{sec:g} (see also Section \ref{sec:summary}) and $\mathcal{U}=([s,s^*]\times\R,\sim)$ is the vertical quotient strip in \eqref{ubueno}. When $\tau\to 0$, we have $\vart\to\8$ (see \eqref{eq:cdvvart}) and $s^*\to s+2\pi$ (see Lemma \ref{lem:t} and \eqref{eq:alfatilde}). 

The limit case of the developing maps $g(z)$ for the domains when $\tau= 0$ was studied in Section \ref{sec:tau0}  and is given by \eqref{gt0}.  Each curve $y\mapsto g(x+iy)$ is a circular arc of curvature \eqref{curt0} joining the points ${\bf p}_0,{\bf p}_1$ in \eqref{eq:p0p1}.  

Then for each $s\in (-2\pi+2\pi/n,0)$,  $g(z)$ maps bijectively the domain $\cU=[s, s+2\pi]\times \R$ into a twice-punctured closed disk $\cD_0=\cD\setminus \{{\bf p}_0,{\bf p}_1\}$, where ${\bf p}_0$ and ${\bf p}_1$ are points of $\parc \cD$. Each of the curves $g(s,y)$ and $g(s+2\pi ,y)$ covers an arc of $\parc \cD\setminus \{{\bf p}_0,{\bf p}_1\}$, both starting at ${\bf p}_0$ and ending at ${\bf p}_1$.

The tangent lines to $\parc \cD$ at ${\bf p}_0$, ${\bf p}_1$ intersect at $(\hat{c},0)$, where
$$\hat{c}= -\cos\left(\frac{\pi}{n} \right)-\sin\left(\frac{\pi}{n}\right) \tan\left(\frac{s}{2}-\frac{\pi}{n}\right)\in\R.$$
The intersection angle between these tangent lines is ${2\pi}/{n}-s$. The radius of the disk $\cD$ is  $R=1/|\kappa(s)|$, where $\kappa$ is given by  \eqref{curt0}.

By dihedral symmetry, this configuration can be extended to a collection of $n$ twice-punctured discs along $\S^1$: 
$$\cD_0\cup\ldots\cup \cD_{n-1},$$
where $\cD_k$ is given by the rotation of angle $2k\pi/n$ of $\cD_0= \cD\setminus \{{\bf p}_0,{\bf p}_1\}$ defined above.  In particular, this  configuration is embedded if and only if $\cD$ is orthogonal to $\S^1$ at ${\bf p}_0$ and ${\bf p}_1$, or equivalently if $\hat{c}=0$, which occurs only for $s=-\pi$.  See Figure \ref{discosnoemb}. When $n=2$, $\cD$ is not a disk, but the half-plane $\{(x_1,x_2): x_1\leq 0\}$. 

\begin{figure}[h]
  \centering
  \includegraphics[width=0.3\textwidth]{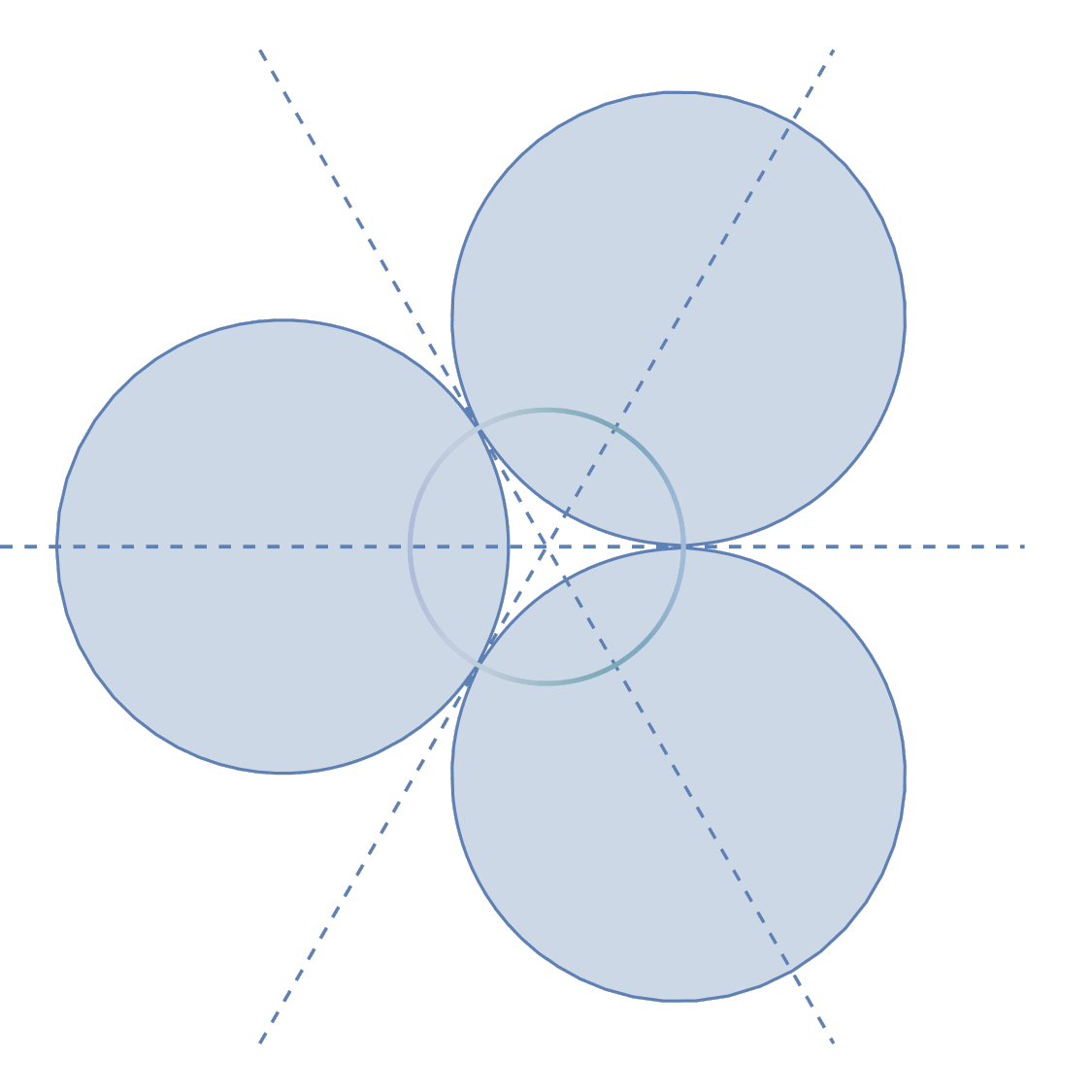} \hspace{0.4cm} \includegraphics[width=0.3\textwidth]{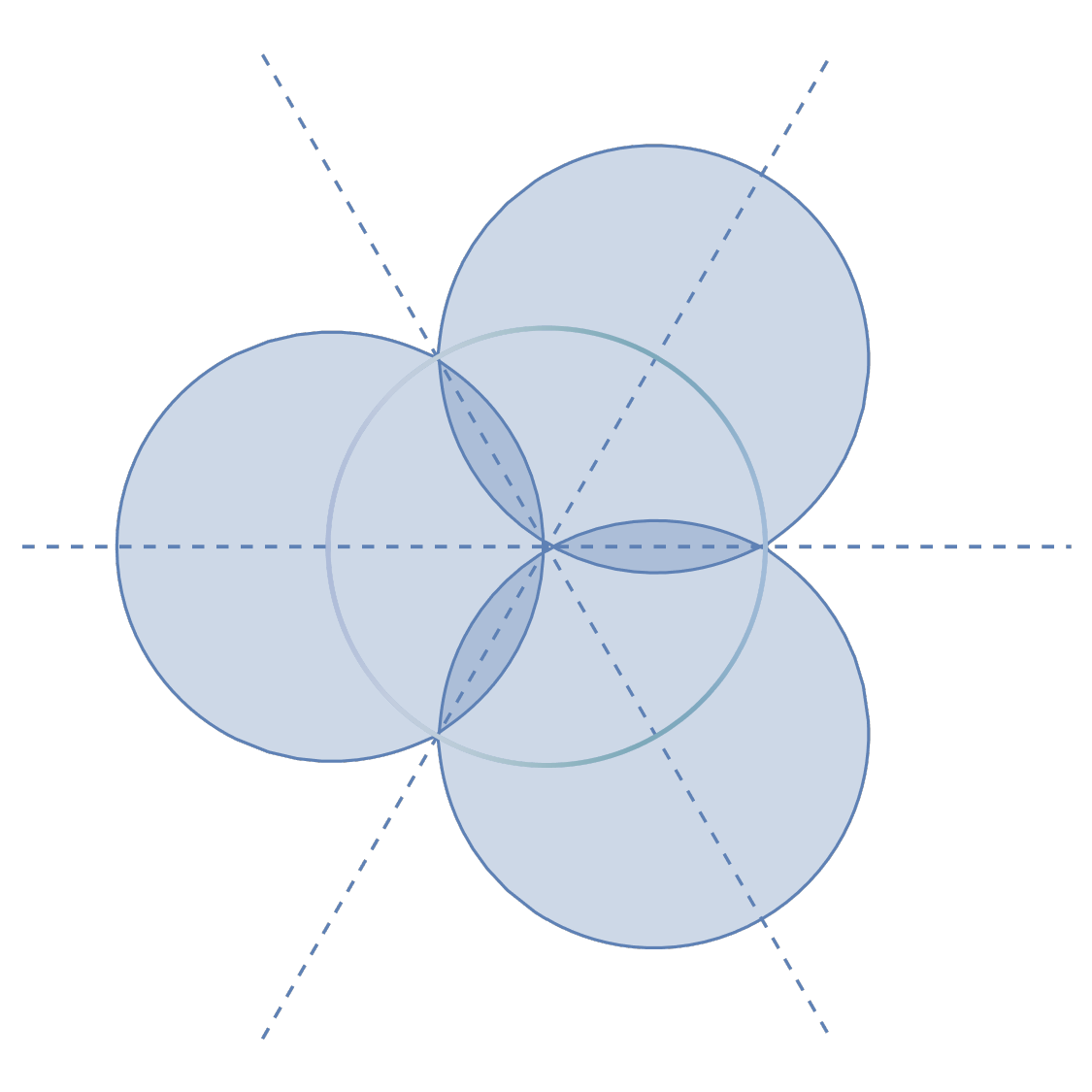}  \hspace{0.4cm} \includegraphics[width=0.3\textwidth]{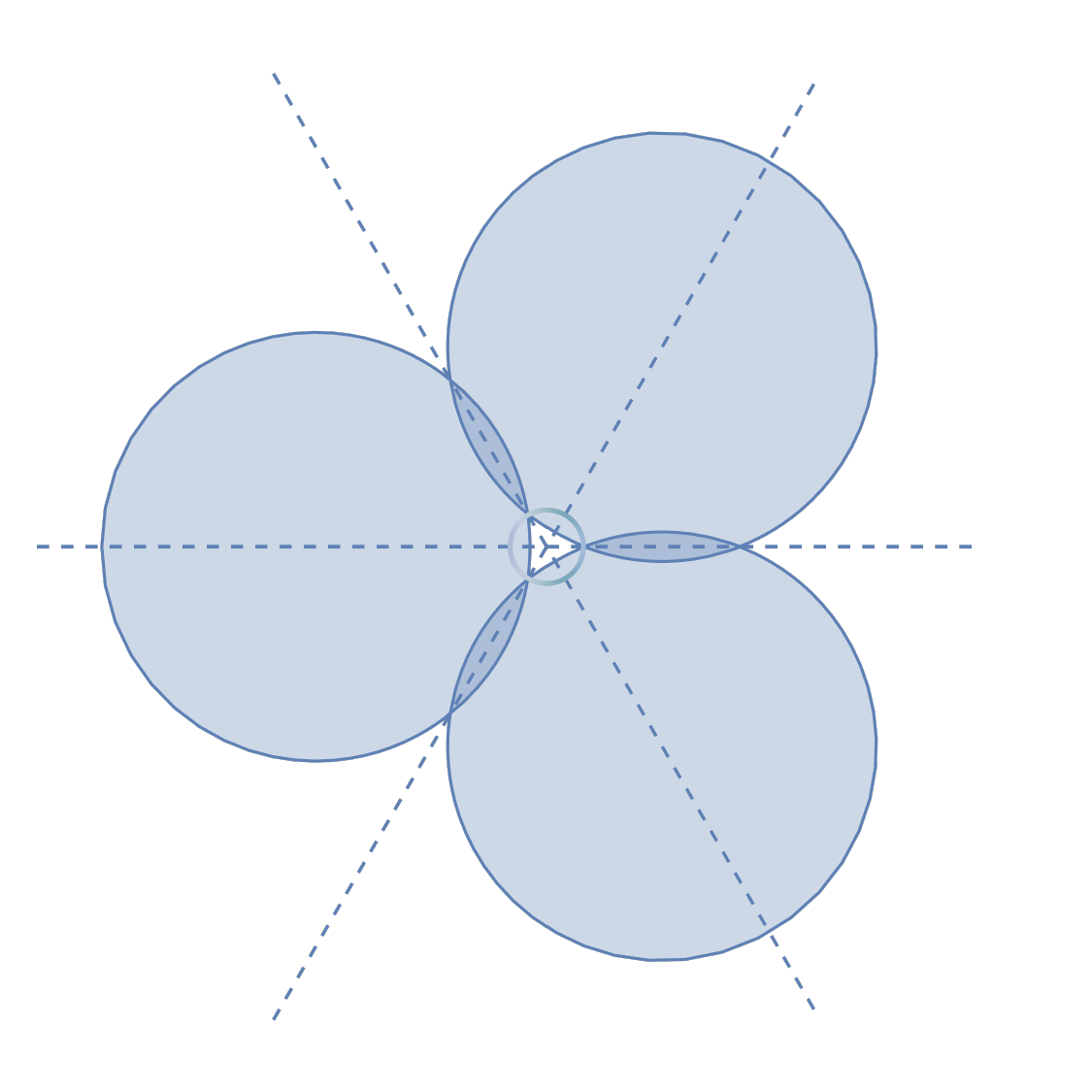}
 \caption{Limit necklaces for $n=3$. Left: the embedded case ($s=-\pi$). Middle: $s>-\pi$. Right: $s<-\pi$. In all figures the highlighted circle is $\S^1$.\label{discosnoemb}}
\end{figure}

\begin{remark}[Description of the necks]
In the limit configuration as $\tau\to 0$, the domains $\overline\Omega_{(s,\tau)}$ converge to a dihedral necklace of disks. The \emph{necks}, that is, the shape of $\overline\Omega_{(s,\tau)}$ near the singular points of the necklace as $\tau\to 0$, are all contained in the intersection of $\S^1$ with the symmetry lines, and must approach after blow-up the \emph{hairspin} domain by Helein, Hauswirth and Pacard \cite{HHP}. This is a consequence of the classification theorem by Traizet \cite{T}. It also can be checked directly in our setting by noting that, after some computations that we omit, $$\lim_{\tau\to 0} \frac{1}{\tau^2} g(z+\vart) = g^*(z):=-1 +\eta_n\left(\frac{z}{2}+\sin\left(\dfrac{z}{2}\right)\right), \hspace{0.5cm} \eta_n=\tan\left(\frac{\pi}{2n}\right).$$ The domain $\Omega^*:=g^*((-\pi,\pi)\times \R)$ is, up to similarity, the hairspin domain in \cite{HHP}.
%
\end{remark}

\subsection{Description of the solution $v(x,y)$}\label{sec:solution0}

We will now study the behaviour of the solutions $v(x,y)$ to \eqref{pdeom}  constructed in  Theorem \ref{th:imm} as $\tau\to 0$.  

We recall that $v(x,y)=a(x)+c(x)\sigma(x,y)$ is defined in terms of  the solution $(\alpha(x),\beta(x))$ to \eqref{system} with initial conditions \eqref{abd} for $(\eta,\tau)=(\eta_n(\tau),\tau)$ as in  Lemma \ref{lem:period}. The limit behaviour of $(\alpha(x),\beta(x))$ as $\tau\to 0$ was studied in Section \ref{sec:tau0} (see \eqref{eq:alfatilde}). 

We first look at $b(x)=4\mathfrak{q}/\alfa(x)$ in \eqref{abcbis}. Let us recall that $\mathfrak{q}=\mathfrak{q}(\tau)<0$ was defined by \eqref{eq:q}, where $t(x)=\beta(x)/\alpha(x)$. Therefore, we can write 
$$b(x)= \dfrac{ 2\ \tilde{\alpha}(s) }{ \beta(s)\ \tilde{\alpha}(x)},$$
where $\tilde{\alpha}(x)=\alpha(x)/\tau^2$. 
When $\tau=0$, we obtain then from \eqref{eq:alfatilde} that
\begin{equation}\label{ecub0}
b(x)= \dfrac{\vart_0}{\sin(x/2)},
\end{equation}
where $\vart_0$ is a non-zero constant.



We next look at $c(x)=1- 2\mathfrak{q} \beta(x)/\alfa(x)$ in \eqref{abcbis}. Writing $4\mathfrak{q}= b(x)\alfa(x)$ and using \eqref{ecub0} we have, with the same arguments, that $c(x)$ has a limit at $\tau=0$, equal to
$$c(x)=1-\vartheta_0\left(\frac{\cos(x/2)}{\sin(x/2)} + \frac{\eta_0-1/\eta_0}{2}\right)= 1-\vartheta_0\left(\frac{\cos(x/2)}{\sin(x/2)} - \frac{\cos(\pi/n)}{\sin(\pi/n)}\right).$$

Finally, we study $a(x)$. We have from \eqref{abcbis}
$$a'(x)= \frac{\mathfrak{q} \eta}{4\tau^2 \tilde\alfa(x)^2} - \frac{4\mathfrak{q}}{\tilde\alfa(x)^2} \int_0^x \tilde\alfa(s)^2 ds.$$ 
The second term goes to zero as $\tau\to 0$, since $\mathfrak{q}(\tau)\to 0$. For the first term, we obtain using again $4\mathfrak{q}= b(x)\alfa(x)$ and \eqref{ecub0} that, at $\tau=0$, we have
$$a'(x) = \frac{\vartheta_0}{4 \sin^2(x/2)}.$$ 
So, 
$$a(x)= a_0 -\frac{\vartheta_0}{2} \frac{\cos(x/2)}{\sin(x/2)},$$ 
for some constant $a_0$, that only amounts to adding a constant to the solution $v(x,y)$. 
So, from now on, we change $a_0$ freely in the next computations. 
Choosing $a_0$ adequately, the functions $v(x,y)=a(x)+c(x)\sigma(x,y)$  converge as $\tau\to 0$ to 
$$v(x,y)= \sigma(x,y) +\vartheta_0 \, \frac{\sin(x/2-\pi/n)}{\cos(x/2-\pi/n) + \cosh(y/2)},$$ where $\sigma(x,y)=-|g(x+iy)|^2/2$, as usual.
Again working up to additive constants and using that $g(z)$ for $\tau=0$ is given by \eqref{gt0}, we obtain
$$v(x,y)=\sigma(x,y) +\frac{\vartheta_0}{\sin(\pi/n)} {\rm Re}(g(z)).$$ 
Or finally, 
$$v(x,y)= -\frac{|g(z)-A_0|^2}{2}, \hspace{0.5cm} A_0:= \frac{\vartheta_0}{2\sin(\pi/n)}.$$ 
Thus, each $v(x,y)$ is a canonical solution (a paraboloid). Since $s^*\to s + 2\pi$, the formulas above for $a(x),b(x),c(x)$ give $a_1=a_2$ and $b_1=-b_2$.

\section{Proof of Theorem \ref{th:main}}\label{sec:mainth}


We are now ready to complete the proof of Theorem \ref{th:main}.  Fix $n\in  \N$, $n\geq 2$. In Section \ref{sec:immersed} we constructed, for each $(s,\tau)\in \cW_{imm}^n$ as in \eqref{wimm}, an 
{\em immersed Serrin domain} $\overline\Omega_{(s,\tau)}\subset\R^2$ depending analytically on $(s,\tau)$ (see Definition \ref{imser} and Theorem \ref{th:imm}). 
By Theorem \ref{embe:domain}, these domains are embedded when $s\in (\mathfrak{h}_0(\tau),\mathfrak{h}_1(\tau))$. Thus, for each $(s,\tau)\in \cW_0\subset \cW_{imm}^n$, with $\cW_0$ defined in \eqref{def:wo}, the planar sets $\Omega_{(s,\tau)}:={\rm int}(\overline\Omega_{(s,\tau)})\subset \R^2$  are Serrin ring domains (see Remark \ref{cuandoserrin}), and they have dihedral symmetry of order $2n$ (see Corollary \ref{re:symmetries}). This proves item (1). The description of the developing maps of the domains in terms of elliptic functions can be seen in the Appendix; thus, item (7) also holds.
\vspace*{0.2cm}

{\it Proof of items (3) and (4):} By definition,  $\mathfrak{h}_0(x)=-\hat{x}$, where $\hat{x}$ is given by Proposition \ref{noem}, and $\mathfrak{h}_1(x)^*=\hat{x}$ (see Lemma \ref{lem:t}).  Thus, items (3) and (4) are a straightforward consequence of Remark \ref{re:bordeOmega}. 

\vspace*{0.2cm}

{\it Proof of item (5):}
  Let us recall that $\overline\Omega_{(s,\tau)}=g(\mathcal{U})$,  where $g(z)=g(z;\tau)$ is the holomorphic map constructed in Section \ref{sec:gaussmaps} (see  Section \ref{sec:summary} for a summary of its properties)  and $\mathcal{U}$ is the vertical quotient strip in \eqref{ubueno}. 
 By our analysis in Section \ref{sec:tau1}, when $\tau= 1$ we have $g(z)=-e^{-c_1 z}$ for a certain constant $c_1>0$. So, in this case, $(x_b^-,x_a^+)=\R$, and the restriction of $g(z)$ to any domain $([s_1,s_2]\times \R,\sim)\subset \cU$ is a bijection onto a radial annulus. Since all such domains are embedded, we obtain
$$\lim_{\tau\to 0} \mathfrak{h}_0(\tau)=-\8 ,\quad \lim_{\tau\to 0}\mathfrak{h}_1(\tau)=0.$$

So, to prove item (5) we only need to check that if $(s,\tau)\to (\hat{s},1)$ for some $\hat{s}<0$, then the values $s^*$ associated to $s$ converge to a (finite) negative number as $\tau\to 1$. To do this, we consider the function $t(x)$, defined in \eqref{defit}, and recall that the values of $s<0$ and $s^*>0$ are related by $t(s)=t(s^*)$. When $\tau=1$, $t(x)$ is defined in $\R-\{0\}$, since by Lemma \ref{lem:xaxb} it holds 
 $x_b^-=-\8$ and $x_a^+=\8$ in that case. Since $t(s)$ is strictly decreasing on both connected components of its domain of definition, in order to see that $(\hat{s})^*$ exists if $\tau=1$ it suffices to show that, for $\tau=1$, it holds
 \begin{equation}\label{eq:limitt}
\lim_{x\to 0^+}t(x) = \8, \qquad \lim_{x\to \8} t(x) = -\8.
\end{equation}
The first equality is trivial, since $\alfa(0)=0$ and $\beta(0)=2$. Regarding the second one, 
by \eqref{eq:omega1} and using \eqref{rox} we obtain 
\begin{equation}\label{ec1}
2c_1=\alpha(x)\dfrac{1}{c_1} e^{c_1 x} + \beta(x)c_1 e^{-c_1 x}.
\end{equation}
As shown in the proof of Lemma \ref{lem:xaxb}, when $\tau=1$ the function $\mathfrak{s}(x)=\alpha(x)\beta(x)$ satisfies  $\mathfrak{s}(x)\to -1/4$ as $x\to \pm\8$. Thus, we obtain from \eqref{ec1} $$\alfa(x)\left(2 c_1 -\frac{1}{c_1}\alfa(x) e^{c_1x} \right) = c_1 \mathfrak{s}(x) e^{-c_1 x} \to 0 \hspace{0.5cm} \text{as $x\to \8$}.$$ This means that $\alfa(x)\to 0$ as $x\to \8$. Using finally that $t(x)=\mathfrak{s}(x)/\alfa(x)^2$ and that $\alfa(x)>0$ if $x>0$, we obtain \eqref{eq:limitt}. This finishes the proof of item (5). 
%

\vspace*{0.2cm}

{\it Proof of item (6):}   
Assume that $(s,\tau)\in \cW_0$ converges to $(\hat{s},0)$ for some $\hat{s}\in \R$, and let $\overline\Omega_{(s,\tau)}$ denote their corresponding annular regions. It was proved in Section \ref{sec:domain0} that, in that case, $\overline\Omega_{(s,\tau)}$ converge to a singular necklace of disks along $\S^1$. Since the regions $\overline\Omega_{(s,\tau)}$ are embedded when $(s,\tau)\in \cW_0$ by item (1), this necklace should also be embedded, and we proved in Section \ref{sec:domain0} that this only happens if $\hat{s}=-\pi$. From here, it follows that, in this $\tau=0$ case, $(\hat{s})^*=\hat{s}+2\pi =\pi$, and $\overline\Omega_{(s,\tau)}$ converge to a  singular necklace of $n$ pairwise tangent disks of the same radius along the unit circle $\S^1$. This proves item (6).
Notice that this also shows  that  
$$\lim_{\tau\to 0} \mathfrak{h}_0(\tau)= \lim_{\tau\to 0}\mathfrak{h}_1(\tau)=-\pi.$$

\vspace*{0.2cm}

{\it Proof of item (2):}  Let $\Omega_j :=\Omega_{(s_j,\tau_j)}$, with $j=1,2$, be two of the Serrin ring domains in item (1), associated to $n_j\geq 2$. Note that $\tau_j\in (0,1)$. In all that follows, the quantities associated to $\Omega_j$ will be denoted with a $j$ subscript without additional clarification, unless necessary.

Assume that $\Omega_2=F(\Omega_1)$ for some similarity $F$ of $\R^2$. Due to their dihedral symmetry groups, it is then clear that $n_1=n_2=n$ and that this similarity fixes the origin. So, $F$ has a rotation factor and a dilation factor. Again by the dihedral symmetry of $\Omega_j$, the rotation factor must be of the form $2\pi/n$, and can be assumed to be the identity. To rule out the dilation term, we note that $\Omega_j$ contains the unit circle $\S^1$ in its interior as the only element of the foliation of $\Omega_j$ by capillary curves where $\alfa_j(x)=0$. Thus, $\Omega_2$ cannot be a proper dilation of $\Omega_1$, since $F$ must preserve this special $\S^1$ curve.

Therefore, $F={\rm Id}$, and $\Omega_2=\Omega_1$. We now prove that this implies $(s_1,\tau_1)=(s_2,\tau_2)$.

After a dilation of the conformal parameters $(x,y)$ of $\Omega_2$, we may assume that the Hopf differentials $q_j$ of $\Omega_j$ coincide, i.e., $\mathfrak{q}_j=\mathfrak{q}<0$. Let $\gamma$ be, say, the exterior boundary curve of $\Omega_1 (=\Omega_2)$, parametrized by $g_j(s_j,y)$. Let $\mathfrak{s}$ denote the arc-length parameter of $\gamma$, which is independent from $j$. Denoting $\alfa_j= \alfa_j(s_j)$, $\beta_j:=\beta_j(s_j)$, we have by construction that $2\mathfrak{q}=\beta_j/\alfa_j$. Thus, letting $Z_j(y):=e^{-\omega_j(s_j,y)},$ we have from \eqref{curcapi} that $$-2\kappa_\gamma = \alfa_j \left(Z_j^2 +2\mathfrak{q}\right),$$ and so, in terms of $\mathfrak{s}$, we arrive at
\begin{equation}\label{czetas}
\alfa_1 \left(Z_1(\mathfrak{s})^2 +2\mathfrak{q}\right)=\alfa_2 \left(Z_2(\mathfrak{s})^2 +2\mathfrak{q}\right).
\end{equation}
We next observe that, by \eqref{eq:Zy}, it holds 
\begin{equation}\label{czetas4}
4 (Z_j)_y^2 = Q_j (Z_j),
\end{equation} 
where $Q_j(Z)$ is the polynomial $Q(s_j,Z)$ in \eqref{eq:roypol}. Since $|g_j'|^2=e^{2\omega_j}$, we have 
%
%
%
\begin{equation}\label{czetas2}
4\left(\frac{dZ_j}{d\mathfrak{s}}\right)^2= Z_j^2  Q_j (Z_j), \hspace{0.5cm} Z_j=Z_j(\mathfrak{s}).
\end{equation}
Differentiating \eqref{czetas} and comparing the result with \eqref{czetas2} we arrive at
\begin{equation}\label{czetas3}
\alfa_1^2 \,Z_1(\mathfrak{s})^4 \, Q_1(Z_1(\mathfrak{s}))=\alfa_2^2 \, Z_2(\mathfrak{s})^4 \, Q_2(Z_2(\mathfrak{s})).
\end{equation}
We next prove that \eqref{czetas}, \eqref{czetas3} cannot hold simultaneously unless $(s_1,\tau_1)=(s_2,\tau_2)$.

Consider first the case $\alfa_1=\alfa_2 (<0)$. By \eqref{czetas}, $Z_1(\mathfrak{s})=Z_2(\mathfrak{s})$, and so by \eqref{czetas3} we deduce that $Q_1=Q_2$, i.e. all the coefficients of these two polynomials coincide. In particular, $\delta_1=\delta_2$, from where $\tau_1=\tau_2$ by \eqref{abd} (we recall that $\eta$ in \eqref{abd} is determined by $\tau$ in our situation, since ${\rm Per}(\eta,\tau)=1/n$). This easily implies $s_1=s_2$, by uniqueness of the solution to $2\mathfrak{q}=\beta(x)/\alfa(x)$ for $x\in (x_b^-,0)$. So, $(s_1,\tau_1)=(s_2,\tau_2)$ in this case.

We finally rule out the case $\alfa_1\neq \alfa_2$. Recall that $\alfa_1,\alfa_2<0$ since we are working with the exterior boundary curve of our ring domain. Assume for definiteness that $\alfa_1<\alfa_2$, and recall that $\mathfrak{q}<0$. Then, by \eqref{czetas} there exists a real analytic function $\varphi=\varphi(\mathfrak{s})$ such that 
\begin{equation}\label{argucosh}
Z_1(\mathfrak{s}) =A_1 \cosh(\varphi(\mathfrak{s})), \hspace{0.5cm} Z_2(\mathfrak{s})=A_2\sinh (\varphi(\mathfrak{s})), \hspace{0.5cm} A_j:=\sqrt{\frac{2\mathfrak{q}(\alfa_1-\alfa_2)}{\alfa_j}}>0.
\end{equation} 
Introducing this expression into \eqref{czetas3} creates a non-constant real analytic function $\Phi(x)$ for which $\Phi(\varphi(\mathfrak{s}))=0$. Thus, $\varphi(\mathfrak{s})$ is constant, what means that both $Z_j(\mathfrak{s})$ are constant. This can only happen if $\tau=1$, a contradiction. This completes the proof of item (2), and so of Theorem \ref{th:main}.

\subsection*{One-dimensional bifurcation from radial annuli} 


We show next that, for each $n\geq 2$, the moduli space $\cW_0$ of Theorem \ref{th:main} contains a real analytic curve $\Upsilon_1$ of Serrin ring domains $\Omega_{(s,\tau)}$ that solve \eqref{overeq00} with $b_1=b_2$. This curve $\Upsilon_1$ starts at some point $(s_1,1)$, with $s_1<0$, of the ``radial'' boundary component $\tau=1$ of the moduli space $\cW_0$. In this sense, the curve $\Upsilon_1$ bifurcates from the family of radial annuli. A similar argument proves that there exists a real analytic curve $\Upsilon_2\subset \cW_0$ that provides solutions to \eqref{overeq00} with $a_1=a_2$.

To show the existence of the curve $\Upsilon_1$ and to detect the bifurcation value $s_1<0$, let us analyze our construction in the case $\tau=1$. Some aspects of this $\tau=1$ case were already treated in Section \ref{sec:tau1}. For example, $\omega$ only depends on $x$ and is given by \eqref{eq:omega1}, and $g(z)=-e^{-c_1 z}$. So, $g$ maps $y$-curves to circles; specifically, each $y\mapsto g(x_0+iy)$ is a circle centered at the origin of radius $R=e^{-c_1 x_0}$.

Take now $s<0$. Using \eqref{eq:limitt}, we can see that Lemma \ref{lem:t} also holds in this $\tau=1$ case, with the same proof. Thus, there exists a unique $s^*>0$ so that $t(s)=t(s^*)$, where $t(x):=\beta(x)/\alfa(x):\R-\{0\}\flecha \R$. The map $s\mapsto s^*$ is the real analytic extension to $\tau=1$ of the one in Lemma \ref{lem:t}. In particular:
\begin{enumerate}
\item
$s\mapsto s^*$ is strictly increasing and positive.
\item
$s^* \to \8$ as $s\to 0^-$.
\end{enumerate}
On the other hand, we consider next the function $\varrho(x):=\alfa(x)\beta(x)$, that was introduced in the proof of Lemma \ref{lem:xaxb} (in that lemma we used the notation $s(x)=\alfa(x)\beta(x)$). When $\tau=1$, $\varrho(x)$ satisfies $$\varrho'(x)^2 = 4(\varrho^2+1/4)^2 (\varrho-\eta^2/4), \hspace{0.5cm} \varrho(0)=0.$$ Because of this equation, $\varrho(x)\to -1/4$ as $x\to \pm \8$, and it has a unique critical point $x_0>0$ (a maximum), at which $\varrho(x_0)=\eta^2/4$. Since $\varrho(0)=0$, this implies that for any $s<0$ there exists a unique $s^\flat>0$ such that $\varrho(s)=\varrho(s^\flat)$. Moreover,
\begin{enumerate}
\item
The map $s\mapsto s^\flat$ is strictly decreasing, with a finite positive limit at $s=0$.
\item
$s^\flat\to \8$ as $s\to -\8$
\end{enumerate}
From the specified properties of the maps $s\mapsto s^*$ and $s\mapsto s^\flat$ we easily deduce the existence of a unique $s<0$ such that $s^*=s^\flat$, that we denote by $s_1$. Therefore, we have $$\alfa(s_1)=-\alfa(s_1^*)<0, \hspace{0.5cm} \beta(s_1)=-\beta(s_1^*)>0.$$ From \eqref{abcbis}, we see that $b(s_1)=-b(s_1^*)$, and the function $s\mapsto b(s)+b(s^*)$ changes sign at $s=s_1$, because $s\mapsto \alfa(s)+\alfa(s^*)$ does.

Once here, for any $(s,\tau)\in \cW_0$, with $\cW_0$ given by \eqref{def:wo}, let $b(s;\tau)$ and $b(s^*;\tau)$ denote the values of the function $b(x)$ in our construction (see \eqref{abcbis}) at the points $x=s$ and $x=s^*$. These are two real analytic functions on the moduli space $\cW_0$, that extend analytically to $\tau=1$. Denote $\Psi(s,\tau):= b(s;\tau)+b(s^*;\tau)$. Then, we have proved that $\Psi(s,1)$ has a change of sign at $s=s_1$. By real analyticity, this means that there exists a real analytic curve $\Upsilon_1$ in $\cW_0$ along which $\Psi(\Upsilon_1)=0$. So, for any $(s,\tau)\in \Upsilon_1$, the Serrin ring domain $\Omega_{(s,\tau)}$ has the property that its associated solution $u$ to \eqref{overeq00} satisfies $b_1=b_2$; see Remark \eqref{cuandoserrin}. This proves our claim.

A similar argument, that we omit, shows that there exists a real analytic curve $\Upsilon_2$ in $\cW_0$ starting from some point $(s_2,1)$ in the $\tau=1$ line, and along which $a_1=a_2$ in \eqref{overeq00} holds.

Recall that our domains $\Omega_{(s,\tau)}$ have a dihedral symmetry group $D_n$ of order $2n$. Thus, each of these curves $\Upsilon_1$, $\Upsilon_2$ can be seen as a one-dimensional bifurcation of Serrin ring domains from the family of radial annuli in $\R^2$, with a fixed symmetry group $D_n$. In these conditions, it can be proved from the bifurcation analysis in \cite{ABM,KS} that the domains $\Omega_{(s,\tau)}$ with $(s,\tau)\in \Upsilon_2$ (resp. $\Upsilon_1$) correspond, when $\tau$ is close to $1$, to the Serrin ring domains by Agostiniani, Borghini and Mazzieri \cite{ABM} (resp. by Kamburov and Sciaraffia \cite{KS} for $n=2$). The discussion of this fact requires some technical considerations from bifurcation theory that we omit. It also relies on noting that all examples in \cite{ABM,KS} actually have a dihedral symmetry group. That is, the ring domains that, by their construction in \cite{ABM,KS}, are invariant by a cyclic group $G$ of rotations of $\R^2$, actually have a larger symmetry group that includes reflections.

\section{The periodic Serrin problem: proof of Theorem \ref{th:bandas}}\label{sec:bandas}

In this section we prove Theorem \ref{th:bandas}, which describes a $1$-dimensional real analytic family of periodic Serrin bands. The arguments here are similar to those used for the construction of ring domains in Theorem \ref{th:main}, but more direct. The key idea is that this time we will seek solutions to $\Delta u+2=0$ that are foliated by \emph{planar} capillary curves. This makes the process simpler, but we need to change many aspects of the proof, because this time the planes containing these curves do not have a common vertical axis.

We will first describe in Section \ref{sec:gflats} a $1$-parameter family of holomorphic developing maps $g(z;\tau)$, with $\tau\in (0,1)$. In Section \ref{sec:vflats} we will construct solutions $v(x,y)$ to $\Delta v+2e^{2\omega}=0$, where $\omega=\log |g'|$. These solutions will be foliated by planar capillary curves. Then in Section \ref{sec:demflats} we will complete the proof of Theorem \ref{th:bandas}.

\subsection{The developing maps}\label{sec:gflats} Fix $\tau\in (0,1]$, and let $\alfa(x)$ be the solution to
\begin{equation}\label{aplan1}
\alfa''=\delta \alfa +2\alfa^3, \hspace{0.5cm} \delta :=-\left(\tau+\frac{1}{\tau}\right),
\end{equation} 
with the initial conditions $\alfa(0)=0$, $\alfa'(0)=1$. By uniqueness, $\alfa(-x)=-\alfa(x)$. Also, from \eqref{aplan1}, we have $$(\alfa')^2 = \alfa^4  +\delta \alfa^2 + 1.$$ Note that the polynomial $q(x)=x^4  +\delta x^2 + 1$ has real roots at $\{\pm \sqrt{\tau},\pm 1/\sqrt{\tau}\}$. 

If $\tau=1$, then $\alfa(x)=\tanh(x)$. If $\tau\in (0,1)$, then $\alfa(x)$ is periodic with $\alfa(\R)=[-\sqrt{\tau},\sqrt{\tau}]$. We denote by $x_a^+$ the first positive zero of $\alfa(x)$. If $\tau=1$, then $x_a^+=\8$.

Let us denote
\begin{equation}\label{def:betaflat}
\beta(x):=-\alfa(x).
\end{equation} 
It is direct then to check that $(\alfa(x),\beta(x))$ is the unique solution to system \eqref{system} with the initial conditions $\alfa(0)=\beta(0)=0$ and $\alfa'(0)=-\beta'(0)=1$.
 
%
\begin{lemma}\label{lioflat}
There exists a unique harmonic function $\omega(x,y)$ defined in $(-x_a^+,x_a^+)\times \R$ such that 
\begin{equation}\label{oveflat}
\omega_x(x,y) = \alfa(x) \sinh \omega(x,y)
\end{equation}
and $e^{\omega(0,0)} =1/ \tau$. If $\tau\in (0,1)$, then $\omega(x,y)$ cannot be extended to $[-x_a^+,x_a^+]\times\R$.
\end{lemma}
\begin{proof}
The proof is the same as that of Lemma \ref{lem:lio}, using \eqref{def:betaflat}.
\end{proof}
We remark that $\omega(x,y)$ in Lemma \ref{lioflat} still satisfies equation \eqref{roy}. Specifically, by the computations at the beginning of the proof of Proposition \ref{pro:sis} we obtain \eqref{eq:wy2} and \eqref{eqa}, where $\beta(x)$ is given by \eqref{def:betaflat}. Then, we use that $(\alfa(x),\beta(x))$ solve \eqref{system} and obtain \eqref{roy} directly from \eqref{eq:wy2}, \eqref{eqa}.

Thus, using \eqref{def:betaflat} and \eqref{roy}, we have in our situation that $Z(y):=e^{-\omega(0,y)}$ satisfies $$Z_y^2 = -\alfa'(0)Z^3 -\delta Z^2 -\alfa'(0)Z,$$ that is, 
\begin{equation}\label{zyflat}
Z_y^2 =p(Z):= -Z (Z^2+\delta Z +1).
\end{equation} 
Observe that $p(z)$ has zeros at $z=0,\tau,1/\tau$, and that $Z(0)=\tau$. So, $Z(y)$ is periodic, and $Z(\R)\subset [\tau,1/\tau]$. If $\tau=1$, then $Z(y)\equiv 1$.

We next define a holomorphic map $g(z)$ on the strip $(-x_a^+,x_a^+)\times \R$, by 
\begin{equation}\label{rel:omg}
e^{2\omega}=|g'|^2,
\end{equation} 
with initial conditions $g(0)=0$, $g'(0)=i/\tau$. Note that $g(z)$ is unique. Since $\alfa(0)=\beta(0)=0$, then by \eqref{curcapi} $g(iy)$ parametrizes a straight line in $\C$, which is actually the $x_1$-axis due to the chosen initial conditions. So, we have $g'(iy)=i/Z(y)\in i\R_+$. From here, $G(z):=-ig'(iz)$ takes positive real values if $z\in \R$, satisfies $G(0)=1/\tau$, and solves the elliptic equation 
\begin{equation}\label{edoG}
(G')^2= p(G) = 	-G \left(G^2 +\delta G +1\right) = -G (G-\tau)(G-1/\tau).
\end{equation} 
This means that $g(z)$ can be described by elliptic functions. We show next how to write $g(z)$ explicitly in terms of a Weierstrass $\wp$-function. To start, let $e_j=e_j(\tau)$, $j=1,2,3$, be
\begin{equation}\label{ess}
e_1= \frac{2-\tau^2}{12 \tau}, \hspace{0.5cm}  e_2= \frac{2\tau^2-1}{12 \tau}, \hspace{0.5cm} e_3=- \frac{\tau^2+1}{12 \tau}.
\end{equation} 
Note that $e_1+e_2+e_3=0$, and all $e_j$ are real with $e_1>e_2>e_3$. Also, let $\wp(z)=\wp(z;\tau)$ be the Weierstrass $\wp$-function given by 
\begin{equation}\label{edope}
\wp'(z)^2 = 4(\wp(z) -e_1)(\wp(z)- e_2)(\wp(z)-e_3).
\end{equation}
By the properties of $e_j$, it is well known then that $\wp(z)$ is doubly periodic over a rectangular lattice in $\C$, and its half-periods $\omega_1>0$, $\omega_2\in i\R_+$ are determined by the values $e_j$, i.e., $\omega_1=\omega_1(\tau)$ and $\omega_2=\omega_2(\tau)$. Moreover, $\wp(\omega_1)=e_1$, $\wp(\omega_1+\omega_2)=e_2$ and $\wp(\omega_2)=e_3$, again by standard properties of the theory. From here, we have:
\begin{proposition}
With the above notations, it holds
\begin{equation}\label{gflat}
g(z)= i \left( -4 \zeta(z+\omega_1) + \frac{(\tau+1/\tau)z}{3} + 4\zeta(\omega_1)\right).
\end{equation}
Here $\zeta(z)=\zeta(z;\tau)$ is the Weirstrass zeta function associated to $\wp$, which satisfies $\zeta'(z)=-\wp(z)$.
\end{proposition}
\begin{proof}
Let $\phi(z)$ denote the holomorphic function on the right-hand side of \eqref{gflat}. Then, $$\phi(0)=0, \hspace{0.5cm} \phi'(0)= i \left(4\wp(\omega_1) + \frac{\tau+1/\tau}{3}\right) = \frac{i}{\tau},$$ where we have used that $\wp(\omega_1) =e_1$, and \eqref{ess}. So, $g(0)=\phi(0)$ and $g'(0)=\phi'(0)$.

We denote next $\Phi(z):= -i \phi'(iz)$. Then, a computation from \eqref{edope} and \eqref{ess} gives $$(\Phi')^2 =-\Phi (\Phi-\tau)(\Phi-1/\tau),$$ i.e. $\Phi(z)$ solves \eqref{edoG}. By differentiation of \eqref{edoG}, both $\Phi(z),G(z)$ solve $G''=p'(G)/2$. Since $\Phi(0)=G(0)=1/\tau$ and $\Phi'(0)=G'(0)=0$, we deduce that $\Phi(z)=G(z)$, and since $g(0)=\phi(0)$, we obtain finally that \eqref{gflat} holds.
\end{proof}
The function $g'(z)$ is doubly periodic, and $g(z)$ is holomorphic in $(-\omega_1,\omega_1)\times \R$. In this way, $x_a^+=\omega_1$, see Lemma \ref{lioflat}. Also, $g(z)$ has the following symmetries, by Schwarz's reflection:
\begin{enumerate}
\item
$g(-x+iy)=\overline{g(x+iy)}$, since $g(i\R)\subseteq \R$.
\item
$g(z+k\omega_2)=-\overline{g(\bar{z}+k\omega_2)} +2 g(k\omega_2)$ for any $z\in \Z$ and $k\in \Z$, since $g'(\R+k\omega_2)\subset i\R$ and $g(k\omega_2)=kg(\omega_2)\in \R$.
\end{enumerate}

Denoting $g(x+iy)=g(x,y)$ and $\omega_2=i\vart$ for $\vart>0$, the second equation above indicates that
\begin{equation}\label{sigaflat}
g(x,y+k \vart) = \Psi_k(g(x,-y+k\vart)), \hspace{0.5cm} \text{for all $k\in \Z$},
\end{equation} 
where $\Psi_k$ is the symmetry of $\R^2$ with respect to the vertical line $x_1= k \,g(0,\vart)$. From \eqref{sigaflat} we also obtain
\begin{equation}\label{periflat}
g(x,y+2\vart)=g(x,y)+g(0,2\vart).
\end{equation}
%
By computations similar to those in Lemma \ref{lem:sigma} we can compute 
\begin{equation}\label{teflat}
\vart=\int_\tau^{1/\tau} \frac{1}{\sqrt{p(z)}}\, dz= \int_0^1 \frac{\sqrt{\tau}}{\sqrt{t(1-t)(t+\tau^2 (1-t))}} \, dt.
\end{equation}
This time we used the change of variable $z=-t(\tau-1/\tau) +\tau$.


\subsection{The solutions in conformal parameters}\label{sec:vflats}
We define next real functions $a(x),b(x)$ by 
\begin{equation}\label{abflat}b(x)=-\frac{2}{\alfa(x)}, \hspace{0.5cm} a'(x)=\frac{2\int_0^x \alfa(s)^2 ds}{\alfa(x)^2}.
\end{equation} 
Note that $a(x)$ is well defined in $(-x_a^+,x_a^+)$, but $b(x)$ is not defined at $x=0$. Also, $a(x)$ is defined up to an additive constant $a_0$ that will be specified later on, and it satisfies $a(x)=a(-x)$, since $\alfa(x)$ is odd.

Let $\mathfrak{h}(x):(-x_a^+,x_a^+)\flecha \R$ be the unique solution to the linear ODE
\begin{equation}\label{edoache}
\mathfrak{h}'' = (\delta +2\alfa^2) \mathfrak{h}, \hspace{0.5cm} \mathfrak{h}(0)=2,  \hspace{0.5cm} \mathfrak{h}'(0)=0.
\end{equation}
From \eqref{aplan1} we see that $\alfa'(x) \mathfrak{h}(x)-\mathfrak{h'}(x)\alfa(x) =2$. Thus, denoting 
\begin{equation}\label{fflat}
f(x):=\frac{\mathfrak{h}(x)}{\alfa(x)}:(-x_a^+,x_a^+)-\{0\}\flecha \R,
\end{equation}
we see that $f'(x)=-2/\alfa(x)^2$. Consider next
\begin{equation}\label{def:eleflat}
L(x,y):=\frac{b(x) e^{\omega}-a'(x)}{f'(x)} = \alfa(x) e^{\omega} + \int_0^x \alfa(s)^2 ds.
\end{equation}
Observe that $L(x,y)$ is real analytic, with $L(0,y)=0$ and $L_x(0,y)=1/Z(y)$.
We prove below that $L(x,y)$ is harmonic. First, from the harmonicity of $\omega$ it follows 
$$\Delta L = \alfa'' e^{\omega} +2\alfa' (\alfa+e^{\omega} \omega_x) + \alfa e^{\omega} \left(\omega_x^2+\omega_y^2\right).$$
We now put into this equation the values of $\alfa''$ in \eqref{aplan1}, of $\omega_x$ in \eqref{oveflat} and of $\omega_y^2$ in \eqref{roy} for $\beta(x)=-\alfa(x)$ as in \eqref{def:betaflat}. A computation shows then that the right side of the above expression cancels, and so $L(x,y)$ is harmonic.  From there, we deduce that
%
\begin{equation}\label{def:L2}
L(x,y)= {\rm Im} \, g(x+iy),
\end{equation}
since both harmonic functions in \eqref{def:L2} have the same Cauchy data along $(0,y)$.

We define next $v(x,y)$ as the solution to system 
\begin{equation}\label{vflat}
v(x,y)=a(x)+f(x) L(x,y), \hspace{0.5cm} v_x(x,y)=b(x) e^{\omega(x,y)} + f(x) L_x(x,y).
\end{equation} 
The compatibility of this system follows directly from the definition of $L(x,y)$.
Then, using the harmonicity of $L$ and \eqref{vflat} we can compute 
$$\Delta v = (v_x)_x +v_{yy}= b'e^{\omega} + b e^{\omega} \omega_x + f' L_x =-2e^{2\omega},$$ where for the last equality we use \eqref{oveflat}, \eqref{abflat}, \eqref{fflat} and \eqref{def:eleflat}. Even though we will not use it directly, we remark that the Hopf differential $\mathfrak{q}$ of $v(x,y)$ can be computed in a similar way, yielding 
\begin{equation}\label{normaq}
\mathfrak{q}= v_{zz}-2\omega_z v_z = -\frac{1}{2}.
\end{equation}
The map $v(x,y)$ is well defined and real analytic around $x=0$. This is immediate, since $a(x)$ is real analytic at $x=0$, and $f(x)L(x,y)$ also is, by \eqref{fflat}, \eqref{def:eleflat}. 

%
%


We also note that, for any $x_0\in (-x_a^+,x_a^+)$, equation \eqref{ove1} holds for the constants $c=0$, $d_1=0$, $d_2=f(x_0)$. In other words, the image of the curve $y\mapsto (g(x_0+iy), v(x_0,y))$ lies in a plane of $\R^3$ that is actually parallel to the $x_1$-axis. 
In this way, our foliation structure is this time by \emph{planar capillary curves}, as already advertised. 

\begin{lemma}\label{simex}
It holds $v(-x,y)=v(x,y)$.
\end{lemma}
\begin{proof}
Is is direct form \eqref{fflat} that $f(x)$ is odd. Also, we realize at once from \eqref{def:L2} that $L(x,y)$ is odd in $x$ as well, and hence $f(x) L(x,y)$ is even in $x$. Finally, since we know that $a(x)$ is even, we deduce from \eqref{vflat} the desired symmetry.
\end{proof}

%

Let $x_0\in (-x_a^+,x_a^+)$. By repeating the argument in the proof of Lemma \ref{lem:lio} for the situation (1) there, we obtain that the map $y\mapsto e^{\omega(x_0,y)}$ is not constant unless $\tau=1$. So, for $\tau\in (0,1)$, this means that $y\mapsto v(x_0,y)$ will be constant if and only if $f(x_0)=0$. We detect next the zeros of $f(x)$. Since $f'(x)=-2/\alfa(x)^2$, $f(x)$ is strictly decreasing, and diverging to $\8$ and $-\8$ at $x=0$ and $x=x_a^+$ respectively. So, there exists a unique $x^*\in (0,x_a^+)$ with $f(x^*)=0$. Since $f(x)$ is odd, $-x^*$ is the only root of $f(x)$ in $(-x_a^+,0)$.


Note that, by construction, the zeros $\{-x^*,x^*\}$ of $f(x)$ are those of $\mathfrak{h}(x)$ in \eqref{edoache}.
%
%
%
%
%
%
%
%
\subsection{Proof of Theorem \ref{th:bandas}}\label{sec:demflats}

We start by collecting some notation. For any $\tau\in (0,1)$, let $g(z)=g(z;\tau)$ denote the holomorphic map \eqref{gflat}. Also, let $\alfa(x)$ be given by \eqref{aplan1}, and let $x^*=x^*(\tau)$ denote the unique $x^*>0$ such that $f(x^*)=0$, as explained at the end of the previous section. Denote $\cU_\tau:= [-x^*,x^*]\times \R$, and consider the restriction of $g(z)$ to $\cU_\tau$. Also, let $\omega=\log |g'|$, which is a well-defined harmonic function on $\cU_\tau$, since $g'\neq 0$ at every point. Recall that, by \eqref{rel:omg}, $\omega(x,y)$ is the harmonic function of Lemma \ref{lioflat}.

Consider also the harmonic function $L(x,y)$ in $\cU_\tau$ defined by \eqref{def:L2}, and $v(x,y)$ defined in \eqref{vflat}, i.e., $$v(x,y):= a(x)+f(x) L(x,y):\cU_\tau\flecha \R,$$ where $a(x),f(x)$ are defined in \eqref{abflat}, \eqref{fflat}. Recall that $f(x)$ is odd, while $a(x)$ is even and defined up to an additive constant $a_0$. From now on, we chose this constant $a_0$ so that 
\begin{equation}\label{defa0}
a(x^*)=a(-x^*)=0.
\end{equation}
 As we have shown, $v(x,y)$ is even with respect to $x$ (Lemma \ref{simex}), and solves $\Delta v+2e^{2\omega}=0$.

Assume for one moment that $g(z)$ is injective on $\cU_\tau$, and denote $\overline\Omega_\tau:= g(\cU_\tau)$. Observe that $\gamma_0(y):=g(iy)$ parametrizes the $x_1$-axis with $\gamma_0'(y)<0$ (since $\gamma_0'(0)=-1/\tau<0$). Thus, $g(-x^*+iy)$ (resp. $g(x^*+iy)$) parametrizes the component of $\parc \Omega$ for which $x_2<0$ (resp. $x_2>0$). The exterior unit normal $\nu$ of $\parc \Omega_\tau$ is given, in the conformal $(x,y)$ coordinates, by $\nu = e^{-\omega}\parc_x$ (resp. by $\nu=-e^{-\omega}\parc_x$) along $x=x^*$ (resp. $x=-x^*$).

With this, we define $u:= v\circ g^{-1}: \overline\Omega_\tau\flecha \R$. Note that $u$ is analytic and, due to the previous equations, it satisfies $\Delta u+2=0$ in $\Omega_\tau$. Also, by \eqref{vflat} and \eqref{defa0}, $u$ has the boundary conditions $$u=0, \hspace{0.5cm} \frac{\parc u}{\parc \nu} =b \hspace{0.5cm} \text{along $\parc \Omega_\tau$},$$ where are are denoting $b:=b(x^*)=-b(-x^*)<0$. We also note that every curve $y\mapsto g(x_0+iy)$ has the vertical symmetries given by \eqref{sigaflat}, and the periodicity in \eqref{periflat}.

In other words: $\Omega_\tau:={\rm int}(\overline\Omega_\tau)$ will be a non-flat periodic Serrin band as long as $g(z)$ is injective on $\cU_\tau$. We prove next this property:

{\bf Claim:} \emph{$g:\cU_\tau\flecha \overline\Omega_\tau$ is a holomorphic bijection for all $\tau\in (0,1)$.}

\noindent \emph{Proof of the claim:} To start, we consider the limit case $\tau=1$. From \eqref{zyflat} we have $Z(y)\equiv 1$ for any $y$, that is, $\omega(0,y)=0$. By \eqref{oveflat}, we have $\omega_x(0,y)=0$ and so, by uniqueness of the Cauchy problem for harmonic functions, $\omega\equiv 0$. Thus, $g(z)=iz$ and, by \eqref{def:L2}, $L(x,y)=x$. We now use that $\alfa(x)=\tanh(x)$ when $\tau=1$ to obtain from \eqref{abflat} and \eqref{fflat} that
\begin{equation}\label{efeuno}
b(x)=-2\coth(x), \hspace{0.5cm} a(x)=a_0+x^2 -2 x \coth(x), \hspace{0.5cm} f(x)= 2(\coth(x)-x),
\end{equation} 
So, by \eqref{vflat}, we obtain $v(x,y)=a_0-x^2$. Thus, the solution $u(x_1,x_2)$ to $\Delta u +2=0$ defined by $u=v\circ g^{-1}$ only depends on $x_2$ and is given by $u(x_1,x_2)= a_0 -x_2^2$.
Moreover, from the expression of $f(x)$ in \eqref{efeuno}, we see that $x^*>0$ is, in this $\tau=1$ case, the unique positive solution $x^\sharp$ to 
\begin{equation}\label{criti}
x^\sharp=\coth(x^\sharp).
\end{equation} 
Also by \eqref{efeuno}, $b(x^\sharp)=-2x^\sharp$ and $a(x^\sharp)= a_0-(x^\sharp)^2$. By \eqref{defa0}, we obtain then $a_0 =(x^\sharp)^2$. 

Therefore, if we denote $\cU_1:=[-x^\sharp,x^\sharp]\times \R$, with $g(z)=g(z;1)=iz$ and $\overline\Omega_1:=g(\cU_1)$, then $g(z):\cU_1 \flecha \overline\Omega_1$ is a holomorphic bijection (in this case, $\overline\Omega_1$ is a flat band of width $2x^\sharp$). 

Consider next the case $\tau\in (0,1)$. For any such $\tau$, we denote $\mathcal{K}_\tau:=[-x^*,x^*]\times [0,2\vart]$, where $\vart=\vart(\tau)>0$ is given by \eqref{teflat}. Then, $\mathcal{K}_\tau$ is a (compact) fundamental domain of $g(z;\tau)$; see \eqref{periflat} for the periodicity of $g(z;\tau)$. Now, taking limits in the second integral of \eqref{teflat} we see that $\vart \to \pi$ as $\tau\to 1$. Thus, $\mathcal{K}_\tau\to \mathcal{K}_1:= [-x^\sharp,x^\sharp]\times [0,2\pi]$ analytically as $\tau\to 1$, where $x^\sharp>0$ is given by \eqref{criti}. Also, $g(z;\tau)$ converges analytically to $g(z;1)=iz$ as $\tau\to 1$.



Thus, since $g(z;1):\cU_1 \flecha \overline\Omega_1$ is a holomorphic bijection, there exists some $\tau_0 \in [0,1)$ such that $g(z;\tau)$ defines a bijection from the compact fundamental domain $\mathcal{K}_\tau$ into $g(\mathcal{K}_\tau)$, for all $\tau\in (\tau_0,1]$. By the periodicity of $g(z;\tau)$ in \eqref{periflat}, we deduce that $g(z;\tau):\cU_\tau\flecha \overline\Omega_\tau$ is also a holomorphic bijection for all $\tau\in (\tau_0,1]$.


Let now $\tau_1\in [0,1)$ denote the smallest value of $\tau_0$ with the above property. We prove next that $\tau_1=0$. Otherwise, we would have that $g(z;\tau):\cU_\tau\flecha \overline\Omega_\tau$ is bijective for all $\tau\in (\tau_1,1)$ but not at $\tau_1$. Note that $g(z):=g(z;\tau)$ is a local diffeomorphism (since $g'$ never vanishes), and that the two boundary components of $\overline\Omega_{\tau}$ never intersect one another, as each of them lies in one of the open half-planes $\{x_2<0\}$ or $\{x_2>0\}$. So, if $g(z;\tau_1)$ ceases to be bijective, it will be because one (and hence, by symmetry, both) of the boundary curves in $\parc \overline\Omega_{\tau_1}$ self-intersects. This means that for $\tau>\tau_1$ close enough to $\tau_1$, the boundary curves of $\overline\Omega_\tau$ cannot be graphs, i.e., that the periodic Serrin band $\Omega_\tau$ cannot be of the form $\Omega^{\varphi}$ in \eqref{bigraph} for any periodic function $\varphi(x)$. This contradicts Theorem \ref{rosic} by Ros-Sicbaldi, and finishes the proof of the claim above.

We now complete the proof of Theorem \ref{th:bandas}. We have already constructed a real analytic $1$-parameter family of periodic Serrin bands $\{\Omega_\tau : \tau\in (0,1]\}$ so that items (1), (2) and (6) hold. 

Regarding item (3), we have from \eqref{gflat} that, for $\tau\in (0,1)$, 
\begin{equation}\label{derig}
g'(z)= i\left(4\wp(z+\omega_1) +\frac{\tau+1/\tau}{3}\right).
\end{equation}
For any $x_0\in (-\omega_1,\omega_1)$, denote $\phi_{x_0}(y):= {\rm Im}(g(x_0+iy))$. Clearly, $|\phi_{x_0}(y)|$ gives the distance of $g(x_0+iy)$ to the $x_1$-axis. Then, by \eqref{derig}, 
$$\phi_{x_0}'(y)= -4 \,{\rm Im} \left(\wp(x_0+iy+ \omega_1)\right).$$ Now, the points where $\wp(x+iy)$ is real are those where $x=k\omega_1$ or $iy =k\omega_2$, for $k\in \Z$. This means that, if $x_0\neq 0$, then $\phi_{x_0}(y)$ is not constant and has critical points exactly when $y= k\, {\rm Im}(\omega_2)$ for $k\in \Z$. We recall here that $\vart>0$ in \eqref{periflat} is given by $\omega_2=i\vart$.  So, along its fundamental period domain $[0, 2\vart]$, $\phi_{x_0}(y)$ has exactly two critical points;  one of them ($y=\vart$) corresponds to a point of minimum distance from the curve $g(x_0+iy)$ to the $x_1$-axis, while the maximum such distance is attained at $y=0$ (or equivalently, at $y= 2\vart$).

If we now take $x_0=x^*$ or $x_0=-x^*$, then $g(x_0+iy)$ parametrizes one of the boundary curves of $\parc \Omega_\tau$. Thus, item (3) holds by the above properties.

Regarding item (4), take $\tau,\tau'\in (0,1)$, and assume that $\Omega_{\tau'}=F(\Omega_{\tau})$ holds for some similarity $F$ of $\R^2$. Since both domains are symmetric bigraphs over the $x_1$-axis, and are symmetric with respect to the $x_2$-axis, we deduce that this similarity $F$ must actually be the composition of a dilation and a horizontal translation, i.e., $\Omega_{\tau'}= \landa \Omega_\tau + {\bf v}$, where $\landa>0$ and ${\bf v} =(v_1,0)\in \R^2$. Now, each of $\Omega_\tau$, $\Omega_{\tau'}$ is periodic, with a $\Z$-family of vertical symmetry lines, and $x_1=0$ is one such line for both of them. Moreover, the maximum distance of $\parc\Omega_\tau$ to the $x_1$-axis happens at $x_1=0$, and the same holds for $\parc \Omega_{\tau'}$. Thus, we can assume that ${\bf v}=(0,0)$, and so $F(x_1,x_2)=(\landa x_1,\landa x_2)$ for some $\landa>0$.

In what follows, we will label quantities with sub or superindices $\tau$ or $\tau'$ when referring to objects defined on $\Omega_\tau$ or $\Omega_{\tau'}$. Let $g_\tau:[-x_\tau^*,x_\tau^*]\times \R\flecha \overline{\Omega_\tau}$ and $g_{\tau'}:[-x_{\tau'}^*,x_{\tau'}^*]\times \R\flecha \overline{\Omega_{\tau'}}$ denote the respective developing maps of $\Omega_{\tau},\Omega_{\tau'}$. Define then 
\begin{equation}\label{changeg}
g_1(z):= \landa g_\tau(z/\landa), \hspace{0.5cm} g_2(z):= g_{\tau'}(z).
\end{equation} 
Denoting $\overline{\Omega_j}$ as the image of $g_j$, we have $\Omega_1=F(\Omega_\tau)=\Omega_{\tau'}=\Omega_2$.

By \eqref{changeg} and the general relation $|g'|^2=e^{2\omega}$, the harmonic function $\omega_1$ associated to $g_1$ is given by $$\omega_1(x,y)= \omega^\tau (x/\landa,y/\landa).$$ Then, by \eqref{oveflat} and \eqref{aplan1} we obtain $$ \alfa_1(x)=\frac{1}{\landa}\alfa^\tau(x/\landa), \hspace{0.5cm} \delta_1= \frac{\delta^\tau}{\landa^2}.$$
Regarding $\tau'$, we have $\alfa_2(x)=\alfa^{\tau'}(x)$ and $\delta_2= \delta^{\tau'}$. So, $\alfa_j(x)$ satisfies \eqref{aplan1} with respect to $\delta_j$. We note for later use that $\alfa_1(0)=\alfa_2(0)=0$, while $\alfa_1'(0)=1/\landa^2$ and $\alfa_2'(0)=1$.


Let us introduce some further notations. We label $s_1:=-x_{\tau}^*/\landa$ and $s_2:=-x_{\tau'}^*$, as well as $\beta_j(x):=-\alfa_j(x)$. With all this notations, we can mimic the proof of item (2) of Theorem \ref{th:main}, see Section \ref{sec:mainth}. Specifically, we obtain that equations \eqref{czetas} and \eqref{czetas3} hold in our context, for $\mathfrak{q}=-1/2$. As in Section \ref{sec:mainth} we have two cases.

If $\alfa_1(s_1)=\alfa_2(s_2)$, then by the argument after \eqref{czetas3}, the polynomials $Q_1$ and $Q_2$ agree. This implies that $\delta_1=\delta_2$ and that $\alfa_1(x)$ and $\alfa_2(x)$ have the same initial values at $x=s_1$ and $x=s_2$, respectively. Thus, by \eqref{aplan1}, we have $\alfa_1(x)=\alfa_2(x+s_2-s_1)$. Since $x=0$ is the only zero of $\alfa_j(x)$, we deduce that $\alfa_1(x)=\alfa_2(x)$, for every $x$. Since $\alfa_1'(0)=1/\landa^2$ and $\alfa_2'(0)=1$, we obtain $\landa=1$. Therefore $\delta_1=\delta_2$, what implies $\tau=\tau'$.

The case $\alfa_1(s_1)\neq \alfa_2(s_2)$ can be ruled out with the same argument described around equation \eqref{argucosh} in Section \ref{sec:mainth}. This proves item (4) of Theorem \ref{th:bandas}.

In order to prove item (5) we must analyze the behavior of the domains $\Omega_\tau$ as $\tau\to 0$. In our construction, we used the normalization that the Hopf differential of the examples is always $\mathfrak{q}=-1/2$, see \eqref{normaq}, since this simplified the description considerably. But with this normalization, the horizontal distance separating two consecutive extrema of the boundary curves of $\Omega_\tau$ diverges to $\8$ as $\tau\to 0$, and the maximum height of $\parc \Omega_\tau$ with respect to the $x_1$-axis also blows up. Similar degeneration problems appear with our conformal parameters $(x,y)$.

To deal with this situation, we make a rescaling of our construction, by considering
\begin{equation}\label{gemu}
g_\mu(z):= \mu g(\mu z), \hspace{0.5cm} \mu:=\sqrt{\tau}\in (0,1).
\end{equation} 
Then, $\Omega_\tau^*:={\rm int}( g_\mu(\cU_\tau^*))$, where $\cU_\tau^* :=[-x^*/\mu,x^*/\mu]\times \R$, satisfies $\Omega_\tau^*=\mu \Omega_\tau$, i.e., the new Serrin band $\Omega_\tau^*$ is a dilation by factor $\sqrt{\tau}$ of $\Omega_\tau$. We seek to control the limit of $\Omega_\tau^*$ as $\tau\to 0$.

For this, we first prove that 
\begin{equation}\label{limuu}
\lim_{\tau\to 0} \frac{x^*}{\mu}= \frac{\pi}{2},
\end{equation} where we recall that $x^*=x^*(\tau)>0$ is the first positive zero of the function $\mathfrak{h}(x)$ in \eqref{edoache}. Denote $\mathfrak{h}_\mu(x):=\mathfrak{h}(\mu x)$. Then, 
\begin{equation}\label{edoache3}
\mathfrak{h}_\mu''(x) = \mu^2\left(\delta +2\alfa(\mu x)^2\right) \mathfrak{h}_\mu(x), \hspace{0.5cm} \mathfrak{h}_\mu(0)=2,  \hspace{0.5cm} \mathfrak{h}_\mu '(0)=0.
\end{equation}
The first positive zero of $\mathfrak{h}_\mu(x)$ is obviously $x^*/\mu$. Also, note that $\mu^2 \delta \to 1$ and $\alfa(\mu x)\to 0$ as $\tau\to 0$. Taking limits in \eqref{edoache3} we obtain $\mathfrak{h}_0''=\mathfrak{h}_0$, and hence $\mathfrak{h}_0(x)= 2\cos(x)$. This proves \eqref{limuu}.

Also observe that, by \eqref{teflat} and \eqref{gemu}, we have 
\begin{equation}\label{geflatmu}
g_\mu(x,y+2\hat{\vart}) = g_\mu(x,y)+g_\mu(2\hat{\vart}),\hspace{0.5cm} \hat\vart := \frac{\vart}{\mu}.
\end{equation}
Now, denoting $\Phi_\mu(z):= -i g_\mu'(iz)$ in \eqref{gemu}, we have $\Phi_\mu(z)= \mu^2 G(\mu z)$, where $G(z):=-i g'(iz)$. So, from \eqref{edoG},
\begin{equation}\label{edogest}
(\Phi_\mu' )^2 = -\Phi_\mu^3 + (\tau^2 +1) \Phi_\mu^2 - \tau^2 \Phi_\mu.
\end{equation}
Also, $\Phi_\mu(0)= \mu^2 G(0)= 1$. Thus, $\Phi_\mu(z)$ converges analytically as $\tau\to 0$ to the non-constant solution $\Phi_0(z)$ to $(\Phi_0')^2= -\Phi_0^2 (\Phi_0-1)$ with $\Phi_0(0)=1$, which is given by $$\Phi_0(z)=\frac{1}{\cosh^2(z/2)}.$$ Hence $$\lim_{\tau\to 0}g_\mu'(z) =i \Phi_0(-iz)= \frac{2i}{1+\cos(z)},$$ from where, using that $g_\mu(0)=0$, $$\lim_{\tau\to 0}g_\mu(z) =2 i \tan(z/2)=:g_0(z).$$

For $z=x_0+iy$ with $x_0> 0$, the curve $y\in \R\mapsto g_0(x_0+iy)$ is a circle arc of radius $2/\sin(x_0)$ centered at $(0,-2\cot(x_0))$ that goes from ${\bf p}_0:=(2,0)$ to ${\bf p}_1:=(-2,0)$ as $y$ increases from $-\8$ to $\8$. When $x_0=\pi/2$, this circle arc is centered at the origin and $g_0([-\pi/2,\pi/2]\times \R)$ is the disk $D(0;2)$ of radius $2$, punctured at ${\bf p}_0$ and ${\bf p}_1$. 

On the other hand, let $\mathcal{K}_\tau$ denote the fundamental period domain of $g_\mu(z)$, which by \eqref{geflatmu} is given by $[-x^*/\mu,x^*/\mu]\times [0,2\vart/\mu]$. As $\tau\to 0$, it follows from \eqref{teflat} and \eqref{limuu} that $$\lim_{\tau\to 0} \mathcal{K}_\tau = \left[-\frac{\pi}{2},\frac{\pi}{2}\right]\times \R.$$

Thus, the fundamental compact pieces $g_\mu(\mathcal{K}_{\tau})\subset \R^2$ of $\Omega_\tau^*$ converge as $\tau\to 0$ to the set $g_0([-\pi/2,\pi/2]\times \R)$, i.e., to the disk $D(0;2)$ of radius $2$, punctured at ${\bf p}_0$ and ${\bf p}_1$, as above. By repeated reflections along the vertical symmetry lines of $\Omega_\tau^*$ and the previous description, we deduce that $\Omega_0^*:=\lim_{\tau\to 0} \Omega_\tau^*$ is formed by $D(0;2)$ and its reflections along vertical lines of the form $x_1=2k$, with $k\in \Z$. Thus, $\Omega_0^*$ is a chain of tangent disks of radius $2$.

In other words, we just proved that the Serrin bands $\Omega_\tau^*$ satisfy item (5) of Theorem \ref{th:bandas}. By their own definition in terms of the domains $\Omega_\tau$ via \eqref{gemu}, it is obvious that the bands $\Omega_\tau^*$ also satisfy the rest of the items, because we have already proved that the Serrin bands $\Omega_\tau$ do.
Thus, by relabeling the domains $\Omega_\tau^*$ as $\Omega_\tau$ for $\tau\in (0,1)$, we conclude the proof of Theorem \ref{th:bandas}.

\begin{remark}
The flat strip domain $\Omega_1 (=\Omega_1^*)$ in Theorem \ref{th:bandas} is actually $\Omega^{\varphi_0}$ in \eqref{bigraph} with $\varphi_0\equiv x^\sharp$, where $x^\sharp>0$ is the unique positive solution to $x=\coth(x)$, see \eqref{criti}. This is exactly the bifurcation value $\landa^*$ for flat strip domains detected in \cite{FMW1} by means of a spectral analysis of the linearized problem. In our situation, the process of arriving at this bifurcation value follows the opposite path. We first construct the non-flat domains $\Omega_\tau$ with $\tau\in (0,1)$, and then we show that they converge to $\Omega^{\varphi_0}$ with $\varphi_0=x^\sharp$ as above.
\end{remark}

\section{Capillary domains and the mKdV hierarchy}\label{sec:kdv}
In this section we prove Theorem \ref{th:kdvintro}. We will actually prove a more general result that deals with \emph{capillary ring domains} (Definition \ref{def:capicu}). Recall that all Serrin ring domains are capillary. 

In Section \ref{sec1:kdv}, as a motivation, we show that certain infinitesimal deformations in conformal parameters of a capillary ring domain $\Omega$ are described by harmonic functions in $\Omega$ that satisfy an adequate Robin boundary condition. In Section \ref{sec2:kdv} we review some well-known facts about the mKdV hierarchy. In Section \ref{sec3:kdv} we use the mKdV hierarchy to create, along a capillary curve, an infinite recursive family of harmonic functions that satisfy the Robin condition described in Section \ref{sec1:kdv}. This will be used in Section \ref{sec4:kdv} to prove our main result: any capillary ring domain comes from an algebro-geometric mKdV potential of a certain order $\mathfrak{m}\in \N$. In Section \ref{sec5:kdv} we will give a characterization of the solutions to $\Delta u+2=0$ foliated by capillary curves as those coming from mKdV potentials with $\mathfrak{m}\leq 1$.

\subsection{Conformal deformation of Serrin ring domains}\label{sec1:kdv}
Let $\{\Omega^t : t\in (-\ep,\ep)\}$ denote a smooth $1$-parameter family of capillary ring domains, and let $u^t:\Omega^t\flecha \R$ denote their associated solutions to \eqref{overeq00}. We assume that all $\Omega^t$ are conformally equivalent. Thus, there exist holomorphic developing maps $g_t:\cU\flecha \overline\Omega^t$ from a fixed vertical quotient strip $\cU$ as in \eqref{quoban}, with $\overline\Omega^t=g_t(\cU)$. Let $v^t(x,y)$ be defined, as usual, by $v^t=u^t \circ g_t:\cU\flecha \R$, and let $\mathfrak{q}^t\in \R$ be the (constant) Hopf differential of $u^t$ in the complex parameters $(x,y)$; see Section \ref{seccapi} for the corresponding definitions.

Let $\parc _j\Omega^t$ denote the boundary component of $\parc \Omega^t$ corresponding to the curve $\{x=s_j\}\subset \parc \cU$, $j=1,2$. By the capillarity condition on $\parc \Omega$, there exist $a^t_j, b^t_j, c^t_j, (d_1)_j^t, (d_2)_j^t\in \R$, for $j=1,2$ and $t\in (-\ep,\ep)$, so that \eqref{ove1}, \eqref{def:ele} and \eqref{ove2} hold for $v^t$ along each $x=s_j$. We note that the capillary ring domains $\Omega^t$ are actually Serrin ring domains if and only if $c_1^t=c_2^t=0$ for every $t$.

We assume additionally:
\begin{enumerate}
\item
$\mathfrak{q}^t \equiv \mathfrak{q}\in \R-\{0\}$, independent of $t$.
\item
$b_j^t \equiv b_j\neq 0$ and $c_j^t\equiv c_j\in \R$ for $j=1,2$, independent of $t$.
\end{enumerate}
Then, denoting $\omega^t:= \log |g_t'|$, which is harmonic on $\cU$, we have from Proposition \ref{ekicapi} that there exist $\alfa_j,\beta_j\in \R$ with $\alfa_j\neq 0$ such that
\begin{equation}\label{skt}
2\omega_x^t (s_j,y)= -\alfa_j e^{-\omega^t(s_j,y)}-\beta_j e^{\omega^t(s_j,y)}, \hspace{0.5cm} \forall y \in \R.
\end{equation} 
Here, the constants $\alfa_j,\beta_j$ are defined in terms of $\mathfrak{q},b_j, c_j$ by \eqref{relcon}. Recall that $\omega^t$ determines $g_t$ up to rotations and translations in $\R^2$.

Denote now $\xi:=\left.\frac{d}{dt}\right|_{t=0} \omega^t$. Then, $\xi$ is a well-defined harmonic function on $\cU$ that, by \eqref{skt}, is a solution to the Robin-type problem
%
%
\begin{equation}\label{jacobi}\def\arraystretch{1.9}\left\{\begin{array}{lll} \xi_{xx}+\xi_{yy}=0 & \text{ in } & \cU, \\
\displaystyle \frac{\parc \xi}{\parc x} = \frac{1}{2}\left(\alfa_j e^{-\omega} + \beta_j e^{\omega}\right) \xi  & \text{ along } & \{x=s_j\}\subset \parc \cU, 
\end{array} \right.
\end{equation} 
for $j=1,2$. From there, the next definition is motivated by the notion of \emph{parametric Jacobi fields} for CMC surfaces by Pinkall and Sterling \cite{PS}.
\begin{definition}\label{conjafi}
A \emph{conformal Jacobi field} on a capillary ring domain $\Omega\subset \R^2$ with developing map $g:\cU\flecha \overline\Omega$ is a solution $\xi$ to problem \eqref{jacobi}, where $\omega:=\log |g'|$.
\end{definition}
Thus, conformal Jacobi fields describe, for conformal parameters $(x,y)$ in a fixed quotient band $\cU$, ``infinitesimal'' deformations of capillary ring domains that preserve the Hopf differential $\mathfrak{q}\in \R$ and the capillarity constants $b_j$, $c_j$ of $\parc \Omega$.

The notion of conformal Jacobi field in Definition \ref{conjafi} still makes obvious sense when $\Omega\subset\R^2$ is a periodic capillary (or Serrin) band.

\subsection{Recursive operators and higher order mKdV equations}\label{sec2:kdv}
The Korteweg-de Vries (KdV) equation is the evolution equation for $u=u(x,t)$ given by
\begin{equation}\label{kdv}
u_t+6u u_x+u_{xxx}=0.
\end{equation} 
It was introduced in the 19th century in \cite{KDV} as a simplified model for shallow water waves, i.e, waves travelling long distances without changing their shape. The KdV equation is fundamental to Mathematical Physics as one of the basic equations of integrable systems theory. This theory includes the study of many non-linear PDEs for which special solutions can be found by algebraic methods.

Rather than working with the KdV equation, it is better for us to work here with another classical non-linear equation: the (modified) mKdV equation, given for $\eta=\eta(x,t)$ by 
\begin{equation}\label{mkdv}
\eta_t-6 \eta^2 \eta_x+\eta_{xxx}=0.
\end{equation} 
The famous Miura transform $u=\eta_x-\eta^2$ sends solutions of \eqref{mkdv} to solutions of \eqref{kdv}. In the literature one finds different normalizations of both the KdV and mKdV equations, that change the coefficients in \eqref{kdv}, \eqref{mkdv}. All of them are equivalent up to a change of variable. We follow here the normalization in \cite{Jos}.

Both KdV and mKdV equations are often studied as the first elements of a sequence of nonlinear higher order evolution equations, called the KdV and mKdV hierarchies, respectively. All these higher order equations share KdV's integrable nature, are defined in a recursive way, and induce pairwise commuting flows. The theory regarding these objects is extensive and sophisticated, so we will restrict here to a brief, self-contained presentation strictly directed to our main objective.

We start by defining the mKdV hierarchy. We consider the sequence of polynomial operators $Q_n[\eta]$ acting on holomorphic functions $\eta=\eta(z)$ and their derivatives, and determined recursively by $Q_{-1}[\eta]=0$, $Q_0[\eta]=\eta$ and
\begin{equation}\label{mkdv3}
Q_{n+1}' -\frac{\eta'}{\eta}Q_{n+1} = Q_n''' -\frac{\eta'}{\eta} Q_n'' -4\eta^2 Q_n'.
\end{equation}
Here we are denoting $$Q_j' = Q_j'[\eta] = \frac{d}{dz}(Q_j[\eta]).$$ The solution to the linear differential equation \eqref{mkdv3} for $Q_{n+1}$ contains an additive term $k_n \eta(z)$, with $k_n$ constant. We will work with the \emph{homogeneous} mKdV hierarchy, i.e, $k_n=0$ for all $n\geq 1$.

The operators $Q_n[\eta]$ have a polynomial nature. From \eqref{mkdv3} we can easily compute the first few operators $Q_n$ as
$$Q_0 [\eta]=\eta, \hspace{0.5cm} Q_1[\eta]= \eta''-2\eta^3,\hspace{0.5cm} Q_2[\eta] = \eta^{(4)}-10 \eta'' \eta^2 -10 \eta (\eta')^2 +6 \eta^5.$$ Each $Q_n[\eta]$ is a polynomial in $(\eta,\dots, \eta^{(2n)})$ that is ``homogeneous'' of order $2n+1$ when we count the order of $\eta^{(k)}$ as $k+1$, for each $k\in \{0,\dots,n\}$.

Denoting $Q_n=Q_n[\eta]$, by our choice $k_n=0$ and \eqref{mkdv3} we have the following properties:
\begin{enumerate}
\item[(1)] If $\eta(x)\in i\R$ for every $x\in \R$, then $Q_n(x)\in i\R$ for every $x\in \R$ and every $n\in \N$.
\item[(2)] If $\eta(iy)\in \R$ for every $y\in \R$, then $Q_n(iy)\in \R$ for every $y\in \R$ and every $n\in \N$.
\end{enumerate}
For any $n\in \N$ we can define a higher-order mKdV equation as the evolution equation for $\eta=\eta(z,t)$ given by 
\begin{equation}\label{evom}
\frac{\parc \eta}{\parc t} = -Q_n'[\eta].
\end{equation} 
Note that \eqref{evom} for $n=1$ is the mKdV equation \eqref{mkdv}, and that \eqref{evom} has order $2n+1$ with respect to $z$-derivatives.

We call \emph{mKdV hierarchy} to the the sequence of these higher-order mKdV equations.

The \emph{stationary} solutions to the $n$-th mKdV equation \eqref{evom} are given by $Q_n[\eta]={\rm const}$. One obtains from \eqref{mkdv3} and $k_n=0$ that if $\eta=\eta(z)$ satisfies $Q_{n}[\eta]={\rm const}$, then $Q_{m}[\eta]=0$ for every $m>n$.

\begin{definition}\label{defi:ag}
An \emph{algebro-geometric potential} of order $\mathfrak{m}\in \N$ of the mKdV hierarchy is a function $\eta(z)$ solving 
\begin{equation}\label{agdefi}
Q_{\mathfrak{m}}[\eta] =a+\sum_{j=0}^{\mathfrak{m}-1} c_j Q_j[\eta], \hspace{0.5cm} \text{for some } a,c_0,\dots, c_{\mathfrak{m-1}} \in \C.
\end{equation}
We remark that \eqref{agdefi} for $\mathfrak{m}=0$ means $\eta(z)=a$, since $Q_{-1}[\eta]=0$ and $Q_0[\eta]=\eta$.
\end{definition}

Algebro-geometric mKdV potentials correspond to higher order stationary solutions of the (non-homogeneous) mKdV hierarchy, i.e., the hierarchy given by \eqref{evom}, \eqref{mkdv3}, this time with non-zero integration constants $k_n$. They can be expressed in terms of Riemann's theta-function on a (maybe singular) compact Riemann surface of genus $\mathfrak{m}$; see e.g. \cite{GW,IM}.

It will be useful for us to consider next the formal expression
\begin{equation}\label{kpelan}
\mathcal{Q}_{\landa} := \sum_{j=0}^{\8} Q_j[\eta] (-4\landa)^{-j-1} = \frac{Q_0[\eta]}{(-4\landa)}  +\frac{Q_1[\eta]}{(-4\landa)^2} +\frac{Q_2[\eta]}{(-4\landa)^3} + \cdots.
\end{equation}
We have then the following differential equation for $\cQ_\landa$, that follows after an easy computation from \eqref{mkdv3} and \eqref{kpelan}, using also that $Q_0[\eta]=\eta$:
\begin{equation}\label{kedopelan}
\mathcal{Q}_\landa'''-\frac{\eta'}{\eta} \cQ_\landa'' + 4(\landa-\eta^2) \mathcal{Q}_\landa ' -4\landa \frac{\eta'}{\eta} \cQ_\landa =0.
\end{equation}

\subsection{The capillary condition and the mKdV hierarchy}\label{sec3:kdv} In this section we prove that if a harmonic function $\omega(x,y)$ satisfies 
$$2\omega_x (0,y)= -\alfa e^{-\omega(0,y)}-\beta e^{\omega(0,y)}, \hspace{0.5cm} \alfa,\beta\in \R,$$
then the mKdV hierarchy generates a sequence $\{h_n\}_{n\in \N}$ of harmonic functions satisfying $$2\frac{\parc h_n}{\parc x}(0,y)=\left(\alfa e^{-\omega(0,y)} -\beta e^{\omega(0,y)}\right) h_n(0,y);$$ 
compare with \eqref{skt}. For simplicity of the proof, and only in this section, in Theorem \ref{th:finite} below we will exchange the roles of $x$ and $y$ and work along the $x$-axis, instead of working along $(0,y)$. In the next section we will go back to our usual context.
\begin{theorem}\label{th:finite}
Let $\omega(x,y)$ be a harmonic function on a domain $\Omega\subset \R^2$ such that $\Omega\cap \{y=0\}$ is a real interval $I$. Assume that 
\begin{equation}\label{bcy}
2\omega_y(x,0)= -\alfa e^{-\omega(x,0)} - \beta e^{\omega(x,0)} \hspace{0.5cm} \text{for all $x\in I$,}
\end{equation}
for some constants $\alfa,\beta\in \R$. For every $n\geq 0$, consider the harmonic function $h_n:= {\rm Re}(Q_n[\omega_z])$ on $\Omega$, where $\{Q_n\}_{n\in\N}$ are the recursive mKdV operators \eqref{mkdv3}. Then it holds 
\begin{equation}\label{bocon}
2\frac{\parc h_n}{\parc y}(x,0)= \left(\alfa e^{-\omega(x,0)} - \beta e^{\omega(x,0)}\right)\,h_n(x,0) \hspace{0.5cm} \text{for all $x\in I$.} 
\end{equation}
\end{theorem}
\begin{proof}
We introduce the notation $\eta:=\omega_z$ and $Q_n[\eta]:=h_n +i \psi_n$. Both $\eta,Q_n[\eta]$ are holomorphic on $\Omega$. Note that if $\eta(x)\equiv 0$ along $I$, then $\omega(x,y)$ is constant and the theorem trivially holds. So, we assume without loss of generality (by analyticity) that $\eta(x)\neq 0$ in $I$. Define along $I:=\Omega\cap \{y=0\}$ the real functions
\begin{equation}\label{def:fi}
s(x):=\frac{1}{2}\left(\alfa e^{-\omega(x,0)} - \beta e^{\omega(x,0)}\right), \hspace{0.5cm} t(x):=\frac{1}{2}\left(\alfa e^{-\omega(x,0)} + \beta e^{\omega(x,0)}\right).
\end{equation}
Note that 
\begin{equation}\label{derst}
s'(x)=-\omega_x (x,0)t(x), \hspace{0.5cm} t'(x)=-\omega_x(x,0) s(x).
\end{equation}
Applying Cauchy-Riemann to $Q_n[\eta]$, the relation \eqref{bocon} we seek is rewritten using \eqref{def:fi} as
\begin{equation}\label{bocon2}
\psi_n'(x)+s(x) h_n(x)=0, \hspace{0.5cm} \text{for every $n\geq 0$}.
\end{equation}
In order to prove \eqref{bocon2}, it is convenient to follow a more algebraic approach. 

Let $\Lambda_n$ denote the space of complex-valued real analytic functions $\phi:\R^{n+3}\flecha \C$, acting on ${\bf x}:=(s,t,x_0,\dots, x_n)$. We start by defining a derivation operator  $\cD:\Lambda_n\flecha \Lambda_{n+1}$ as follows: for any $\phi\in \Lambda_n$, we define $\cD\phi\in \Lambda_{n+1}$ by
\begin{equation}\label{dero}
\cD \phi({\bf x},x_{n+1}):=-x_0 \left(t \frac{\parc \phi}{\parc s}({\bf x}) +s \frac{\parc \phi}{\parc t}({\bf x})\right)+\sum_{j=0}^n x_{j+1} \frac{\parc \phi}{\parc x_j} ({\bf x}).
\end{equation}
The motivation for the definition of this operator $\cD$ is as follows: given any $\phi\in \Lambda_n$, we have in our conditions directly from \eqref{dero} and \eqref{derst} that, for any $x\in I$,
\begin{equation}\label{reldero}
\frac{d}{dx} \left(\phi\left(s(x),t(x),\omega_x(x,0),\dots, \omega_x^{(n)}(x,0)\right)\right) = \cD \phi\left(s(x),t(x),\omega_x(x,0),\dots, \omega_x^{(n)}(x,0)\right).
\end{equation}
Here, $\omega_x^{(n)}(x,0)$ denotes the $n$-th derivative of $\omega_x(x,0)$. Let us explain how this operator $\cD$ applies to our purposes. 

First, observe that by \eqref{bcy}, \eqref{def:fi} we have
\begin{equation}\label{defv}
\eta(x,0):=\omega_z(x,0)=\frac{1}{2}\left(\omega_x(x,0)+i \,t(x)\right).
\end{equation} 
Also, recall that $Q_n\equiv Q_n[\eta]$ is a polynomial in the variables $(\eta,\dots, \eta^{(2n)})$. Using then \eqref{defv} and \eqref{derst} it follows that $Q_n(x,0)$ can be written as a multivariable polynomial $Q_n\in \Lambda_{2n}$ in the variables $(s,t,x_0,\dots, x_{2n})$, where we denote
\begin{equation}\label{varpoli}
s=s(x), \hspace{0.5cm} t=t(x), \hspace{0.5cm} x_j=\omega_x^{(j)}(x,0) \hspace{0.2cm} \text{ for any $j\in \N$}.
\end{equation}
The $k$-th derivative $Q_n^{(k)}$ of the holomorphic function $Q_n[\eta]$ also satisfies that $Q_n^{(k)}(x,0)$ can be written as a polynomial $Q_n^{(k)}\in \Lambda_{2n+k}$ in the variables \eqref{varpoli}. Under this identification we obtain from \eqref{reldero} that $Q_n'(x,0)\equiv \cD Q_n$, and in the same way $Q_n''(x,0)\equiv \cD(\cD Q_n)$, etc.

For instance, when viewed as polynomials, we have from \eqref{derst} that 
\begin{equation}\label{derst2}
s'=\cD s = -x_0t, \hspace{0.5cm} t' =\cD t= -x_0 s, \hspace{0.5cm} s''=\cD(\cD s)= -x_1 t +x_0^2 s, \hspace{0.5cm} \text{etc.}
\end{equation}
So, $s',t'\in \Lambda_0$. Also, $s''\in \Lambda_1$, etc. Similarly, from \eqref{defv},
\begin{equation}\label{defv2}
\eta= \frac{1}{2}\left(x_0+i t\right)\in \Lambda_0, \hspace{0.5cm} \eta'=\cD \eta= \frac{1}{2}\left(x_1-i  x_0 s\right)\in \Lambda_1.
\end{equation}

\noindent \underline{{\bf Convention}:} For the rest of the proof of Theorem \ref{th:finite}, we will consider all quantities such as $s(x)$, $t(x)$, $Q_n(x,0)$, $\cQ_\landa(x)$, etc. as well as their $k$-th derivatives in $x$, as real analytic functions depending on the variables $(s,t,x_j)$ as in \eqref{varpoli}. 
Using \eqref{reldero}, in this convention we will identify $\frac{d}{dx} \phi(x,0)\equiv \cD\phi$ for any such function $\phi(x,0)$. Thus, we will use $\phi'$ to mean $\cD\phi$, and $\phi''$ to mean $\cD(\cD\phi)$, etc. For instance, we will write (see \eqref{kpelan}) $$\cQ_\landa' =\cD (\cQ_\landa)=\sum_{j=0}^\8 \cD Q_j (-4\landa)^{-j-1},$$ and the differential equation \eqref{kedopelan} holds formally (i.e., term by term) with this convention.


We can now continue with the proof of Theorem \ref{th:finite}. Recall that our objective is to show \eqref{bocon2}. In our multivariable notation, $\psi_n'+s \,h_n$ is an element of $\Lambda_{2n+1}$, and \eqref{bocon2} is written as 
\begin{equation}\label{bocon20}
\psi_n'+s \,h_n=0, \hspace{0.5cm} \text{for every $n\geq 0$}.
\end{equation}
Let us express this condition as a formal series. For that, we define 
\begin{equation}\label{def:fila}
\Phi_\landa:= \left(\cQ_{\landa}'+ i s\cQ_{\landa}\right).
\end{equation} 
Here, $\cQ_{\landa}$ is the formal series in \eqref{kpelan}, where with our convention $Q_n\in \Lambda_{2n}$. Denoting 
\begin{equation}\label{defjota}
\cJ_{\landa}:={\rm Im}(\Phi_{\landa}),
\end{equation} 
we have then $$\cJ_{\landa}=\sum_{j=0}^{\8} \frac{\psi_j'+s\, h_j}{(-4\landa)^{j+1}} \hspace{0.5cm}\left(:=\sum_{j=0}^\8 \frac{J_j}{(-4\landa)^{j+1}}\right).$$ So, in particular, \eqref{bocon20} is equivalent to 
\begin{equation}\label{bocon3}
\cJ_{\landa} \equiv 0.
\end{equation}

To prove \eqref{bocon3} we will show that $\cJ_\landa$ satisfies a (formal) linear homogeneous differential equation; see \eqref{odeseries2} below. Our next objective is to arrive at \eqref{odeseries2}.


We first observe that, by \eqref{kedopelan}, we can express $\cQ_{\landa}'''$ as a linear combination of $\cQ_\landa,\cQ_\landa',\cQ_\landa''$, with coefficients depending only on $\eta,\eta'$ as in \eqref{defv2}, and $\landa$. 
Therefore, if we consider \eqref{def:fila} together with its first and second derivatives, we obtain 
\begin{equation}\label{sistemq}
\left(\def\arraystretch{1.3}\begin{array}{c} \Phi_\landa \\ \Phi_\landa' \\ \Phi_\landa'' \end{array}\right) = \left(\def\arraystretch{1.3}
\begin{array}{ccc} 
i s & 1 & 0 \\ i s' & i s & 1 \\ 4 \landa \frac{\eta'}{\eta}+is'' & -4\landa + 4\eta^2+2is' & \frac{\eta'}{\eta} +is.
\end{array}\right)\left(\def\arraystretch{1.3}\begin{array}{c} \cQ_\landa \\ \cQ_\landa' \\ \cQ_\landa''\end{array}\right).\end{equation}
Let $L$ be the coefficient matrix in \eqref{sistemq}. From \eqref{derst2} and \eqref{defv2} we can compute directly 
\begin{equation}\label{detel}
{\rm det} (L) = \frac{(4\lambda-s^2+t^2) \left(x_1 -s t\right)}{2 \eta}.
\end{equation} 
We will always assume without loss of generality that $\eta\in \Lambda_0$ has no zeros; see our comment at the beginning of the proof of Theorem \ref{th:finite}.

We consider first the special case ${\rm det}(L)\equiv 0$ for all $\landa$. This condition is equivalent to 
\begin{equation}\label{odete}
x_1=st,
\end{equation} 
which is itself equivalent by \eqref{defv2} to $\eta'+i s \eta=0$. Since $Q_0=\eta$, we have $Q_0'+ i s Q_0= 0$. We next prove by induction that in this case $\Phi_{\landa}\equiv 0$, which obviously implies \eqref{bocon3}. Assume that 
\begin{equation}\label{induq}
Q_n'+i s Q_n = 0
\end{equation}
for some $n\in \N$. Using that $\eta'+i s \eta=0$ as explained above, we obtain from \eqref{mkdv3} that $$Q_{n+1}'+ i s Q_{n+1} = Q_n''' + i s Q_n'' -4 \eta^2 Q_n'.$$ Using then the induction hypothesis \eqref{induq} and its derivatives together with \eqref{odete}, 
\eqref{derst2} and \eqref{defv2}, we deduce that the right hand-side of this equation is equal to zero. Consequently, if ${\rm det}(L)\equiv 0$, then $\Phi_{\landa}\equiv 0$, and in particular \eqref{bocon3} holds.

So, we assume in what follows that ${\rm det}(L)\not\equiv 0$. By analyticity, without loss of generality, we assume then that ${\rm det}(L)\neq 0$ at every point. 
Thus, we can invert \eqref{sistemq} to obtain 
\begin{equation}\label{invele}
(\cQ_\landa,\cQ_\landa',\cQ_\landa'')^t=L^{-1}\cdot (\Phi_\landa,\Phi_\landa',\Phi_\landa'')^t.
\end{equation} 
Differentiating the last row of \eqref{sistemq} and using \eqref{kedopelan}, one can express again $\Phi_\landa'''$ as a linear combination of $\cQ_\landa,\cQ_\landa',\cQ_\landa''$, namely,
\begin{equation}\label{fitercera}
\Phi_\landa'''= \cB_0 \cQ_\landa+ \cB_1 \cQ_\landa' + \cB_2 \cQ_\landa'',
\end{equation} 
where 
$$\def\arraystretch{1.7}\begin{array}{lcl}
\cB_0 & =& is'''+\frac{4\landa(\eta''+is \eta')}{\eta},\\\cB_1 \ & = & 3is''+12 \eta \eta'-4i(\landa-\eta^2)s,\\
\cB_2 & = & 3is' +\frac{\eta'+is \eta'}{\eta}+4\eta^2 -4\landa.
\end{array}$$
Thus, by \eqref{invele},
\begin{equation}\label{fitercera2}
\Phi_\landa'''= (\cB_0,\cB_1,\cB_2)\cdot L^{-1} \cdot (\Phi_\landa,\Phi_\landa',\Phi_\landa'')^t.
\end{equation}
Denote $$\cA_3:= \frac{8\eta \,{\rm det}(L)}{4\landa + t^2-s^2}= 4(x_1- s\,t)\in \Lambda_1,$$ see \eqref{detel}. Then, \eqref{fitercera2} yields 
a third order ODE for the formal series $\Phi_\landa$, of the form 
\begin{equation}\label{odeseries}
\cA_3 \Phi_\landa ''' + \cA_2 \Phi_\landa''+ \cA_1 \Phi_\landa' + \cA_0 \Phi_\landa =0.
\end{equation}
The coefficients $\cA_k$ can be computed from \eqref{derst2}, \eqref{defv2}; we obtain
$$\def\arraystretch{1.5}\begin{array}{lcl}
\cA_2 & =& x_2+4x_0\left(s^2+t^2\right),\\\cA_1\ & = & -4(x_1-s\,t)\left(4\landa+t^2-x_0^2\right),\\
\cA_0 & = & -\cA_1' + 4x_0 (x_1-s \,t)(5x_1+ 3 s\, t),
\end{array}$$
where to compute $\cA_1'=\cD \cA_1$ we use \eqref{derst2}. Of these expressions, we will only need that each $\cA_k$ can be written as 
\begin{equation}\label{decoma}
\cA_k=\hat{A}_k + \landa A_k,
\end{equation} 
$k=0,\dots, 3$, where $A_k, \hat{A}_k$ do not depend on $\landa$, and:
\begin{enumerate}
\item[(P1)]
All $\cA_k$ are real. 
\item[(P2)]
$A_2=A_3=0$.
\item[(P3)]
$A_1=-16\left(x_1-s\,t\right)\in \Lambda_1$.
\item[(P4)]
$A_0=-A_1'=16\left(x_2+x_0(s^2+t^2)\right)\in \Lambda_2.$
\end{enumerate}
Now, we recall that $\cJ_{\landa}={\rm Im}(\Phi_\landa)$, see \eqref{defjota}. Since all coefficients $\cA_k$ are real by (P1), we obtain from \eqref{odeseries} and \eqref{decoma} that 
\begin{equation}\label{odeseries2}
\cA_3 \cJ_{\landa}''' + \cA_2 \cJ_{\landa}''+\cA_1 \cJ_{\landa}'+\cA_0 \cJ_{\landa}=0.
\end{equation}

Once we have arrived to the linear differential equation \eqref{odeseries2}, we will see how to obtain \eqref{bocon2} from it. We will proceed by induction.

That \eqref{bocon2} holds for $n=0$ is easy. First, we recall that $\eta=Q_0=h_0+i\psi_0$. Then, we use \eqref{derst2}, \eqref{defv2} to obtain $\psi_0'+s\, h_0=0$, as desired. So, the first term in $\cJ_\landa$ is zero.

Arguing by induction, assume that the first $n$ terms of $\cJ_{\landa}$ vanish, that is, $$\cJ_{\landa}=\sum_{k=n+1}^{\8} \frac{J_k}{(-4\landa)^{k+1}}.$$ Then, we can use \eqref{decoma} to decompose \eqref{odeseries2} in terms of $\landa$. Using then the relations $A_2=A_3=0$ in (P2), we obtain from \eqref{odeseries2} that 
\begin{equation}\label{edoj}
A_0 J_{n+1} + A_1 J'_{n+1}=0.
\end{equation} 
Using now that $A_0=-A_1'$ by (P4) above, \eqref{edoj} yields 
\begin{equation}\label{edoj3}
\left(\frac{J_{n+1}}{A_0}\right)'=0, \hspace{0.5cm} \text{i.e.,} \hspace{0.5cm} \cD\left(\frac{J_{n+1}}{A_0}\right)=0.
\end{equation}
To solve \eqref{edoj3} we use the following elementary lemma for the derivation operator $\cD$.
\begin{lemma}
If $\cD\phi=0$  for some $\phi\in \Lambda_{n}$, then $\phi=f(s^2-t^2)$ for some function $f(x)$.
\end{lemma}
\begin{proof}
Let $j\in \N$ be the largest integer so that $\frac{\parc\phi}{\parc x_j}\neq 0$. Then, from \eqref{dero}, $\cD\phi$ depends linearly on $x_{j+1}$ in a non-trivial way. Thus, since $\cD\phi=0$, we deduce that $\phi$ cannot depend on any $x_j$, i.e. $\phi=\phi(s,t)$. Using again \eqref{dero}, we obtain from $\cD\phi=0$ that $t \phi_s+s\phi_t=0$, i.e. $\phi(s,t)=f(s^2-t^2)$ for some function $f(x)$.
\end{proof}

From this lemma and \eqref{edoj3}, we deduce that 
\begin{equation}\label{relconf}
J_{n+1}=f(s^2-t^2)A_0,
\end{equation} for some real function $f(x)$. So, we need to prove that $f(x)= 0$. For that purpose, we will first compare the ``homogeneity'' orders  of $J_{n+1}$ and $A_0$. Let us explain this concept.

We will say that a polynomial $\phi\in \Lambda_n$ in the $(s,t,x_0,\dots, x_n)$-variables is \emph{KdV-homogeneous} if each of its monomials is ``homogeneous'' of the same order, when we count orders so that $s,t$ have order one and $x_j$ has order $j+1$. This notion is consistent with the homogeneity of the mKdV operators $Q_n[\eta]$ explained in Section \ref{sec2:kdv}., and the relation \eqref{reldero}. Then we have
\begin{lemma}
If $\phi\in \Lambda_n$ is a KdV-homogeneous polynomial of order $k$, then $\cD\phi\in \Lambda_{n+1}$ is a KdV-homogeneous polynomial of order $k+1$.
\end{lemma}
\begin{proof}
Direct from \eqref{dero}.
\end{proof}
Note that $Q_0=\eta$ is KdV-homogeneous of order $1$, by \eqref{defv2}. By the homogeneity properties of the mKdV operators $Q_n[\eta]$, we deduce:
\begin{corollary}
Each polynomial $Q_n\in \Lambda_{2n}$ is KdV-homogeneous of order $2n+1$.
\end{corollary}
From the above corollary, $J_{n+1}:=\psi_{n+1}'+s \, h_{n+1}$ is KdV-homogeneous of order $2n+3$, and $A_0$ has order $3$, see (P4). So, by \eqref{relconf}, we see that $f(x)$ must be a \emph{standard} real homogeneous function of order $n$, i.e. $f(x)=c_n x^n$ for some $c_n\in \R$. Therefore,  \eqref{relconf} and (P4) give
\begin{equation}\label{relconf2}
\psi_{n+1}'+s \, h_{n+1}=16 c_n \,(s^2-t^2)^n\,\left(x_2+x_0(s^2+t^2)\right),
\end{equation}
for some $c_n\in \R$. If $c_n\neq 0$, the right-hand side of \eqref{relconf2} has a non-zero monomial term proportional to $x_0\, t^{2n+2}$. Let us check that this type of monomial cannot happen in the left-hand side of \eqref{relconf2}. This is obvious for $s\, h_{n+1}$, due to the $s$ factor.

So, let $\mu$ be a monomial of $\psi_{n+1}$, and assume that one of the monomials of $\mu'=\cD\mu$ is of the form $x_0\, t^{2n+2}$, multiplied by some constant. By the definition of the operator $\cD$, one easily sees that $\mu$ cannot contain any $x_j$-term, and moreover, it must be of the form 
\begin{equation}\label{estono}
\mu =a_0 \,s \,t^{2n+1},
\end{equation}
for some $a_0\neq 0$. However, $\psi_{n+1}$ does not contain monomials with this structure. To see this, we recall that $\psi_{n+1}={\rm Re}(Q_{n+1})$, and that $Q_{n+1}[\eta]$ is a polynomial depending on $\eta,\dots, \eta^{(2n+2)}$. Also, $s$ does not appear in the expression of $\eta$ in \eqref{defv2}. Hence, the way $s$ appears in $Q_{n+1}$ when viewed as polynomial in $\Lambda_{2n+2}$ is associated to $\eta^{(k)}$ for some $k>0$, after using $\cD t=-x_0 s$ some number of times; see \eqref{derst2}. This means that $s$ always appears in $Q_{n+1}$ multiplied by some of the variables $x_j$. So, the structure \eqref{estono} is impossible for any monomial coming from $\psi_{n+1}$. Therefore, the left-hand side of \eqref{relconf2} has no monomial term of the form $a_0\, x_0\, t^{2n+2}$, with $a_0\neq 0$.

This implies that $c_n=0$ in \eqref{relconf2}, that is, $J_{n+1}=0$. By induction, this shows that \eqref{bocon3} holds. Thus, we obtain \eqref{bocon2}, and so \eqref{bocon}. This completes the proof of Theorem \ref{th:finite}.
\end{proof}

%

%
%

\begin{remark}
First order conditions such as \eqref{bcy} are known in the theory of integrable systems as \emph{integrable boundary conditions}; see e.g. \cite{AGGH,Kha} and references therein. Theorem \ref{th:finite} is inspired by the finite-type property of capillary CMC annuli by Kilian-Smith \cite{KS}. Our approach here, based on recursive operators and the mKdV hierarchy, is different.
\end{remark}

\subsection{Capillary ring domains are algebro-geometric}\label{sec4:kdv}
For Theorem \ref{th:finite2} below, we recall the notion of \emph{capillary ring domain} $\Omega\subset \R^2$ (Definition \ref{def:capicu}), we let $g:\cU\flecha \overline\Omega$ denote its developing map (Definition \ref{develop}), and denote as usual $\omega:=\log |g'|$. We also denote by $b_j$ the capillarity constant (Definition \ref{capilarcons}) of the boundary curve $\parc_j \Omega$, for $j=1,2$. Finally, we recall the notion of \emph{conformal Jacobi field} for $\Omega$ in Definition \ref{conjafi}.

Since Serrin ring domains are capillary, Theorem \ref{th:finite2} below also applies to them.
\begin{theorem}\label{th:finite2}
Let $\Omega$ be a capillary ring domain with capillarity constants $b_j\neq 0$. For any $n\geq 0$ let $h_n:\cU\flecha \R$ be the harmonic function $h_n:={\rm Im} \,Q_n[\omega_z]$, where $\{Q_n\}_{n\in\N}$ are the recursive mKdV operators \eqref{mkdv3}. Then every $h_n$ is a conformal Jacobi field for $\Omega$.

As a consequence, there exist $\mathfrak{m}\in \N$ and constants $c_0,\dots, c_{\mathfrak{m}-1}\in \R$ such that 
\begin{equation}\label{finih}
h_{\mathfrak{m}}=\sum_{j=0}^\mathfrak{m-1} c_j h_j.
\end{equation}
(For $\mathfrak{m}=0$ this means $h_0=0$). Moreover, if \eqref{finih} holds for some $\mathfrak{m}$, then for any $k\geq 0$ the harmonic function $h_{\mathfrak{m}+k}$ is also a linear combination of $h_0,\dots, h_\mathfrak{m-1}$.
\end{theorem}
\begin{proof}
By the capillarity conditions along $\parc \Omega$ and the assumption $b_j\neq 0$, we see from Proposition \ref{ekicapi} that, along each curve $(s_j,y)$ in $\parc \cU$, there exist constants $(\alfa_j,\beta_j)$ such that \eqref{capicon} holds, $j\in \{1,2\}$. By Theorem \ref{th:finite} (interchanging the roles of $x$ and $y$ in it), every $h_n={\rm Im}\, Q_n[\omega_z]$ is a solution on the quotient band $\cU$ in \eqref{quoban} to the Robin problem
\begin{equation}\label{jacobi2}\def\arraystretch{1.9}\left\{\begin{array}{lcl} \Delta h_n =0 & \text{ in } & \cU, \\
\displaystyle \frac{\parc h_n}{\parc x} = \frac{1}{2}\left(\alfa_j e^{-\omega} + \beta_j e^{\omega}\right) h_n  & \text{ along } & \{x=s_j\}\subset \parc \cU, 
\end{array} \right.
\end{equation} 
In other words, every $h_n$ is a conformal Jacobi field for $\Omega$; see Definition \ref{conjafi}.

Now, by elliptic theory, and since the conformal annulus $\cU$ is compact, the space of solutions to \eqref{jacobi} for fixed constants $\alfa_j,\beta_j$ is finite-dimensional. Thus, from \eqref{jacobi2}, we deduce that there exists $\mathfrak{m}\geq 0$ so that \eqref{finih} holds. Denoting $Q_n:=Q_n[\omega_z]$, we obtain
\begin{equation}\label{hipin0}
Q_{\mathfrak{m}}= a_0 + \sum_{j=0}^\mathfrak{m-1} c_j Q_j.
\end{equation} 
for some $a_0\in \R$. That is, $\eta=\omega_z$ is an algebro-geometric mKdV potential (Definition \ref{defi:ag}).


Finally, arguing by induction, assume $h_{\mathfrak{m}+k}=\sum_{j=0}^{\mathfrak{m}-1} \hat{c}_j h_j$ with $\hat{c}_0,\dots, \hat{c}_{\mathfrak{m}-1} \in \R$, for some $k\in \N$. Then,
\begin{equation}\label{hipin}
Q_{\mathfrak{m}+k}=a_k+\sum_{j=0}^{\mathfrak{m}-1} \hat{c}_j Q_j,
\end{equation} 
for some $a_k\in \R$. Using \eqref{hipin0}, \eqref{hipin} and the recursive relation \eqref{mkdv3} we have
$$\def\arraystretch{2.6}\begin{array}{lll}
\displaystyle Q'_{\mathfrak{m}+k+1} -\frac{\eta'}{\eta} Q_{\mathfrak{m}+k+1} & = & \displaystyle \sum_{j=0}^{\mathfrak{m}-1} \hat{c}_j \left(Q_j''' -\frac{\eta'}{\eta} Q_j'' -4\eta^2 Q_j'\right) \\ & = & \displaystyle \sum_{j=0}^{\mathfrak{m}-1} \hat{c}_j \left(Q'_{j+1}-\frac{\eta'}{\eta} Q_{j+1}\right) \\ \displaystyle & = & \displaystyle  \sum_{j=0}^{\mathfrak{m}-1} d_j \left(Q_j'-\frac{\eta'}{\eta} Q_j \right) -a_0 \,\hat{c}_{\mathfrak{m}-1}\, \frac{\eta'}{\eta'}.
\end{array}$$ Here, $d_j:= \hat{c}_{\mathfrak{m}-1} \, c_j + \hat{c}_{j-1}\in \R$ for $j\in 0,\dots, \mathfrak{m}-1$, denoting $\hat{c}_{-1}=0$.

Thus, $H:= Q_{\mathfrak{m}+k+1} -\sum_{j=0}^{\mathfrak{m}-1} d_j Q_j$ satisfies $H'-\frac{\eta'}{\eta}H= -a_0 \,\hat{c}_{\mathfrak{m}-1} \, \frac{\eta'}{\eta},$ which integrates to $H= a_0 \,\hat{c}_{\mathfrak{m}-1} + \delta \eta$ for some $\delta \in \R$. Noting finally that $Q_0=\eta$, we obtain $$Q_{\mathfrak{m}+k+1}= a_0 \,\hat{c}_{\mathfrak{m}-1} + \sum_{j=0}^{\mathfrak{m}-1} \hat{d}_j Q_j,$$ for adequate constants $\hat d_0,\dots, \hat d_{\mathfrak{m}-1}\in \R$. Therefore, $h_{\mathfrak{m}+k+1}=\sum_{j=0}^{\mathfrak{m}-1} \hat{d}_j h_j$. This finishes the proof.
%
%
\end{proof}

Recall that $\Delta u+2=0$ has the (trivial) paraboloid solutions in $\R^2$
\begin{equation}\label{capo}
u({\bf x}):=c -\frac{1}{2}\| {\bf x}-v_0\|^2,  \hspace{0.5cm} c\in \R,\ v_0\in \R^2,
\end{equation} 
and these are characterized by having zero Hopf differential, see Remark \ref{rem:arriba}. As a consequence of Theorem \ref{th:finite2}, we obtain:



\begin{corollary}\label{cor:finite}
Let $\Omega\subset\R^2$ be a Serrin ring domain, or more generally, a capillary ring domain, whose associated solution $u$ is \emph{not} a paraboloid \eqref{capo}. Let $g:\cU\flecha \overline\Omega$ be its developing map. Then, the meromorphic function 
\begin{equation}\label{etapote}
\eta(z):=\frac{g''(z)}{2g'(z)}
\end{equation} 
is an algebro-geometric potential for the meromorphic mKdV hierarchy.
\end{corollary}
\begin{proof}
Let $b_j$ denote the capillarity constants of $\Omega$ along the boundary curves $x=s_j$ of $\parc \cU$, for $j=1,2$. If both $b_j\neq 0$, the result follows directly from Theorem \ref{th:finite2}, by noting that $2\omega_z=g''/g'$; see equation \eqref{hipin0}.

Assume now that, say, $b_1=0$ and $b_2\neq 0$. Then, by \eqref{ecube} and $\mathfrak{q}\neq 0$, we see that $y\mapsto \omega(s_1,y)$ is constant. Denote $\eta:=\omega_z=\frac{g''}{2g'}$. Then, $\eta(s_1+iy)\in \R$ for every $y$. As a consequence, $h_n:={\rm Im} \, Q_n[\eta]$ also satisfies $h_n(s_1+iy)\in \R$, for every $y$; see property (2) of the mKdV operators $Q_n$ before \eqref{evom}. Thus, $h_n$ solves the mixed Robin-Dirichlet problem (compare with \eqref{jacobi2})
\begin{equation}\label{jacobi4}\def\arraystretch{1.9}\left\{\begin{array}{lcl} \Delta h_n =0 & \text{ in } & \cU, \\\displaystyle h_n  = 0  & \text{ along } & \{x=s_1\}\subset \parc \cU, \\
\displaystyle \frac{\parc h_n}{\parc x} = \frac{1}{2}\left(\alfa_2 e^{-\omega} + \beta_2 e^{\omega}\right) h_n  & \text{ along } & \{x=s_2\}\subset \parc \cU, 
\end{array} \right.
\end{equation} 
on the compact quotient band $\cU$, for some $\alfa_2,\beta_2\in \R$. The space of solutions to \eqref{jacobi4} is again finite dimensional by elliptic theory. Thus, both \eqref{finih} and \eqref{hipin0} hold, and $\eta$ in \eqref{etapote} is an algebro-geometric mKdV potential.

Assume finally that $b_1=b_2=0$. Then $h_0:={\rm Im}(\omega_z)$ is identically zero, since $h_0=0$ in $\parc\cU$. Thus, $\eta=Q_0[\eta]=\omega_z$ is a real constant $a_0$, and so an algebro-geometric mKdV potential, see Definition \ref{defi:ag}. This completes the proof.
\end{proof}

We remark that, in the case of Serrin ring domains, if one of the capillarity constants $b_j$ is zero, then $\Omega$ is a radial annulus. This follows from the uniqueness results of Reichel \cite{Rei}, Agostiniani, Borghini and Mazzieri \cite{ABM} and Borghini \cite{Bo}.

\begin{remark}
Theorem \ref{th:finite2} and Corollary \ref{cor:finite} also hold for periodic Serrin bands and, more generally, for periodic capillary bands, as we detail next. 

Let $\Omega\subset \R^2$ be a capillary band (Definition \ref{def:capicu}). We assume $\Omega$ is periodic, i.e., $\Omega=\Omega +{\bf v}_0$ for some ${\bf v}_0\in \R^2-\{{\bf 0}\}$, and its associated solution $u$ to $\Delta u+2=0$ in $\Omega$ is also periodic, i.e., $u({\bf x}+{\bf v}_0) = u({\bf x})$ for all ${\bf x}\in \Omega$. Let $g(z):[s_1,s_2]\times \R\flecha \overline\Omega$ be its developing map (Definition \ref{develop}), and $\omega:=\log |g'|$. We assume that $u(x_1,x_2)$ is not a (trivial) paraboloid solution \eqref{capo}, i.e., the Hopf differential $\mathfrak{q}$ is non-zero.


If the capillarity constants $b_1,b_2$ are both non-zero, then the proof of Theorem \ref{th:finite2} shows that each $h_n:=Q_n[\omega_z]$ is a conformal Jacobi field, i.e., it solves \eqref{jacobi} on $[s_1,s_2]\times \R$. Moreover, by our periodicity assumptions, $\omega(x,y)$ is $T$-periodic in $y$ for some $T>0$, and all conformal Jacobi fields $h_n$ are well defined on the compact quotient band $\cU$ in \eqref{quoban}, solving \eqref{jacobi2}. The rest of the proof of Theorem \ref{th:finite2} is identical.

The extension of Corollary \ref{cor:finite} follows the same argument using this time \eqref{jacobi4}.
\end{remark}

\subsection{Solutions foliated by capillary curves: characterization}\label{sec5:kdv}
We have seen in Section \ref{sec4:kdv} that if $\Omega\subset\R^2$ is a capillary ring domain or a periodic capillary band with non-zero Hopf differential $\mathfrak{q}$ (i.e. $u$ is not a paraboloid), then the relation \eqref{finih} holds for some $\mathfrak{m}\in \N$. 
\begin{definition}
We call \emph{spectral genus} of the capillary domain $\Omega$ to the smallest $\mathfrak{m}\in \N$ for which \eqref{finih} holds for $\Omega$.
\end{definition}
The study of such capillary domains $\Omega$ with spectral genus $\mathfrak{m}\geq 2$, is quite involved, as the solutions to \eqref{finih} are described in terms of algebraic data on a hyperelliptic curve of genus $\mathfrak{m}$; see e.g. \cite{GW,IM}. When $\mathfrak{m}\leq 1$ the situation is simpler, and we can obtain a full satisfactory characterization:

\begin{theorem}\label{th:n1}
Let $\Omega\subset \R^2$ be a capillary ring domain or a periodic capillary band with $\mathfrak{q}\neq 0$. 
Then, $\Omega$ is foliated by capillary curves (Definition \ref{def:capring}) if and only if the spectral genus $\mathfrak{m}$ of $\Omega$ is $\mathfrak{m}\leq 1$.
\end{theorem}
We will actually prove a more general result, of a local nature (Theorem \ref{th:n3} below), from which Theorem \ref{th:n1} follows immediately.

For that, let $u$ be a solution to $\Delta u+2=0$ on a smooth domain $\Omega\subset \R^2$ for which there exists a conformal eigenline parametrization $$g(z): \cV=(s_1,s_2)\times (t_1,t_2)\flecha \Omega$$ with $\mathfrak{q}\neq 0$; see Definition \ref{eigenline}. We denote $z=x+iy$ and $\omega=\log |g'|$. So, in particular, all $y$-curves $y\mapsto g(x_0+iy)$ are Hessian eigencurves for $u$. 

Consider then the hierarchic sequence of harmonic functions $h_n:={\rm Im} \,Q_n[\omega_z]$, where $\{Q_n\}_{n\in\N}$ are the recursive mKdV operators \eqref{mkdv3}.


\begin{theorem}\label{th:n3}
$\Omega$ is foliated by capillary curves (Definition \ref{def:capring}) if and only if $\Omega$ has at least one capillary $y$-curve with capillarity constant $b\neq 0$, and \eqref{finih} holds on $\Omega$ for $\mathfrak{m}\leq 1$.
\end{theorem}
\begin{proof}
We divide the proof into two cases.
\vspace{0.1cm}

\underline{\emph{The case $\mathfrak{m}=0$}:} In that situation $h_0={\rm Im}(\omega_z)=0$, i.e., we have $\omega(x)=ax+b$ for constants $a,b\in \R$. We distinguish two possibilities.

First, assume that $a=0$, i.e., $\omega\equiv b$, constant. Then, from $|g'|^2=e^{2\omega}$ we obtain, up to translations and rotations, that $g(z)=e^b\, z$. Observe now that $\mathfrak{q}=v_{zz}-2\omega_z v_z =v_{zz}\in \R$, and recall that $\Delta v +2e^{2\omega}=0$. From here, we easily deduce that $v(x,y)$, and so $u:=v\circ g^{-1}$, is a polynomial of degree two. Specifically, 
\begin{equation}\label{cua0}u(x_1,x_2)=a_1 x_1^2 + a_2 x_2^2 + L(x_1,x_2),
\end{equation} 
with $L$ linear and $a_1+a_2=-1$. In particular, it is immediate that all $x_1$-curves and all $x_2$-curves are capillary curves for $u$, by Lemma \ref{joach}; indeed, each such curve is a Hessian eigencurve of $u$, and $u=P$ along it for some polynomial $P$ as in \eqref{cupo}. For instance, $u(x_1,b)=P(x_1,b)$, for $P=a_1(x_1^2+x_2^2) + (a_2-a_1)b^2 +L(x_1,x_2)$.


Assume next that $a\neq 0$. Again, after rotations and translations, we have $g(z)=c e^{az}$ for some $c>0$. We next use that $v(x,y)$ can be written as $v=\psi -|g|^2/2$, with $\psi(x,y)$ a harmonic function; see \eqref{rep}. Then, $$\mathfrak{q}=v_{zz}-2\omega_z v_z = \psi_{zz} -a \psi_z.$$ Therefore, $\psi_z= \frac{-\mathfrak{q}}{a} + d e^{az}$, for some constant $d$, and thus $$v(x,y)=-\frac{2\mathfrak{q}}{a} x + 2ad \,e^{a x} \cos y - \frac{c^2}{2} e^{2ax}.$$ From here it is easy to check that $u:=v\circ g^{-1}$ is this time of the form 
\begin{equation}\label{ra0}
u(x_1,x_2)= c_1 \log r - \frac{r^2}{2} + L(x_1), \hspace{0.5cm} r:=\sqrt{x_1^2+x_2^2},
\end{equation} 
where $c_1\neq 0$ and $L(x_1)$ is linear. That is, up to the addition of a linear term, $u$ is radial. In particular, $u$ is again foliated by capillary curves. 

\vspace{0.1cm}

\underline{\emph{The case $\mathfrak{m}=1$}:} We start with some preparation for the proof. First, note that $h_0={\rm Im} (\omega_z)=-\omega_y/2$. Using that $\omega$ is harmonic, we can compute $h_1= {\rm Im}(Q_1[\omega_z])={\rm Im}(\omega_{zzz}-2\omega_z^3)$ as
\begin{equation}\label{cara1}
4h_1=2\omega_{yyy}+3 \omega_x^2\omega_{y}- \omega_y^3.
\end{equation}
We assume from now on that $\omega_y$ is not identically zero; indeed, note that if $\omega_y\equiv 0$, then $h_0=h_1\equiv 0$ and we are in the case $\mathfrak{m}=0$ treated above.

Let now $y\mapsto g(x_0+iy)$ be a capillary curve with capillarity constant $b\neq 0$. By Proposition \ref{ekicapi}, there exist $\alfa,\beta\in \R$, $\alfa\neq 0$, so that
\begin{equation}\label{cara20}
2 \omega_x (x_0,y)=-\alfa e^{-\omega(x_0,y)} -\beta e^{\omega(x_0,y)}
\end{equation} 
holds for every $y$. Now, differentiate \eqref{cara20} with respect to $y$, and compute the values of $\alfa,\beta$ in the resulting linear system of this equation and \eqref{cara20}, to obtain along $y\mapsto (x_0,y)$
\begin{equation}\label{caraal}
\alfa= e^{\omega}\left(\frac{\omega_{xy}}{\omega_y}-\omega_x\right), \hspace{0.5cm} \beta= -\frac{e^{-\omega} (\omega_{xy}+\omega_x\omega_y)}{\omega_y}.
\end{equation} 
Next, differentiate any equation in \eqref{caraal} again with respect to $y$, to obtain along $y\mapsto (x_0,y)$
\begin{equation}\label{cara0}
\omega_y \omega_{xyy} = \omega_x \omega_y^3 +\omega_{xy}\,\omega_{yy}.
\end{equation} 

We now proceed with the proof. Assume first that $u$ is foliated by capillary curves, as in Definition \ref{def:capring}. So, we have
\begin{equation}\label{cara2}
2 \omega_x=-\alfa(x) e^{-\omega} -\beta(x) e^{\omega},
\end{equation} 
for functions $\alfa(x),\beta(x)$, with $\alfa(x)\not\equiv 0$. In particular, \eqref{cara0} holds for every $(x,y)$. 
We next differentiate \eqref{cara1} with respect to $x$, and use \eqref{cara0} and its $y$-derivative to show that $$\omega_y (h_1)_x  - \omega_{xy}\, h_1=0,$$ which means that $h_1 = K(y) \omega_y$, for some function $K(y)$. Since $h_1,\omega_y$ are both harmonic, this implies that $K''\omega_y +2K' \omega_{yy}=0$. Assuming that $K(y)$ is not constant, this gives $(\omega_{yy}/\omega_y)_x=0$, that is, $\omega_y \omega_{xyy}=\omega_{xy}\omega_{yy}.$ By \eqref{cara0}, this implies that $\omega_x=0$. But now, this would imply from \eqref{cara2} that $\omega$ is constant, which is not impossible (we are assuming $\omega_y\neq 0$).

So $K(y)$ is constant, and since $\omega_y =-2h_0$, we have $h_1=c_0 h_0$ for some $c_0\in \R$.  That is, \eqref{finih} holds for $\mathfrak{m}= 1$. This proves the first implication.

For the converse, we reverse the process, as follows. Assume that \eqref{finih} holds for $\mathfrak{m}= 1$, that is, $h_1= c_0 h_0$ for some $c_0\in \R$. This lets us obtain from \eqref{cara1} the values of $\omega_{yyy}$ and, by differentiation, of $\omega_{xyyy}$ and $\omega_{yyyy}$, in terms of lower order derivatives and $c_0$. A computation using these values and the harmonicity of $\omega$ shows then that the derivative of \eqref{cara0} with respect to $y$ vanishes, that is, 
\begin{equation}\label{cara40}
\omega_y \omega_{xyy} = \omega_x \omega_y^3 +\omega_{xy}\,\omega_{yy}+K(x),
\end{equation} 
for some function $K(x)$. Now, we differentiate \eqref{cara40} with respect to $x$, and use in the resulting equation the harmonicity of $\omega$ with the expressions obtained above for $\omega_{yyy}, \omega_{xyyy}, \omega_{yyyy}$, and \eqref{cara40} itself. A somewhat lengthy but direct computation gives in this way that $K'(x)=0$. So,
\begin{equation}\label{cara4}
\omega_y \omega_{xyy} = \omega_x \omega_y^3 +\omega_{xy}\,\omega_{yy}+k,
\end{equation} 
holds for some $k\in \R$. We now recall that, by hypothesis, $\Omega$ has at least one capillary curve $y\mapsto g(x_0+iy)$ with $b\neq 0$. In particular, \eqref{cara20}, and so \eqref{cara0}, hold along $y\mapsto (x_0,y)$. This implies therefore that $k=0$ in \eqref{cara4}, i.e., \eqref{cara0} holds for all $(x,y)$. We deduce from \eqref{cara0} that there exist functions $\alfa(x),\beta(x)$ so that \eqref{caraal} is satisfied. Then, \eqref{cara2} follows directly from \eqref{caraal}. Thus, $u$ is foliated by capillary curves.
\end{proof}

Note that we have proved that the level $\mathfrak{m}=0$ of the hierarchy is composed by the most basic solutions to $\Delta u+2=0$,  namely, quadratic polynomials \eqref{cua0} and radial solutions \eqref{ra0}, up to an additive linear term.

\section{Final discussion and open problems}\label{sec:problems}
The limits of Serrin's classical theorem have been tested along the years through the construction of non-trivial \emph{exceptional domains} $\Omega\subset \R^n$ where $\Delta u +f(u)=0$ can be solved with constant (or locally constant) Dirichlet and Neumann values along $\parc \Omega$. Many of these domains are constructed by bifurcation, see e.g. \cite{ABM,CF,DZ,EFRS,FMW1,FMW2,KS,RRS2,Ru,RSW,Sic,SS,Wh}. It is then a fundamental problem of global bifurcation theory to understand the moduli space of such domains from a global perspective. Theorems \ref{th:main} and \ref{th:bandas} in this paper can provide a blueprint for such descriptions. 


Let us discuss a specific example. In \cite{EFRS}, Enciso, Fernández, Ruiz and Sicbaldi constructed, via bifurcation, $1$-parameter families of ring-type domains $\Omega\subset \R^2$ where the Neumann solution to $\Delta u +\landa u=0$ has locally constant Dirichlet data. This result gives a negative solution to a generalization of Schiffer's famous conjecture, see Yau \cite[Problem 80]{Y}. In the view of Theorem \ref{th:main} one is led to expect that the bifurcation branches of such annular domains should persist for large values of the parameter, and end in the form of a dihedral necklace of disks. 

Similarly, a global structure result in the spirit of our Theorem \ref{th:bandas} can be conjectured for the non-trivial exceptional periodic domains in $\R^n$ constructed by Sicbaldi \cite{Sic} and Schlenk-Sicbaldi \cite{SS} for $\Delta u+\landa u=0$, and by Fall, Minlend and Weth \cite{FMW1} for $\Delta u+2 =0$. Finding such a global moduli space structure in these theories seems quite challenging. Our results suggest the possibility that such bifurcation branches could have some adequate \emph{foliation structure} (as we have shown that happens for the examples in \cite{FMW1} when $n=2$). This foliation structure, should it exist, could provide a path to achieve the desired global description.


On the other hand, the torsion equation $\Delta u+2=0$ in $\R^2$ appears naturally associated to other types of overdetermined boundary conditions in some important theories in fluid dynamics. One such theory is the study of water waves of constant vorticity. A very nice account on the current state of this theory and its connection with the torsion equation can be found in \cite{DDMW}. Even though the overdetermined boundary conditions for water waves of constant vorticity are not of Serrin type, many problems in this theory are similar in spirit to the ones treated here. It is a very intriguing question to determine if our theory based on capillary curves can be applied to this classical important context. For instance, it is natural to ask, in view of Theorem \ref{th:bandas}, whether the overhanging solitary water waves discovered by Dávila, Del Pino, Musso and Wheeler \cite{DDMW} can be constructed via elliptic functions through some adequate foliation structure, or if they can be characterized algebraically in terms some KdV-type hierarchy.

Back to the Serrin problem for $\Delta u+2=0$ in unbounded domains $\Omega\subset \R^2$, we expect:
\begin{conjecture}\label{conj:bandas}
Any Serrin band in $\R^2$ is, up to similarity, one of the periodic bands $\Omega_\tau$, $\tau\in (0,1]$, of Theorem \ref{th:bandas}.
\end{conjecture}
This conjecture can be seen as an overdetermined version of the uniqueness of Riemann's minimal examples in $\R^3$ by Meeks, Pérez and Ros \cite{MPR}; see Sections \ref{intro:bands} and \ref{intro:space} regarding this analogy. Conjecture \ref{conj:bandas} is also the natural overdetermined version of an important uniqueness theorem by Korevaar, Kusner and Solomon \cite{KKS}, which states that \emph{any properly embedded CMC surface in $\R^3$ of finite topology and exactly two ends must be a Delaunay surface}. We note that, for the Serrin problem, it follows by the Ros-Sicbaldi theory in \cite{RS} (see e.g. Theorem \ref{rosic}) that any Serrin domain in $\R^2$ of finite connectivity and exactly two unbounded boundary components must be a Serrin band.

The results by Korevaar, Kusner and Solomon \cite{KKS} also show that, if $\Sigma$ is a properly embedded CMC surface in $\R^3$ with finite topology, then each \emph{end} of $\Sigma$ is asymptotic to a Delaunay example. In the context of Serrin's problem, again by Ros-Sicbaldi's theory \cite{RS}, any unbounded Serrin domain $\Omega\subset \R^2$ of finite connectivity has a finite number of \emph{ends}, than can be thought of as \emph{Serrin half-bands} diffeomorphic to $[0,1]\times [0,\8)$. Our Serrin bands $\Omega_\tau$ in Theorem \ref{th:bandas} provide then a $1$-parameter family of different possible asymptotic behaviors of such ends. In the view of \cite{KKS}, the following question is very interesting: \emph{if $\Omega\subset \R^2$ is a Serrin domain of finite connectivity, is every end of $\Omega$ asymptotic to one of the Serrin bands $\Omega_\tau$ of Theorem \ref{th:bandas}?}

For the case of (bounded) Serrin ring domains is it natural to ask, in analogy 
with Conjecture \ref{conj:bandas}: \emph{let $\Omega\subset \R^2$ be a Serrin ring domain with boundary values $u=0$ along $\parc \Omega$. Must then $\Omega$ be, up to similarity, one of the ring domains $\Omega_{(s,\tau)}$ of Theorem \ref{th:main}?} We recall here that the annuli constructed by Agostiniani, Borghini and Mazzieri in \cite{ABM} satisfy these hypotheses. The question can be conceived then as whether all Serrin ring domains with zero Dirichlet data belong to the bifurcation branch of the ring domains of \cite{ABM}, or equivalently, to the bifurcation curve $\Upsilon_2 \subset \cW_0$ constructed at the end of Section \ref{sec:mainth}.


The above uniqueness problem for ring domains can be viewed as the Serrin problem version of the classical Lawson conjecture \cite{La}, famously proved by Brendle \cite{B}, according to which the Clifford torus is the only embedded minimal torus in the $3$-sphere. 

\appendix 
\section{Appendix: Developing maps via elliptic functions}

In this section we give an explicit formula for the two-parameter family of developing maps $g(z)$ constructed in Section \ref{sec:gaussmaps}. This formula is given in terms of the Weierstrass elliptic $\wp$-function defined on an adequate rectangular lattice whose associated half-periods $(\omega_1,\omega_2)$ are functions of the parameters $(\eta,\tau)$ associated to $g(z)$. For some properties of $\wp(z)$ that will be used here, we refer to the books \cite{Ch,WW}.

Recall our usual notation $Z(y):=e^{-\omega(0,y)}$. Let $\phi(z)$ be the holomorphic function defined by $\phi (iy)= Z(y)$. Note that $\phi(iy)>0$ for all $y$, and $\phi(0)=2/\eta$. Define now 
\begin{equation}\label{defiha}
\hat\phi(z):= c \, \phi(z/c), \hspace{0.5cm} c^3 := -\frac{\eta \tau^2}{8}<0.
\end{equation} 
This time $\hat\phi(iy)<0$ for every $y$, with $\hat\phi(0)=2c/\eta$. Due to the equation $4Z_y^2=p(Z)$, we have $$4\phi'(z)^2 = -p(\phi(z)),$$ and so the function $\hat\phi(z)$ satisfies 
\begin{equation}\label{odefi}
(\hat\phi'(z))^2 =(\phi'(z/c))^2= -\frac{1}{4} p(\phi(z/c))=-\frac{1}{4} p(\hat\phi(z)/c).
\end{equation} 
Using the expression for $c<0$ in \eqref{defiha}, we have
 $$(\hat\phi')^2 = -\hat{\phi}^3 + \mathfrak{p} \hat\phi^2 + \mathfrak{q} \hat\phi +1,$$ for adequate real constants $\mathfrak{p},\mathfrak{q}$ in terms of $\eta,\tau$: $$\mathfrak{p}=\frac{\left(\eta ^2-1\right) \tau ^2-1}{\left(\eta  \tau ^2\right)^{2/3}}, \hspace{0.5cm} \mathfrak{q}:= \frac{\tau ^2 \left(\eta ^2 \left(\tau ^2+1\right)-1\right)}{\left(\eta  \tau ^2\right)^{4/3}}.$$ We next define $$b:=\frac{\mathfrak{p}}{3}, \hspace{0.5cm} \hat{g}_2 :=\frac{\mathfrak{p}^2+3\mathfrak{q}}{12}, \hspace{0.5cm} \hat{g}_3 := -\frac{1}{432}(2 \mathfrak{p}^3 + 9 \mathfrak{p}\mathfrak{q} +27).$$
 Let $\hat\wp(z)$ be the Weierstrass $P$-function (with pole at the origin) that satisfies 
 \begin{equation}\label{odep1}
 (\hat\wp')^2 =4 \hat\wp^3-\hat{g}_2 \hat\wp -\hat{g}_3
 \end{equation} 
 for the previous choices of $\hat{g}_2,\hat{g}_3$. The modular discriminant $\Delta_{\rm mod}$ of $\hat\wp$ is given by 
 $$\Delta_{\rm mod} := \hat{g}_2^3-27\hat{g}_3^2 =\frac{(1+\eta^2)^2(1-\tau^2)^2(1+\tau^2 \eta^2)^2}{256 \eta^4 \tau^4}\geq 0.$$ Hence, if $\tau\in (0,1)$, we have $\Delta_{\rm mod}>0$, and this implies that the fundamental lattice of $\hat\wp(z)$ is rectangular. In addition, one can easily check that the function $$P(z):=\frac{b-\hat\phi(z)}{4}$$ satisfies the differential equation \eqref{odep1} of $\hat\wp(z)$, due to \eqref{odefi}. Since the only solutions to \eqref{odep1} are those of the form $\hat\wp(z+\landa)$ for $\landa \in \C$, we have $$\hat\phi(z)=b-4 \hat\wp(z+\landa),\hspace{0.5cm} \landa \in \C.$$ 
 Let $\hat\omega_1>0, \hat\omega_2\in i\R_+$ be the half-periods of $\hat\wp(z)$, so that $\hat\wp(z)$ is $2\hat\omega_1$-periodic and $2\hat\omega_2$-periodic. The function $\hat\wp(z)$ only takes real values along lines of the form $x=k\hat\omega_1$ or $iy= k\hat\omega_2$, for $k\in \Z$ and $z=x+iy$. In particular, $\hat\wp(z)$ takes real values at $z=\hat\omega_1$ and $z=\hat\omega_1+\hat\omega_2$, and it is well known that, in these conditions, 
 \begin{equation}\label{todohats}
 \hat\wp(\hat\omega_1)>\hat\wp(\hat\omega_1+\hat\omega_2).
 \end{equation}
These points represent the maximum and minimum values of $y \mapsto \hat\wp(\hat\omega_1+iy)$.

Recall next that $\hat\phi(iy)\in \R$ for every $y\in \R$. Since this property only holds for $\hat\wp$ along vertical lines of the form $\hat\omega_1+ iy$ (and its horizontal translations by $2\hat\omega_1$), we conclude that $\landa=\hat\omega_1+i \vartheta$ for some $\vartheta\in \R$. Taking into account now that $\hat\phi(iy)$ has a minimum at $y=0$, we deduce from \eqref{todohats} and the previous discussion that $$\hat\phi(z)= b-4\hat\wp(z+\hat\omega_1+\hat\omega_2).$$ Therefore, 
\begin{equation}\label{ecuap1}
\phi(z)= \frac{1}{c}\left(b-\frac{4}{c^2} \wp (z+\omega_1+\omega_2))\right),
\end{equation}
where $\wp$ is the Weierstrass function associated with $$g_2:= c^4 \hat{g}_2, \hspace{0.5cm} g_3:= c^6 \hat{g}_3.$$ This function is related to $\hat\wp$ by $\wp(z)= c^2 \hat\wp (cz),$ with half-periods $\omega_1 =\hat\omega_1/c$, $\omega_2 =\hat\omega_2/c$.

We prove next that
\begin{equation}\label{ecuap2}
g'(z)= \frac{1}{\phi(z)}\, \exp \left(-\int_0^z \frac{1}{\phi(w)} \, dw \right).
\end{equation} 
First we note that \eqref{ecuap2} holds at $z=0$, since $g'(0)=2/\eta$. Thus, it suffices to prove that the derivative of \eqref{ecuap2} holds along the $x=0$ axis, i.e., that for every $z=iy$ we have
\begin{equation}\label{ecuap3}
\frac{g''(z)}{g'(z)} = -\frac{1+\phi'(z)}{\phi(z)} .
\end{equation} 
To do so, we first see that, since $e^{2\omega}=|g'(z)|^2$, we have $$\frac{g''(iy)}{g'(iy)} = 2\omega_z (0,y) = -e^{\omega(0,y)} -i \omega_y (0,y),$$ where we have used \eqref{rox} together with the initial conditions $\alfa(0)=0$ and $\beta(0)=2$. Using now that $\phi(iy)=e^{-\omega(0,y)}$, we directly obtain the validity of \eqref{ecuap3}, and hence of \eqref{ecuap2}. Now, by \eqref{ecuap1} and \eqref{ecuap2} we have $$g'(z)= \frac{c^3}{b c^2 -4\wp (z+\omega_1+\omega_2)}\, \exp \left(-c^3 \int_0^z \frac{1}{b c^2 -4\wp (w+\omega_1+\omega_2)} \, dw \right).$$  Since $g'(x)\in \R$ if $x\in \R$ and $g(0)\in \R$, we deduce that $g(x)\in \R$. Also,
\begin{equation}\label{fgp}
\frac{g'(z)}{g(z)} = \frac{-c^3}{b c^2 -4\wp (z+\omega_1+\omega_2)} \hspace{0.3cm} \left(= \frac{-1}{\phi(z)} \right).
\end{equation} 
We point out that, in \eqref{fgp},
\begin{equation}\label{cteswei}
b c^2 = \frac{1}{12}(\eta^2 \tau^2-\tau^2-1) = \frac{\delta}{3}, \hspace{0.5cm} c^3 := -\frac{\eta \tau^2}{8}<0.
\end{equation}



Let us find a better expression for $g(z)$. The computations are similar to \cite[Lemma 5.4]{FHM}, so we will omit some details. Let $\zeta(z),\sigma(z)$ denote the classical Weierstrass zeta and sigma functions associated with $\wp(z)$. They satisfy $\zeta'(z)=\wp(z)$ and $\sigma'(z)/\sigma(z)= \zeta(z)$. In addition, let $\mu$ be the unique point of the rectangular lattice spanned by the half-periods $\{\omega_1,\omega_2\}$ where $\wp(z)= bc^2/4$. The values of $\wp(z)$ at the half-periods are $\wp(\omega_1)=e_1$, $\wp(\omega_2)=e_2$ and $\wp(\omega_1+\omega_2)=e_3$, where $e_1<e_3<e_2$ are given by the relation $$\wp'(z)^2 = 4(\wp(z)-e_1)(\wp(z)-e_2)(\wp(z)-e_3).$$ One can check that $bc^2/4\in (e_3,e_2)$, and so ${\rm Im}(\mu)=\omega_2$, i.e., $\mu$ lies in the segment between $\omega_2$ and $\omega_1+\omega_2$. At this point, from the previous differential equation for $\wp(z)$ and the fact that $\wp'(\mu)<0$, one has $\wp'(\mu)=c^3/4$.

Consider next, in terms of these quantities, the holomorphic function given by $$G(z)=g_0 \exp(-2\zeta(\mu)(z+\omega_1+\omega_2))\left[\frac{\sigma(\mu+(z+\omega_1+\omega_2))}{\sigma(\mu-(z+\omega_1+\omega_2))}\right],$$ where $g_0\in \C$ is a constant defined so that $G(0)=-1$.
Then, 
$$\frac{G'(z)}{G(z)} = \frac{\wp'(\mu)}{\wp(\mu)-\wp(z+\omega_1+\omega_2)} = \frac{c^3}{b c^2 -4\wp (z+\omega_1+\omega_2)}.$$ Comparing this equation with \eqref{fgp}, we obtain $g(z)=1/G(z).$ So,
\begin{equation}\label{ecuapen4}
g(z)=g_0 \exp(2\zeta(\mu)(z+\omega_1+\omega_2))\left[\frac{\sigma(\mu-(z+\omega_1+\omega_2))}{\sigma(\mu+(z+\omega_1+\omega_2))}\right].
\end{equation}

Finally, we remark some facts that follow quite directly from the previous discussion, and that are used at some point throughout the text.



\begin{enumerate}
\item
From \eqref{fgp} and the definition of $\mu$ above, we see that $\phi(z)$ has a pole at $z=\mu-(\omega_1+\omega_2)$. This is a negative real number that we call $x_\mu$, and $\phi(x+iy)$ is defined for any $x\in (x_\mu,0]$. Thus, $\omega(x,y)$ is well defined in $(x_\mu,0]\times \R$ and cannot be extended to $x=x_\mu$. This means that $x_\mu=x_b^-$, see Lemma \ref{lem:lio}. In particular, by \eqref{fgp}, $g$ has a pole at $z=x_b^-$. Since $g(i\R)\subset \S^1$, we see by inversion that $g(-x_b^-)=0$.
\item
By a similar argument, we obtain $x_a^+=\omega_1$. Moreover, $g(x)\in \R$ for any $x\in (x_b^-,x_a^+)$, with $g'(x)>0$, since $g'(0)=2/\eta>0$ and $g'(z)$ never vanishes in $(x_b^-,x_a^+)\times \R$. In particular, note that $(x_b^-,x_a^+)\subset (-\omega_1,\omega_1)$.
\item
Also, by comparing the $y$-periodicity of $\omega$ and \eqref{fgp}, we obtain $i\vart=\omega_2$, where $\vart>0$ was defined in Lemma \ref{lem:sigma}.
\end{enumerate}

%
%



\def\refname{References}

\vskip 0.2cm

\noindent Alberto Cerezo

\noindent Departamento de Geometría y Topología \\ Universidad de Granada (Spain) \\ Departamento de Matemática Aplicada I \\ Universidad de Sevilla (Spain)

\noindent  e-mail: {\tt cerezocid@ugr.es}

\vskip 0.2cm

\noindent Isabel Fernández

\noindent Departamento de Matemática Aplicada I,\\ Instituto de Matemáticas IMUS \\ Universidad de Sevilla (Spain).

\noindent  e-mail: {\tt isafer@us.es}

\vskip 0.2cm

\noindent Pablo Mira

\noindent Departamento de Matemática Aplicada y Estadística,\\ Universidad Politécnica de Cartagena (Spain).

\noindent  e-mail: {\tt pablo.mira@upct.es}

\vskip 0.4cm

\noindent 
This research has been financially supported by Projects PID2020-118137GB-I00 and PID2024-160586NB-I00 funded by MCIN/AEI/10.13039/501100011033, and CARM, Programa Regional de Fomento de la Investigación, Fundación Séneca-Agencia de Ciencia y Tecnología Región de Murcia, reference 21937/PI/22


\begin{thebibliography}{9}

{\small

\bibitem{Ab} U. Abresch, Constant mean curvature tori in terms of elliptic functions, {\it J. Reine Angew. Math.} {\bf 394} (1987), 169--192.

\bibitem{AGGH} V. Adler, B. Gurel, M. Gurses, I. Habibullin, Boundary conditions for integrable equations, {\it J. Phys. A: Math. Gen.} {\bf 30} (1997), 3505--3513.

\bibitem{ABM} V. Agostiniani, S. Borghini, L. Mazzieri, On the Serrin problem for ring-shaped domains, {\it J. Eur. Math. Soc.} {\bf 27} (2024), 2705--2749.

\bibitem{A1} A.D. Alexandrov, Uniqueness theorems for surfaces
in the large, I, {\it Vestnik Leningrad Univ.} {\bf 11} (1956),
5--17. (English translation: {\it Amer. Math. Soc. Transl.}  {\bf 21} (1962), 341--354).



%

\bibitem{B} S. Brendle, Embedded minimal tori in $\S^3$ and the Lawson conjecture, {\it Acta Math.} {\bf 211} (2013), 177--190.

\bibitem{Bo} S. Borghini, Symmetry results for Serrin-type problems in doubly connected domains, {\it Math. Eng.} {\bf 5} (2023), 1--16.

\bibitem{Ce} A. Cerezo, Free boundary minimal annuli in geodesic balls of $\H^3$, arXiv:2502.20303.

\bibitem{CFM} A. Cerezo, I. Fernández, P. Mira, Annular solutions to the partitioning problem in a ball, {\it J. Reine Angew. Math.} {\bf 828} (2025), 57--81.

\bibitem{CFM2} A. Cerezo, I. Fernández, P. Mira, Free boundary CMC annuli in spherical and hyperbolic balls, {\it Calc. Var.} {\bf 64}, 37 (2025).


\bibitem{Bob} A. Bobenko, All constant mean curvature tori in $\R^3$, $\S^3$, $\H^3$ in terms of theta-functions, {\it Math. Ann.} {\bf 290} (1991), 209--245.

%



\bibitem{CF} G. Cao-Labora, A.J. Fernández, A contractible Schiffer counterexample on the half-sphere, arXiv:2510.05732.

\bibitem{Ch} K. Chandrasekharan, Elliptic Functions, Grundlehren der mathematischen Wissenschaften, {\bf 281}, Springer Verlag, 1985.




\bibitem{DZ} G. Dai, Y. Zhang. Sign-changing solution for an overdetermined elliptic problem
on unbounded domain. {\it J. Reine Angew. Math.} {\bf 803} (2023), 267--293.

\bibitem{DDMW} J. Dávila, M. Del Pino, M. Musso, M.H. Wheeler, Overhanging solitary water waves, arXiv:2409.01182

\bibitem{DPW} M. Del Pino, F. Pacard, J. Wei. Serrin's overdetermined problem and constant mean curvature
surfaces, {\it Duke Math. J.} {\bf 164} (2015), 2643--2722.

%

\bibitem{EFRS} A. Enciso, A.J. Fernández, D. Ruiz. P. Sicbaldi, A Schiffer-type problem for annuli with applications to stationary planar Euler flows, {\it Duke Math. J.} {\bf 174}, No. 6, 2025.

\bibitem{FMW1} M. M. Fall, I. A. Minlend T. Weth, Unbounded periodic solutions to Serrin's overdetermined boundary value problem, {\it Arch. Rational Mech. Anal.} {\bf 223} (2017), 737--759.


\bibitem{FMW2} M. M. Fall, I. A. Minlend T. Weth, The Schiffer problem on the cylinder and on the $2$-sphere, {\it J. Eur. Math. Soc.}, to appear (2025).


\bibitem{FHM} I. Fernández, L. Hauswirth, P. Mira, Free boundary minimal annuli immersed in the unit ball, {\it Arch. Rational Mech. Anal.} {\bf 247} (2023), 108.

\bibitem{FL} A. Fraser, M. Li, Compactness of the space of embedded minimal surfaces with free boundary in three-manifolds with nonnegative Ricci curvature and convex boundary, {\it J. Differential Geom.} {\bf 96} (2014), 183--200.

\bibitem{GW} F. Gesztesy, R. Weikard, Lamé potentials and the stationary (m)KdV hierarchy, {\it Math. Nachr.} {\bf 176} (1995), 73--91.

\bibitem{HHP} F. Helein, L. Hauswirth, F. Pacard, A note on some overdetermined problems, {\it Pacific  J. Math.} {\bf
250}
(2011), 319--334.



%

%
%


%
%


\bibitem{Hi} N. Hitchin, Harmonic maps from a $2$-torus to the $3$-sphere, {\it J. Diff. Geom.} {\bf 31} (1990), 627--710.


%
%
%

\bibitem{KS} N. Kamburov, L. Sciaraffia. Nontrivial solutions to Serrin's problem in annular domains, {\it Ann. Inst. H. Poincaré C}, 2020.
%
%
%



\bibitem{KDV} D.J. Korteweg, G. de Vries, On the change of form of long waves advancing in a
rectangular canal, and a new type of long stationary waves. {\it Phil. Mag.} {\bf 39} (1895) 422--443.

\bibitem{Ho} H. Hopf, H. Uber Flachen mit einer Relation zwischen den Hauptkrummungen, {\it Math. Nachr.} {\bf 4} (1951), 232--249.

\bibitem{IM} A.R. Its, V.B. Matveev, Schrodinger operators with the finite-band spectrum and the
N-soliton solutions of the Korteweg-de Vries equation, {\it Theoret. and Math. Phys.} {\bf 23} (1975), 343--355.

\bibitem{Jos} N. Joshi, The second Painleve hierarchy and the stationary KdV hierarchy, {\it Publ. Res. Inst. Math. Sci.} {\bf 40} (2004), 1039--1061.

\bibitem{Kha} I.T. Khabibullin, Sine-Gordon equation on the semi-axis. {\it Theor. Math. Phys.} {\bf 114} (1998), 90--98.

\bibitem{KSV} D. Khavinson, A.Y. Solynin, D. Vassilev, Overdetermined boundary value problems, quadrature domains and applications, {\it Comput. Methods Funct. Theory} {\bf 5} (2005), 19--48.

\bibitem{KSm} M. Kilian, G. Smith, On the elliptic sinh-Gordon equation with integrable boundary conditions, {\it Nonlinearity}, {\bf 34} (2021), 5119.

\bibitem{KKS} N.J. Korevaar, R. Kusner, B. Solomon. The structure of complete embedded surfaces with constant mean curvature. {\it J. Diff. Geom.} {\bf 30} (1989), 465--503.

\bibitem{La} H.B. Lawson, Complete minimal surfaces in $\S^3$, {\it Ann. Math.} {\bf 92} (1970), 335--374.



%
%

\bibitem{LWW} Y. Liu, K. Wang, J. Wei, On Smooth Solutions to One Phase-Free Boundary Problem in $\R^n$, {\it Int. Math. Res. Notices}, {\bf 2021} (2021), 15682--15732.

\bibitem{Me} W. H. Meeks, The topology and geometry of embedded surfaces of constant mean curvature, {\it J. Diff. Geom.} {\bf 27} (1988), 539--552.

\bibitem{MPR} W.H. Meeks, J. Pérez, A. Ros, Properly embedded minimal planar domains, {\it Ann. Math.} {\bf 181} (2015),  473--546.

\bibitem{NT} C. Nitsch, C.Trombetti, The classical overdetermined Serrin problem, {\it Complex Var. Elliptic Equ.} {\bf 63} (2018), 1107--1122.
 
\bibitem{Nit} J.C.C. Nitsche, Stationary partitioning of convex bodies, \emph{Arch. Rational Mech. Anal.} {\bf 89} (1985), 1--19.

\bibitem{PP} L.E. Payne, G.A. Philippin, On two free boundary problems in potential theory. {\it J. Math. Anal. Appl.} {\bf 161} (1991), 332--342.

\bibitem{PS} U. Pinkall, I. Sterling, On the classification of constant mean curvature tori, {\it Ann. Math.} {\bf 130} (1989), 407--451.

\bibitem{Rei} W. Reichel. Radial symmetry by moving planes for semilinear elliptic BVPs on annuli and other non-convex domains. In Elliptic and parabolic problems (Pont-a-Mousson, 1994), volume 325 of Pitman Res. Notes Math. Ser., pages 164--182. Longman Sci. Tech., Harlow, 1995.

\bibitem{RRS1} A. Ros, D. Ruiz, P. Sicbaldi, A rigidity result for overdetermined elliptic problems in the plane, {\it Comm. Pure Appl. Math.}, {\bf 70} (2017), 1223--1252. 

\bibitem{RRS2} A. Ros, D. Ruiz, P. Sicbaldi, Solutions to overdetermined elliptic problems in nontrivial exterior domains, {\it J. Eur. Math. Soc.} {\bf 22} (2020), 253--281.

\bibitem{RS} A. Ros, P. Sicbaldi, Geometry and topology of some overdetermined elliptic problems, {\it J. Differential Equations}, {\bf 255} (2013), 951-977.

\bibitem{Ru} D. Ruiz, Nonsymmetric sign-changing solutions to overdetermined elliptic problems in bounded domains, {\it J. Eur. Math. Soc.}, to appear.

\bibitem{RSW} D. Ruiz, P. Sicbaldi, J. Wu, Overdetermined elliptic problems in onduloid-type
domains with general nonlinearities. {\it J. Funct. Anal.} {\bf 283} (2022), Paper No. 109705,
26.

\bibitem{Se} J. Serrin, A symmetry problem in potential theory, {\it Arch. Rational Mech. Anal.} {\bf 43} (1971), 304--318.

\bibitem{Sic} P. Sicbaldi, New extremal domains for the first eigenvalue of the Laplacian in flat tori. {\it Calc. Var.} {\bf 37} (2010), 329--344.

\bibitem{Sic2} P. Sicbaldi, A short survey on overdetermined elliptic problems in unbounded domains. Current trends in analysis, its applications and computation, 451--461, Trends Math. Res. Perspect., Birkhauser/Springer,
Cham, 2022.

\bibitem{SS} F. Schlenk, P. Sicbaldi, Bifurcating extremal domains for the first eigenvalue of the Laplacian, {\it  Adv. Math.} {\bf 229} (2012), 602--632.

\bibitem{Sk} E.K. Sklyanin, Boundary conditions for integrable equations, {\it Funct. Anal. Appl.} {\bf 21} (1987), 164--166.

\bibitem{Sir} B. Sirakov, Symmetry for exterior elliptic problems and two conjectures in potential theory,
{\it Ann. Inst. H. Poincaré Anal. Non Linéaire} {\ bf 18} (2001), 135--156. 

\bibitem{Sir2} B. Sirakov, Overdetermined Elliptic Problems in Physics. In: Berestycki, H., Pomeau, Y. (eds) Nonlinear PDE's in Condensed Matter and Reactive Flows. NATO Science Series, vol 569. Springer, Dordrecht.

%

%
%
%
%
%

\bibitem{Sh} M. Shiffman, On Surfaces of Stationary Area Bounded by Two Circles, or Convex Curves, in Parallel Planes, {\it Ann. Math.} {\bf 63} (1956), 77--90.

\bibitem{T} M. Traizet. Classification of the solutions to an overdetermined elliptic problem in the plane,
{\it Geom. Funct. Anal.} {\bf 24} (2014), 690--720.



%
%
%

\bibitem{Wn} H.F. Weinberger, Remark on the preceeding paper of Serrin, {\it Arch. Rational Mech. Anal.} {\bf 43} (1971), 319--320.


\bibitem{W0} H.C. Wente, Counterexample to a conjecture of H. Hopf, {\it Pacific J. Math.} {\bf 121} (1986), 193--243.

\bibitem{W} H.C. Wente, Constant mean curvature immersions of Enneper type, {\it Mem. Amer. Math. Soc.} {\bf 478} (1992).

\bibitem{Wh} M.H. Wheeler, Non-symmetric solutions to an overdetermined problem for the Helmholtz equation in the plane, arXiv:2509.00455.
%
%

\bibitem{WGS} N.B. Willms, G.M.L. Gladwell, D. Siegel, Symmetry theorems for some overdetermined boundary value problems on ring domains, {\it Z. Angew. Math. Phys.} {\bf 45} (1994).
556--579.

\bibitem{Y} S.T. Yau, Problem section, in \emph{Seminar on differential geometry}, Ann.
of Math. Stud. {\bf 102}, Princeton Univ. Press, Princeton, 1982, 669--706.

\bibitem{W2} H.C. Wente, Tubular capillary surfaces in a convex body. Advances in geometric analysis and continuum mechanics, Edited by P. Concus and K. Lancaster, International Press 1995, 288--298.

\bibitem{WW} E.T. Whittaker, G.N. Watson, A course of modern analysis, Cambridge University Press.}

\end{thebibliography}
\end{document}